\documentclass[11pt, reqno]{amsart}
\usepackage{amsmath,amsthm,amsfonts,amssymb,amscd}
\usepackage{epsfig}
\input{epsf.tex}

\newtheorem{theorem}{Theorem}[section]
\newtheorem{proposition}[theorem]{Proposition}
\newtheorem{lemma}[theorem]{Lemma}
\newtheorem{corollary}[theorem]{Corollary}
\newtheorem{define}[theorem]{Definition}

\def\Empty{}

\catcode`\@=11

\def\section{\@startsection {section}{1}{\z@}{-3.5ex plus -1ex minus 
-.2ex}{2.3ex plus .2ex}{\large\bf}}

\def\fnum@figure{{\small Figure \thefigure}}
\def\fakefigure{\def\@captype{figure}}

\long\def\@makecaption#1#2{
    \vskip 10pt 
    \def\FCap{#2} \def\NoCap{\ignorespaces}
    \ifx \FCap\NoCap
       \setbox\@tempboxa\hbox{#1}  
      \else
       \setbox\@tempboxa\hbox{#1: \small \it #2}
    \fi
    \ifdim \wd\@tempboxa >\hsize   
        \unhbox\@tempboxa\par      
      \else                        
        \hbox to\hsize{\hfil\box\@tempboxa\hfil}  
    \fi}

\def\@oddhead{\hbox{}\rightmark \hfil \rm\thepage}
\def\sectionmark#1{\markright {\sc{\ifnum \c@secnumdepth >\z@
      \S\thesection.\hskip 1em\relax \fi #1}}}

\catcode`\@=12

\def\oplabel#1{
  \def\OpArg{#1} \ifx \OpArg\Empty {} \else
  	\label{#1}
  \fi}

\newlength{\saveu}

\catcode`\@=11

\newcommand{\ca}{\mbox{${\mathcal A}$}}

\newcommand{\aaa}{\mbox{${\mathcal A}$}}

\newcommand{\hhu}{\mbox{${\mathcal H} ^u$}}
\newcommand{\hhs}{\mbox{${\mathcal H} ^s$}}

\newcommand{\hp}{\mbox{${\mathcal H}$}}

\newcommand{\cc}{\mbox{$\mathcal C$}}

\newcommand{\oo}{\mbox{$\mathcal O$}}
\newcommand{\oos}{\mbox{${\mathcal O}^s$}}
\newcommand{\oou}{\mbox{${\mathcal O}^u$}}

\newcommand{\rrrr}{\mbox{${\bf R}$}}

\newcommand{\mi}{\mbox{$\widetilde M$}}

\newcommand{\ls}{\mbox{$\Lambda^s$}}
\newcommand{\lu}{\mbox{$\Lambda^u$}}

\newcommand{\wls}{\mbox{$\widetilde \Lambda^s$}}
\newcommand{\wlu}{\mbox{$\widetilde \Lambda^u$}}

\newcommand{\wwp}{\mbox{$\widetilde \Phi$}}

\newcommand{\fol}{\mbox{$\mathcal F$}}

\newcommand{\fn}{\mbox{$\widetilde {\mathcal F}$}}

\newcommand{\ws}{\mbox{$\widetilde  W^s$}}
\newcommand{\wu}{\mbox{$\widetilde  W^u$}}

\catcode`\@=12

\def\centeredepsfbox#1{\centerline{\epsfbox{#1}}}

\begin{document}

\title{Pseudo-Anosov flows in toroidal $3$-manifolds}

\author{Thierry Barbot}

\author{S\'{e}rgio R. Fenley}
\thanks{Reseach partially supported by NSF grant DMS-0305313.}

\address{Thierry Barbot\\ Universit\'{e} d'Avignon et des pays de Vaucluse\\
LANLG, Facult\'{e} des Sciences\\
33 rue Louis Pasteur\\
84000 Avignon, France.}

\email{thierry.barbot@univ-avignon.fr}

\address{S\'ergio Fenley\\Florida State University\\
Tallahassee\\FL 32306-4510, USA \ \ and \ \ 
Princeton University\\Princeton\\NJ 08544-1000, USA}

\email{fenley@math.princeton.edu}

 \email{}
\maketitle

{\small 
\noindent
{{\underline {Abstract}} $-$
We first prove rigidity results for pseudo-Anosov flows in
prototypes of toroidal
$3$-manifolds:
we show that a pseudo-Anosov flow in a Seifert 
fibered manifold is up to finite covers topologically
equivalent to a geodesic flow and we show
that a pseudo-Anosov flow in a solv manifold
is topologically equivalent to a suspension Anosov flow.
Then we study the interaction of a general pseudo-Anosov flow with
possible Seifert fibered pieces in the torus decomposition:
if the fiber is associated with a periodic orbit of the
flow, we show that there is a standard and very simple 
form for the flow 
in the piece using
Birkhoff annuli. This form is strongly connected with
the topology of the Seifert piece.
We also construct a 
large new class of examples in many graph manifolds, which is
extremely general and flexible.
We construct other new classes of examples,
some of which are generalized pseudo-Anosov flows
which have one prong singularities and which show that the above 
results in
Seifert fibered and solvable manifolds do not apply to
one prong pseudo-Anosov flows.
Finally we also analyse immersed and embedded
incompressible tori in optimal position with respect 
to a pseudo-Anosov flow.
\footnote{AMS mathematics classification: Primary: 37D20, 37D50; Secondary: 57M60,
57R30}
}
}

\section{Introduction} 

The goal of this article is to start a systematic study of pseudo-Anosov flows
in toroidal $3$-manifolds.
More specifically we analyse such flows in closed manifolds which are
not hyperbolic or 
in pieces of the torus decomposition which are not hyperbolic and we 
obtain substantial results in Seifert fibered pieces.
We also produce many new examples of pseudo-Anosov flows,
including a large new class in graph manifolds
and 
we study optimal position of tori with respect to arbitrary pseudo-Anosov flows.

The study of hyperbolic flows in toroidal manifolds was initiated by
Ghys \cite{Gh}, who 
analysed Anosov flows in $3$-dimensional circle bundles. 
Ghys showed
that up to finite covers, the flow is topologically 
equivalent to the geodesic flow in the unit tangent
bundle of a hyperbolic surface. 
This was later strengthened by the first author who showed
that this holds if the manifold is Seifert fibered \cite{Ba1}.
In the mid 70's a generalization of Anosov flows 
called pseudo-Anosov flows
was introduced by Thurston \cite{Th2}. He showed that these are extremely
important for the study of surfaces and $3$-manifolds \cite{Th1,Th2,Th3}.
The difference from Anosov flows is that one allows finitely many
singularities which are each of $p$-prong type. In the applications
to the topology of $3$-manifolds there is a requirement that $p$ is at
least $3$, which is the convention here as well.
Pseudo-Anosov flows have been 
used very successfully to analyse the topology and geometry
of $3$-manifolds \cite{Mo1,Mo2,Mo3,Ga-Ka,Fe3,Fe7,Fe8} and they
are much more common than Anosov flows
\cite{Fr,RSS,Fe6}. They are much more flexible because for instance
they survive
most Dehn surgeries on closed orbits \cite{Fr}, see also 
section \ref{sec:examples} for new examples. In addition
as opposed to Anosov flows, pseudo-Anosov flows
can be constructed transverse to Reebless foliations
in vast generality if the manifold is atoroidal
\cite{Mo3,Fe4,Cal1,Cal2,Cal3}, yielding deep
geometric information.

In this article we analyse several aspects of pseudo-Anosov flows in toroidal 
manifolds. In the presence of a general pseudo-Anosov flow the manifold is
always irreducible \cite{Fe-Mo}. By the geometrization theorem \cite{Pe1,Pe2,Pe3}
a three manifold with a pseudo-Anosov flow is either hyperbolic,
Seifert fibered, a solv manifold
or the torus decomposition of the manifold is non trivial.

Notice that there is an ongoing broad study of pseudo-Anosov flows
in {\underline {closed}}, hyperbolic manifolds by the second author \cite{Fe3,Fe7,Fe8}, which 
is mostly orthogonal to this article.
In our situation classical $3$-dimensional topology will play a much 
bigger role.

A {\em {topological equivalence}} between two flows is a homeomorphism
which sends orbits to orbits.
We first analyse Seifert fibered manifolds.
Despite the much bigger flexibility of pseudo-Anosov flows we prove a 
 strong rigidity
theorem, extending the
result of \cite{Ba1} for Anosov flows (Theorem~\ref{Seifert}):

\vskip .1in
\noindent
{\bf {Theorem A}} $-$ Let $\Phi$ be a pseudo-Anosov in a Seifert fibered 
$3$-manifold. Then up to finite covers, $\Phi$ is topologically
equivalent to a geodesic flow in the unit tangent bundle of a
hyperbolic surface.
\vskip .1in

In particular the flow does not have singularities and is topologically
Anosov.
The proof of theorem A splits into two cases depending on whether
the fiber is 
homotopic to a closed orbit of the flow or not. 
In fact later on this dichotomy will be fundamental for 
the study of pseudo-Anosov flows restricted to an arbitrary Seifert fibered
piece of the torus decomposition of the manifold.
In the proof of theorem A 
we start by showing that the first case cannot happen.
In the other case we prove that there
are no singularities and also that the the stable/unstable foliations
are slitherings as introduced by Thurston \cite{Th4,Th5}.
This produces two actions of the fundamental group
on the circle, which are used to produce a 
$\pi_1$-invariant conjugacy of the
orbit space with the orbit space of the geodesic flow.
This is enough to prove theorem A.
Here orbit space refers to the orbit space of the flow lifted
to the universal cover. For a pseudo-Anosov flow, this orbit
space is always homeomorphic to the plane \cite{Fe-Mo} and 
hence the flow in the universal cover is topologically
a product.

Next we analyse pseudo-Anosov flows in three manifolds with
virtually solvable fundamental group. Here again there is a very strong
rigidity result (Theorem~\ref{solva}):

\vskip .1in
\noindent
{\bf {Theorem B}} $-$ Suppose that $\Phi$ is a pseudo-Anosov flow in 
a three manifold with virtually solvable fundamental group. Then $\Phi$ has
no singularities and is topologically equivalent to a suspension
Anosov flow.
\vskip .1in

The proof of theorem B is roughly as follows.
Suppose first that the fundamental group is solvable and consider 
a normal rank two abelian subgroup. 
The first case is that this subgroup 
acts non freely on the orbit space. 
In this case we show that the subgroup preserves a structure
in the universal cover
called a chain of lozenges (described below). By normality the whole fundamental
group of the manifold will preserve this chain of lozenges.
We also show that the stabilizer of a chain of lozenges
is at most a finite extension of ${\bf Z}^2$, which
leads to a contradiction.
It follows that the rank two abelian subgroup acts freely on the orbit 
space and by previous results this implies that the flow is
topologically equivalent to a suspension Anosov flow \cite{Fe5}.
If the manifold is virtually solvable then the flow is covered
by a suspension Anosov flow and one can show that the original
flow is also a suspension Anosov flow.

One difference between theorems A and B is that in theorem B the flow
is topologically equivalent to a suspension Anosov flow, whereas in
theorem A, we only prove it is equivalent to a geodesic flow up
to finite covers. 
The condition on finite covers cannot be removed from theorem  A
as can be seen by unwrapping
the fiber direction. See detailed explanation after the proof
of theorem A in section \ref{seifconj}.

The proofs of both theorems A and B use amongst other tools,
the study of actions on the leaf
spaces of the stable/unstable foliations in the universal
cover. These topological spaces already have a key role
in the context of Anosov flows \cite{Gh, Ba1,Fe1,Fe2}. In the more
general context of pseudo-Anosov flows, 
these leaf spaces are generalizations of both
trees and non Hausdorff simply connected one manifolds and
are called non Hausdorff trees \cite{Fe5}. A key fact used, generalizing
a previous result in the case of non 
Hausdorff simply connected one manifolds \cite{Ba5}, is that
a group element acting freely on the non Hausdorff tree
has an axis \cite{Fe5,Ro-St}.
Notice that for a pseudo-Anosov flow, the axis may not be properly embedded in the respective
leaf space.

This theme of analysising the structure of the flow in the universal cover is prevalent
in a lot of the study of pseudo-Anosov flows and is central to the results of this
article. This is used to give topological and homotopic information about the manifold, 
and it also aids questions of rigidity of the flows and large scale geometry
of the flow and the manifold.
This previous, extensive topological study of pseudo-Anosov flows substantially
simplifies the proofs of theorems A and B. 

\vskip .1in
Next we consider 
manifolds with non trivial torus decomposition.
The overarching goal is to understand the flow in each piece of the torus decomposition
and then analyse how the pieces are glued. 
In this article we do a substantial analysis of one type of
Seifert fibered pieces (the periodic type, see below) and we study 
the tori in the boundary
of the pieces of a torus decomposition.
One of our main goals is to produce 
a large new class of examples. These examples are much more
naturally understood after the structure of periodic Seifert pieces is
analysed and the structure of tori is better understood. 

In terms of the relation with pseudo-Anosov flows,
Seifert fibered pieces in the torus decomposition fall in two categories:
if the piece admits a Seifert fibration where the fiber is freely
homotopic to a closed orbit of the flow we say that the
piece is {\em periodic}, otherwise the piece is called a {\em free}
piece. Equivalently the Seifert piece is free if and only
if the action in the orbit space of a deck transformation corresponding
to a fiber in any possible Seifert fibration is free.
This dichotomy between free pieces and periodic pieces
is fundamental.
For example if the whole manifold is Seifert then 
one main step in the proof of theorem
A is to show that the piece is a free piece. 
For solvable manifolds, after cutting along a fiber,
the piece is also free. 
For Anosov flows, the case of free Seifert pieces
has been extensively analysed in \cite{Ba3}, 
giving a nearly final conclusion in the following 
case: ${\bf R}$-covered Anosov flows on
graph manifolds where all Seifert fibered pieces are free.
Recall that a {\em graph manifold} is an irreducible $3$-manifold
 where the pieces 
of the torus decomposition are all Seifert.
An Anosov flow is ${\bf R}$-covered if (say) its stable foliation is $\rrrr$-covered.
A foliation is $\rrrr$-covered if the  its lift to the universal cover
has leaf space homeomorphic
to the real numbers \cite{Fe1}.

To understand pseudo-Anosov flows in pieces of the torus decomposition one wants
to cut the manifold along tori and analyse the flow in each piece. Therefore one wants
the cutting torus to be in good position with respect to the flow.
The best situation for a general given torus
is that there is a torus isotopic to it which is transverse
to the flow. But this is not always possible. A good representative of a 
much more common situation is the following: consider the geodesic flow in the
unit tangent bundle of a closed hyperbolic surface (an Anosov flow). Let $\alpha$ be
a simple closed geodesic and let $T$ be the torus of unit vectors along $\alpha$. 
Then $T$ is embedded and incompressible but is not tranverse to the flow: it contains
two copies of $\alpha$ corresponding to the two directions along $\alpha$ and is 
otherwise transverse to the flow. This is the best position amongst all
tori isotopic to 
$T$.

Hence it is essential to understand  the interaction between
$\pi_1$-injective tori and pseudo-Anosov flows. 
Consider a ${\bf Z}^2$ subgroup of the fundamental group:
if it acts freely on the orbit space then the flow
is topologically equivalent to a suspension Anosov flow \cite{Fe5}.
Otherwise some element in ${\bf Z}^2$ does not act freely on the orbit
space and is associated to a closed orbit of the
flow. In the last case the ${\bf Z}^2$ 
describes a non trivial free homotopy from
a closed orbit to itself.
Any free homotopy between closed orbits can be put in
a canonical form as a union of immersed Birkhoff annuli
\cite{Ba2,Ba3}.
A {\em Birkhoff annulus} is an immersed annulus so that
each boundary component is a closed orbit of the flow and
the interior of the annulus is transverse to the flow.
A {\em Birkhoff torus} or {\em Birkhoff Klein bottle} is essentially 
a $\pi_1$-injective surface which is a 
union of Birkhoff annuli (see section \ref{sec:birk}).
Given an embedded incompressible torus $T$, one looks for 
an isotopic copy which is a Birkhoff torus.

A Birkhoff annulus lifts to a {\em lozenge} in the universal cover:
the boundaries lift to periodic orbits and the interior
lifts to a partial ideal
 quadrilateral region $D$ in the orbit space:
two opposite vertices of $D$ are lifts of the boundary orbits,
two vertices of $D$ are ideal and the stable/unstable foliations
in $D$ form a product structure.
The boundary orbits are the corners of the lozenge.
Lozenges are the building blocks in the universal
cover associated to free homotopies between closed orbits
and  they
are fundamental for much of the theory of Anosov flows (\cite{Ba2,Fe2})
and more generally, of pseudo-Anosov flows
\cite{Fe3,Fe5}.
Unless the flow is suspension Anosov, then any ${\bf Z}^2$ in
the fundamental group has associated to it an (essentially)  unique {\em chain} of
lozenges, where some elements of ${\bf Z}^2$ act fixing the corners and
some elements act freely. 
In the next two results one goal is to look for the best position of 
embedded incompressible tori.
In Proposition~\ref{toriem}, we prove
(see definition~\ref{def:string} for the notion of a string of lozenges):

\vskip .1in
\noindent
{\bf {Theorem C}} $-$ Let $T$ be a $\pi_1$-injective torus 
and let ${\mathcal C}$ be the $\pi_1(T)$ invariant chain of lozenges.
Suppose there is a corner $\alpha$ of ${\mathcal C}$ and a covering
translation $g$ with $g(\alpha)$ in the interior of a lozenge in ${\mathcal C}$.
Then ${\mathcal C}$ is a string of lozenges. In addition $T$ is homotopic
into a free Seifert fibered piece.
\vskip .1in

One relevance of this result is that we also prove the following:
if no corner of ${\mathcal C}$ is mapped
into the interior of a lozenge in ${\mathcal C}$ then one can homotope
$T$ to a union of Birkhoff annuli so that the periodic orbits in the 
annuli do not
intersect the union of the interiors of the Birkhoff annuli. This is half way to 
producing an embedded torus homotopic to $T$ which is 
a union of Birkhoff annuli.
The second conclusion of theorem C implies for instance that if $T$ is 
the boundary torus between 2 hyperbolic pieces in the torus decomposition,
then the situation of theorem C cannot happen.
The general result concerning best position of embedded tori is the following
(Theorem ~\ref{embedd}):

\vskip .1in
\noindent
{\bf {Theorem D}} $-$ Suppose that $M$ is orientable and
that the pseudo-Anosov flow is not topologically
equivalent to a suspension Anosov flow.
Let $T$ be an embedded, incompressible
torus in $M$. Then either 1) $T$ is isotopic to an embedded
Birkhoff torus, or
2) $T$ is homotopic to a weakly embedded Birkhoff 
torus $T'$ and $T$ (or $T'$) is contained in a periodic Seifert
fibered piece, or
3) $T$ is isotopic to the boundary of the tubular neighborhood of an embedded Birkhoff-Klein bottle
contained in a free Seifert piece.
\vskip .1in

{\em {Weakly embedded}} means that $T'$ is embedded except perhaps 
along the closed orbits contained in the Birkhoff annuli.
All the possibilities in Theorem D indeed happen: 1) is the typical situation
when the flow is a geodesic flow of an orientable
surface (or more generally, a Handel-Thurston example, see \cite{Ha-Th}),
2) occurs in the Bonatti-Langevin examples \cite{Bo-La}, and 3) occurs in the geodesic flow on
non-orientable closed surfaces (see the last remark of section~\ref{seifconj}).

\vskip .1in
One consequence of this study of standard forms for tori is the
following (Proposition~\ref{singint}):

\vskip .1in
\noindent
{\bf {Theorem E}} $-$ Let $\alpha$ be a singular orbit of a pseudo-Anosov flow.
Then $\alpha$ is homotopic into a piece of the torus decomposition of the 
manifold.
\vskip .1in

If the manifold is atoroidal or Seifert fibered the statement is trivial.
Notice that the result is clearly not true for regular periodic orbits
as there are many transitive Anosov flows in graph
manifolds which are not Seifert fibered \cite{Ha-Th}.

The results above help tremendously to understand canonical flow
neighborhoods 
associated to periodic Seifert fibered pieces (section~\ref{sec:perio}):

\vskip .1in
\noindent
{\bf {Theorem F}} $-$ 
Let $\Phi$ be a pseudo-Anosov flow in $M$ orientable and
let $P$ be a periodic Seifert fibered piece of the
torus decomposition of $M$.
 Then there is a finite union $Z$ of Birkhoff
annuli, which is embedded except perhaps at the boundaries of the
Birkhoff annuli and which is a model for the  core of $P$: a sufficiently small
neighborhood of $Z$ is a representative for the Seifert piece $P$. 
The finite union $Z$ is 
well defined up to flow isotopy.
\vskip .1in

This is a remarkably simple form for the flow in the piece $P$.
It follows that the dynamics of the flow restricted to the
piece is extremely simple: there are finitely many closed orbits $-$
the union of the boundary of the Birkhoff annuli. All other orbits
are either in the stable or unstable leaves of the closed orbits
or enter and exit the manifold with boundary.
In addition the theorem says that in periodic Seifert pieces
the flow is intimately connected with the topology of the
Seifert piece. This provides a strong relation between
dynamics and topology.
Notice for future reference that it is not true in general
that one can make the boundary of a neighborhood of
the Birkhoff annuli transverse to the flow.

The basic ideas of the proof of theorem F are as follows:
The fiber in $P$ is represented by a closed orbit of the flow and any
${\bf Z}^2$ in $\pi_1(P)$ can be represented by a Birkhoff
torus which has this closed orbit. The Seifert piece being
periodic implies that the situation of theorem C cannot happen,
and we can adjust the Birkhoff annuli so that the interiors
are embedded and disjoint. Three manifold topology and the
study of chains of lozenges implies that we can choose 
finitely many of these Birkhoff annuli which carry all of
$\pi_1(P)$. This produces $Z$ and $P$ can be represented by
a small neighborhood of $Z$.

\vskip .1in
We are now ready to describe 
the main family of examples we produce, see section \ref{sec:examples}. 
The contruction uses the understanding of the structure given by Theorem F and it shows that
the description given in Theorem F is actually realizable in a wide variety 
of cases, at least when 
one requires that the boundary of the periodic Seifert pieces are transverse
to the flow.

In fact in the construction of theorem G we 
{\underline {allow}} one prongs.
If there are one prongs, these generalized pseudo-Anosov flows
are called {\em one prong pseudo-Anosov flows}.
Classically they originated in Thurston's work \cite{Th2} since he
constructed 
pseudo-Anosov homeomorphisms of the two sphere, having for example
four one prong singularities. A suspension of these homeomorphisms
produces a one prong pseudo-Anosov flow.
In this case the universal cover is ${\bf S}^2 \times {\bf S}^1$ and hence
$M$ is not irreducible,
but still the flow in the universal cover is topologically a product
flow and the orbit space is  ${\bf S}^2$ which is a two manifold.
Other examples with one prongs are obtained doing Dehn surgery on
periodic orbits of pseudo-Anosov flows \cite{Fr}, but here very little is known
about the resulting one prong pseudo-Anosov flows.

\vskip .1in
\noindent
{\bf {Theorem G}} $-$ There is a large family of (possibly one
prong) pseudo-Anosov flows 
in graph manifolds and manifolds fibering over the circle
with fiber a torus,   where the flows are
 obtained by glueing simple
building blocks. 
The building blocks are homeomorphic to solid tori and
they are canonical flow neighborhoods
of intrinsic (embedded)
Birkhoff annuli. The building blocks 
have tangential boundary, transverse
boundary and only 2 periodic orbits. A collection of blocks is first glued along annuli
in their tangential boundary
to obtain Seifert fibered manifolds with boundary,  and which have a 
semi-flow
transverse to the boundary with finitely many periodic orbits.
Under very general and specified conditions these can be glued
along their boundaries (transverse to the flow) 
to produce (possibly one prong) pseudo-Anosov flows in the
resulting closed manifolds.
In addition one can do any Dehn surgery (except for one) in the periodic
orbits of the middle step to obtain new (possibly one prong) pseudo-Anosov
flows.
\vskip .1in

This family is a 
vast generalization of the Bonatti-Langevin construction \cite{Bo-La}.
The constructions in theorem G are very general producing 
for example one prong pseudo-Anosov flows in all but one torus bundle
over the circle. 
This shows that theorem B also does not hold if one allows one prongs.
In the construction in theorem G, 
if the middle step produces 
a flow without one prong periodic orbits (this is immediate to
check), then the resulting final flow
in the closed manifold will be pseudo-Anosov in a graph
manifold. All the Seifert fibered pieces are periodic pieces.
This construction is extremely general producing a very large class
of new examples.

An appealing way to describe the examples of theorem G in the absence of the
Dehn surgeries is the following:
the manifolds with transverse boundary in the middle step
are circle bundles, with fibers preserved by the local flow, and projecting 
to a local flow of Morse-Smale type on a surface $S$ with boundary: there
is a finite number of singular points (prong singularities) in 
$S$, stable and unstable manifolds
joining the singular points to the boundary, and all other orbits go from boundary component to another.
This picture can be encoded in the combinatorial data of a fat graph satisfying some conditions.

In a subsequent article we will show that the examples of theorem G without
one prongs (hence pseudo-Anosov) are
transitive. By construction they have many  transverse tori to the flow.
We stress that the family of examples in theorem G is entirely new
and is constructed by assembling non pseudo-Anosov blocks of semi-flows and glueing.
Recall that the majority of constructions of pseudo-Anosov flows,
besides those transverse to foliations, 
are obtained by modifying some original pseudo-Anosov flow:
\ 1) Dehn surgery on closed orbits of flows,
introduced by Goodman for Anosov flows \cite{Go}
and extended by Fried for pseudo-Anosov flows
\cite{Fr}, \ 2) The derived from Anosov construction
of blow up of orbits and glueing by Franks and Williams \cite{Fr-Wi},
\ 3) The shearing construction along tori, by Handel and Thurston
\cite{Ha-Th}.

When the flows of theorem G do not have $p$-prong singularities or one prongs,
they are new examples of Anosov flows.
In this case these Anosov flows are never contact.
This is because all contact Anosov flows
are $\rrrr$-covered \cite{Ba6}.
In addition if an Anosov flow is $\rrrr$-covered and admits a 
transverse torus $T$, then it has to be topologically equivalent to
a suspension  and $T$ must be a cross section \cite{Fe1,Ba1}. 
In our situation consider the transverse tori
which are the boundary components of the middle glueing pieces:
they do not intersect all orbits of the flow and cannot be cross sections.
This proves that the flows are not contact.

\vskip .1in
We now describe additional families of new examples.
At the end of section~\ref{seifconj}, we produce some interesting new examples
using branched cover constructions:

\vskip .1in
\noindent
{\bf {Theorem H}} $-$ 1) There is an infinite family of one prong pseudo-Anosov flows
with two one prong singular orbits and no other singular orbits where the
manifold is Seifert fibered. They are doubly branched covered by 
the Handel-Thurston examples \cite{Ha-Th}.
2) There are also infinitely many examples of one
prong pseudo-Anosov flows which are doubly branched covered by
a geodesic flow in a hyperbolic surface and where the original manifolds
are not irreducible.
\vskip .1in

As remarked above the Handel Thurston examples are in graph manifolds
which are not Seifert fibered.
Part 1) of theorem H shows that theorem A does not hold in Seifert fibered manifolds
if one allows one prong orbits. 
The manifolds in part 2) are not irreducible and neither homeomorphic
to ${\bf S}^2 \times {\bf S}^1$. At the beginning of section~\ref{sec:examples},
we improve these examples to show that a mixed behavior of 
Seifert fibered pieces is possible:

\vskip .1in
\noindent
{\bf {Theorem I}} $-$ There are examples of pseudo-Anosov flows in graph
manifolds with one periodic piece and an arbitrary  number of free pieces. 
\vskip .1in

The flows in theorem I are obtained as branched cover constructions
of the examples 2) in theorem H.

At this point there is no good understanding of the general structure of one prong
pseudo-Anosov flows and they can be much less well behaved than pseudo-Anosov flows.
In this article we do not analyse at all the structure of one prong pseudo-Anosov flows,
but only construct many examples of these, some of which highlight the
differences with pseudo-Anosov flows in Seifert fibered manifolds,
solvable manifolds and graph manifolds.

The first examples of an Anosov flow in a graph manifold where the
pieces are periodic were constructed by Bonatti and Langevin
\cite{Bo-La}: they are extremely special cases of the examples provided by
theorem G.  
The general cases requires different arguments to prove the
pseudo-Anosov behavior.
The systematic study of Anosov flows in graph
manifolds was started in \cite{Ba3,Ba4}.

In the final section of this article we discuss further
questions/comments/conjectures concerning pseudo-Anosov flows
in toroidal manifolds.
In this article we do not really analyse free Seifert pieces, but
in the final section we have some comments and questions about them.


\section{Background}

\noindent
{\bf {Pseudo-Anosov flows $-$ definitions}}

\begin{define}{(pseudo-Anosov flows)}
Let $\Phi$ be a flow on a closed 3-manifold $M$. We say
that $\Phi$ is a {\em pseudo-Anosov flow} if the following conditions are
satisfied:


- For each $x \in M$, the flow line $t \to \Phi(x,t)$ is $C^1$,
it is not a single point,
and the tangent vector bundle $D_t \Phi$ is $C^0$ in $M$.

- There are two (possibly) singular transverse
foliations $\ls, \lu$ which are two dimensional, with leaves saturated
by the flow and so that $\ls, \lu$ intersect
exactly along the flow lines of $\Phi$.

- There is a finite number (possibly zero) of periodic orbits $\{ \gamma_i \}$,
called singular orbits.
A stable/unstable leaf containing a singularity is homeomorphic 
to $P \times I/f$
where $P$ is a $p$-prong in the plane and $f$ is a homeomorphism
from $P \times \{ 1 \}$ to $P \times \{ 0 \}$.
In addition $p$ is at least $3$.

- In a stable leaf all orbits are forward asymptotic,
in an unstable leaf all orbits are backwards asymptotic.
\end{define}

Basic references for pseudo-Anosov flows are \cite{Mo1,Mo2} and
\cite{An} for Anosov flows. A fundamental remark is that the ambient manifold
supporting a pseudo-Anosov flow (without $1$-prongs) is necessarily irreducible - the
universal covering is homeomorphic to ${\bf R}^3$ (\cite{Fe-Mo}).

\begin{define}{(one prong pseudo-Anosov flows)}{}
A flow $\Phi$ is a one prong pseudo-Anosov flow in $M^3$ if it satisfies
all the conditions of the definition of pseudo-Anosov flows except
that the $p$-prong singularities  can also be
$1$-prong ($p = 1$).
\end{define}

\vskip .05in
\noindent
{\bf {Torus decomposition}} 

Let $M$ be an irreducible closed $3$--manifold. If $M$ is orientable, it  has a unique (up to isotopy) 
minimal collection of disjointly embedded incompressible tori such that each component of $M$ 
obtained by cutting along the tori is either atoroidal or Seifert-fibered \cite{Ja,Ja-Sh}
and the pieces are isotopically maximal with this property. 
If $M$ is not orientable,
a similar conclusion holds; the decomposition has to be performed along tori, but also along
some incompressible embedded Klein bottles.

Hence the notion of maximal Seifert pieces in $M$ is well-defined up to isotopy. If $M$ admits
a pseudo-Anosov flow, we say that a Seifert piece $P$  is {\em periodic} if there is a 
Seifert fibration on $P$ for which a regular
fiber is freely homotopic to a periodic orbit of $\Phi$. If not,
the piece is called {\em free.} 

\vskip .05in
\noindent{\bf {Remark. }} 
In a few circunstances, the Seifert fibration is not unique: it happens for example 
when $P$ is homeomorphic to a twisted line bundle over the Klein bottle or
$P$ is $T^2 \times I$. 
We stress out that our convention is to say that the Seifert piece is free
if 
{\underline {no}} Seifert fibration in $P$ has fibers homotopic to a periodic orbit.

\vskip .1in
\noindent
{\bf {Orbit space and leaf spaces of pseudo-Anosov flows}} 

\vskip .05in
\noindent
\underline {Notation/definition:} \ 
We denote by $\mi$ the universal covering of $M$, and by $\pi_1(M)$ the fundamental group of $M$,
considered as the group of deck transformations on $\mi$.
The singular
foliations lifted to $\mi$ are
denoted by $\wls, \wlu$.
If $x \in M$ let $W^s(x)$ denote the leaf of $\ls$ containing
$x$.  Similarly one defines $W^u(x)$
and in the
universal cover $\ws(x), \wu(x)$.
Similarly if $\alpha$ is an orbit of $\Phi$ define
$W^s(\alpha)$, etc...
Let also $\wwp$ be the lifted flow to $\mi$.

\vskip .05in

We review the results about the topology of
$\wls, \wlu$ that we will need.
We refer to \cite{Fe2,Fe3} for detailed definitions, explanations and 
proofs.
The orbit space of $\wwp$ in
$\mi$ is homeomorphic to the plane $\rrrr^2$ \cite{Fe-Mo}
and is denoted by $\oo \cong \mi/\wwp$. There is an induced action of $\pi_1(M)$ on $\oo$.
Let 

$$\Theta: \ \mi \ \rightarrow \ \oo \ \cong \ \rrrr^2$$

\noindent
be the projection map: it is naturally $\pi_1(M)$-equivariant.
If $L$ is a 
leaf of $\wls$ or $\wlu$,
then $\Theta(L) \subset \oo$ is a tree which is either homeomorphic
to $\rrrr$ if $L$ is regular,
or is a union of $p$-rays all with the same starting point
if $L$ has a singular $p$-prong orbit.
In addition $L$ is a closed subset of $\mi$ or equivalently
$L$ is properly embedded in $\mi$.
The foliations $\wls, \wlu$ induce $\pi_1(M)$-invariant singular $1$-dimensional foliations
$\oos, \oou$ in $\oo$. Its leaves are $\Theta(L)$ as
above.
If $L$ is a leaf of $\wls$ or $\wlu$, then 
a {\em sector} is a component of $\mi - L$.
Similarly for $\oos, \oou$. 
If $B$ is any subset of $\oo$, we denote by $B \times \rrrr$
the set $\Theta^{-1}(B)$.
The same notation $B \times \rrrr$ will be used for
any subset $B$ of $\mi$: it will just be the union
of all flow lines through points of $B$.
We stress that for pseudo-Anosov flows there are at least
$3$-prongs in any singular orbit ($p \geq 3$).
For example, the fact that the orbit space in $\mi$ is
a $2$-manifold is not true in general if one allows
one prongs.

\begin{define}
Let $L$ be a leaf of $\wls$ or $\wlu$. A slice of $L$ is 
$l \times \rrrr$ where $l$ is a properly embedded
copy of the reals in $\Theta(L)$. For instance if $L$
is regular then $L$ is its only slice. If a slice
is the boundary of a sector of $L$ then it is called
a line leaf of $L$.
If $a$ is a ray in $\Theta(L)$ then $A = a \times \rrrr$
is called a half leaf of $L$.
If $\zeta$ is an open segment in $\Theta(L)$ 
it defines a {\em flow band} $L_1$ of $L$
by $L_1 = \zeta \times \rrrr$.
We use the same terminology of slices and line leaves
for the foliations $\oos, \oou$ of $\oo$.
\end{define}

If $F \in \wls$ and $G \in \wlu$ 
then $F$ and $G$ intersect in at most one
orbit.

We abuse convention and call 
a leaf $L$ of $\wls$ or $\wlu$ {\em periodic}
if there is a non trivial covering translation
$g$ of $\mi$ with $g(L) = L$. This is equivalent
to $\pi(L)$ containing a periodic orbit of $\Phi$.
In the same way an orbit 
$\gamma$ of $\wwp$
is {\em periodic} if $\pi(\gamma)$ is a periodic orbit
of $\Phi$. Observe that in general, the stabilizer of an element $\alpha$
of $\oo$ is either trivial, or a cyclic subgroup of $\pi_1(M)$.

\vskip .2in
\noindent
{\bf {Product regions}}

Suppose that a leaf $F \in \wls$ intersects two leaves
$G, H \in \wlu$ and so does $L \in \wls$.
Then $F, L, G, H$ form a {\em rectangle} in $\mi$, ie. every stable leaf between $F$ and $L$
intersects every unstable leaf between $G$ and $H$. In particular, 
there is  no singularity 
in the interior of the rectangle \cite{Fe3}.

There will be two generalizations of rectangles: 1) perfect fit, which is a rectangle
with one corner orbit removed 
(definition \ref{def:pfits}) and 2) lozenge, which is a 
 rectangle with two opposite corners removed (definition \ref{def:lozenge}).
We will also denote by rectangles, perfect fits, lozenges
and product regions the projection of these regions to
$\oo \cong \rrrr^2$.

\begin{define}{}{}
Suppose $A$ is a flow band in a leaf of $\wls$.
Suppose that for each orbit $\alpha$ of $\wwp$ in $A$ there is a
half leaf $B_{\alpha}$ of $\wu(\alpha)$ defined by $\alpha$ so that: 
for any two orbits $\gamma, \beta$ in $A$ then
a stable leaf intersects $B_{\beta}$ if and only if 
it intersects $B_{\gamma}$.
%
%
This defines a stable product region which is the union
of the $B_{\gamma}$.
Similarly define unstable product regions.
\label{defsta}
\end{define}

The main property of product regions is the following:
for any $F \in \wls$, $G \in \wlu$ so that 
$(i) \ F \cap A  \
\not = \ \emptyset \ \ {\rm  and} \ \ 
 (ii) \ G \cap A \ \not = \ \emptyset,
\ \ \ {\rm then} \ \ 
F \cap G \ \not = \ \emptyset$.
There are no singular orbits of 
$\wwp$ in $A$.

%

\begin{theorem}{(\cite{Fe3})}{}
Let $\Phi$ be a pseudo-Anosov flow. Suppose that there is
a stable or unstable product region. Then $\Phi$ is 
topologically equivalent to a suspension Anosov flow.
In particular $\Phi$ is non singular.
\label{prod}
\end{theorem}

%

In particular:

\begin{define}{(\cite{Fe1})}\label{def:product}
A pseudo-Anosov flow is {\em product} (or {\em splitting} in the terminology of \cite{franks})
if the entire orbit space is a product region, ie if every leaf of
its stable foliation $\wls$ intersects every leaf of its unstable foliation $\wlu$.
\end{define}

\begin{proposition}{}{}
A (topological) Anosov flow is product if and only if it is topologically equivalent to a suspension 
Anosov flow. In particular $M$ fibers over the circle with
fiber a torus and Anosov monodromy.
\label{susp}
\end{proposition}

Hence, in the sequel, 
we will use {\em product pseudo-Anosov flow} 
as an abbreviation for {\em pseudo-Anosov flow 
topologically equivalent to a suspension.}

\vskip .2in
\noindent
{\bf {Perfect fits, lozenges and scalloped chains}}

Recall that a foliation $\fol$ in $M$ is $\rrrr$-covered if the
leaf space of $\fn$ in $\mi$ is homeomorphic to the real line $\rrrr$
\cite{Fe1}.

\begin{define}{(\cite{Fe2,Fe3})}\label{def:pfits}
Perfect fits -
Two leaves $F \in \wls$ and $G \in \wlu$, form
a perfect fit if $F \cap G = \emptyset$ and there
are half leaves $F_1$ of $F$ and $G_1$ of $G$ 
and also flow bands $L_1 \subset L \in \wls$ and
$H_1 \subset H \in \wlu$,
so that 
%
the set 

$$\overline F_1 \cup \overline H_1 \cup 
\overline L_1 \cup \overline G_1$$

\noindent
separates $M$ and forms an a rectangle $R$ with a corner removed:
The joint structure of $\wls, \wlu$ in $R$ is that of
a rectangle with a corner orbit removed. The removed corner
corresponds to the perfect of $F$ and $G$ which 
do not intersect.
\end{define}

We refer to fig. \ref{loz}, a for perfect fits.
There is a product structure in the interior of $R$: there are
two stable boundary sides and two unstable boundary
sides in $R$. An unstable
leaf intersects one stable boundary side (not in the corner) if
and only if it intersects the other stable boundary side
(not in the corner).
We also say that the leaves $F, G$ are {\em asymptotic}.

%
%
%
%

\begin{figure}
\centeredepsfbox{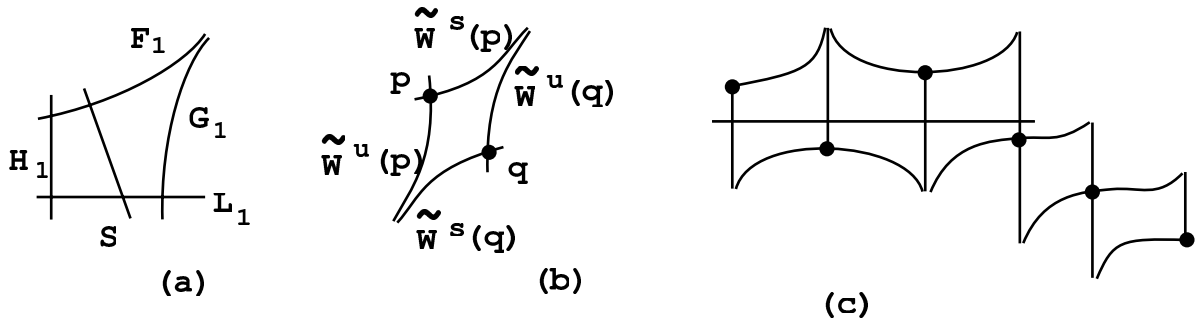}
\caption{a. Perfect fits in $\mi$,
b. A lozenge, c. A chain of lozenges.}
\label{loz}
\end{figure}

\begin{define}{(\cite{Fe2,Fe3})}\label{def:lozenge}
Lozenges - A lozenge $R$ is a region of $\mi$ whose closure
is homeomorphic to a rectangle with two corners removed.
More specifically two points $p, q$ define the corners
of a lozenge if there are half leaves $A, B$ of
$\ws(p), \wu(p)$ defined by $p$
and  $C, D$ half leaves of $\ws(q), \wu(q)$ defined by $p, q$, so
that $A$ and $D$ form a perfect fit and so do
$B$ and $C$. The region bounded by the lozenge
$R$ does not have any singularities.
%
%
%
%
The sides of $R$ are $A, B, C, D$.
The sides are not contained in the lozenge,
but are in the boundary of the lozenge.
There may be singularities in the boundary of the lozenge.
See fig. \ref{loz}, b.
\end{define}


There are no singularities in the lozenges,
which implies that
$R$ is an open region in $\mi$.


Two lozenges are {\em adjacent} if they share a corner and
there is a stable or unstable leaf
intersecting both of them, see fig. \ref{loz}, c.
Therefore they share a side.
A {\em chain of lozenges} is a collection $\{ \cc _i \}, 
i \in I$, where $I$ is an interval (finite or not) in ${\bf Z}$;
so that if $i, i+1 \in I$, then 
${\mathcal C}_i$ and ${\mathcal C}_{i+1}$ share
a corner, see fig. \ref{loz}, c.
Consecutive lozenges may be adjacent or not.
The chain is finite if $I$ is finite.

\begin{define}{(scalloped chain)}\label{def:scallop}
Let ${\mathcal C}$ be a chain of lozenges.
If any two
successive lozenges in the chain are adjacent along 
one of their unstable sides (respectively stable sides),
then the chain is called {\em s-scalloped} 
(respectively {\em u-scalloped}) (see
fig. \ref{pict} for an example of a $s$-scalloped region). 
Observe that a chain is s-scalloped if
and only if there is a stable leaf intersecting all the 
lozenges in the chain. Similarly, a chain is u-scalloped
if and only if there is an unstable leaf intersecting 
all the lozenges in the chain. 
The chains may be infinite.
A scalloped chain is a chain that is either $s$-scalloped or
$u$-scalloped.
\end{define}

For simplicity when considering scalloped chains we also include any half leaf which is
a boundary side of two of the lozenges in the chain. The union of these
is called a {\em {scalloped region}} which is then a connected set.

We say that two orbits $\gamma, \alpha$ of $\wwp$ 
(or the leaves $\ws(\gamma), \ws(\alpha)$)
are connected by a 
chain of lozenges $\{ {\mathcal C}_i \}, 1 \leq i \leq n$,
if $\gamma$ is a corner of ${\mathcal C}_1$ and $\alpha$ 
is a corner of ${\mathcal C}_n$.

\vskip .2in
\noindent
{\bf {Fat tree of lozenges}}

\begin{define}{(fat tree of lozenges ${\mathcal G}(\alpha)$)}
\label{def:tree}
Let  $\alpha$ be an orbit of $\wwp$. We define ${\mathcal G}(\alpha)$ as the graph such that:

-- the vertices ${\mathcal G}(\alpha)$ are orbits of $\wwp$ connected to
$\alpha$ by a chain of lozenges, 

-- there is an edge in ${\mathcal G}(\alpha)$ between $\beta$ and $\gamma$ if and only if there is
a lozenge with corners $\alpha$, $\beta$.
\end{define}

One easily proves (see for example \cite{Fe2} for Anosov flows):

\begin{proposition}
For every $\alpha$ in $\oo$, ${\mathcal G}(\alpha)$  is a tree.
\end{proposition}

In particular for any two orbits $\delta, \gamma$ connected by a chain
of lozenges, then  there is a unique indivisible or minimal 
chain of lozenges $-$ where no backtracking on lozenges
is allowed.

The proposition implies that
${\mathcal G}(\alpha)$ is naturally embedded in the $2$-plane $\oo$. 
Hence, once fixed an orientation on $\oo$, there is, for every vertex $\alpha$, 
a cyclic order on the set
of edges of ${\mathcal G}(\alpha)$ adjacent to $\alpha$.
Moreover, ${\mathcal G}(\alpha)$
is naturally equipped with a structure of a {\em fat graph:} it is a retract of 
an orientable surface with boundary (the tubular neighborhood of its embedding in $\oo$).
This object will be extremely useful in this article.

If ${\mathcal C}$ is a lozenge with corner orbits $\beta, \gamma$ and
$g$ is a non trivial covering translation 
leaving $\beta, \gamma$ invariant (and so also the lozenge),
then $\pi(\beta), \pi(\gamma)$ are closed orbits
of $\wwp$ which are freely homotopic to the 
{\underline {inverse}} of each
other \cite{Fe2}.
Here we consider the closed orbits
$\pi(\beta), \pi(\gamma)$ traversed in the positive flow direction
and we allow $\pi(\beta), \pi(\gamma)$ to be non indivisible
closed orbits. In other words it is the closed orbit associated
to the deck transformation $g$, which may not be indivisible.

\begin{theorem}{(\cite{Fe2,Fe3})}{}
Let $\Phi$ be a pseudo-Anosov flow in $M^3$ closed and let 
$F_0 \not = F_1 \in \wls$.
Suppose that there is a non trivial covering translation $g$
with $g(F_i) = F_i, i = 0,1$.
Let $\alpha_i, i = 0,1$ be the periodic orbits of $\wwp$
in $F_i$ so that $g(\alpha_i) = \alpha_i$.
Then $\alpha_0$ and $\alpha_1$ are connected
by a finite chain of lozenges 
$\{ {\mathcal C}_i \}, 1 \leq i \leq n$ and $g$
leaves invariant each lozenge 
${\mathcal C}_i$ as well as their corners.
\label{chain}
\end{theorem}

In particular:

\begin{proposition}
Let $g$ be a non-trivial element of $\pi_1(M)$ fixing two orbits $\alpha$ 
and $\gamma$. Then ${\mathcal G}(\alpha)={\mathcal G}(\gamma)$.
\end{proposition}

We think of a fat tree as a simplicial tree.
Observe that $g$ as above
naturally acts simplicially on ${\mathcal G}(\alpha)$. It does not necessarily 
preserve the cyclic order 
on links of vertices in ${\mathcal G}(\alpha)$, since it does not necessarily preserve 
the orientation of $\oo$.

\begin{define}{(the tree ${\mathcal G}(g)$)} Let $g$ in $\pi_1(M)$ 
fixing an orbit $\alpha$. 
The $g$-fixed points in ${\mathcal G}(\alpha)$ 
form a connected subtree because of simplicial action.
This subtree is denoted by ${\mathcal G}(g)$.
\end{define}

From this observation we infer several interesting facts:

\begin{proposition}\label{pro:treefacts}
Let $g$ be a non-trivial element of $\pi_1(M)$. All of
the following statements are true:

\begin{enumerate}

\item For any $n \neq 0$, $g$ admits a fixed point in $\oo$ if and only 
if $g^n$ admits a fixed point in $\oo$ .
\item Assume that $g$ fixes an orbit $\alpha \in \oo$. 
Then, some positive power $g^{p}$ acts trivially on ${\mathcal G}(\alpha)$. 
\item Let $p$ be an integer as in item 2. Let $Z(g^{p})$ be the
pseudocentralizer of $g^{p}$ in $\pi_1(M)$, ie. the subgroup comprised of
elements $f$ such that $fg^{p}f^{-1}=g^{\pm p}$. Then $Z(g^{p})$
acts on the tree ${\mathcal G}(\alpha)={\mathcal G}(g^{p})$. 
\item Assume that $g$ preserves a lozenge $\mathcal L$. Then, $g$ 
preserves individually each corner of $\mathcal L$. Moreover, $g$
preserves the orientation of $\oo$, and acts trivially on 
${\mathcal G}(\alpha)={\mathcal G}(\beta)$, where $\alpha$ and $\beta$ 
are the corners of $\mathcal L$.
\end{enumerate}
\end{proposition}

\begin{proof}\-
\begin{enumerate}
\item Suppose $g^n(\alpha) = \alpha$ with $\alpha$ orbit of $\wwp$.
Then $g^n(g(\alpha)) = g(\alpha)$, so by theorem \ref{chain},
$\alpha$ and $g(\alpha)$ are connected by a chain of lozenges
and therefore ${\mathcal G}(\alpha) = {\mathcal G}(g(\alpha)) = 
g({\mathcal G}(\alpha))$. Hence $g$ acts on ${\mathcal G}(\alpha)$.
The result now follows easily
from the fact that if $g$ acts freely on a tree, then $g^n$ acts freely on the tree.
\item Let $k$ be the number of prongs at $\alpha$.
Then $g^{2}$ preserves the orientation of $\oo$, 
hence the cyclic ordering of the link of $\alpha$. 
Hence $g^{2k}$ fixes every vertex of ${\mathcal G}(\alpha)$
adjacent to $\alpha$. But if $g^{2k}$ fixes a point $\gamma$
in ${\mathcal G}(\alpha)$ and an edge in ${\mathcal G}(\alpha)$ adjacent to $\gamma$, 
it fixes every vertex adjacent to $\gamma$ (once more, 
due to the preservation of orientation of $\oo$ by $g^{2k}$).
Our claim follows by induction.
\item Let $f$ in $Z(g^p)$ and $\beta$ a vertex in ${\mathcal G}(\alpha)$.
Then $g^p f(\beta) = ff^{-1}g^pf(\beta) = f(g^{\pm p}(\beta))=f(\beta)$ by (2). By theorem \ref{chain} $f(\beta)$ 
is in ${\mathcal G}(\alpha)$ and so $f$ acts on ${\mathcal G}(\alpha)$.
\item Let $\alpha$, $\beta$ be the corners of $\mathcal L$. 
Assume by way of contradiction that $g(\alpha)=\beta$ and $g(\beta)=\alpha$.
Let $A$, $C$ be the stable half leaves of $\ws(\alpha)$, $\ws(\beta)$ contained 
in the closure of $\mathcal L$. Then, $g(A)=C$, and composing $g$
with the holonomy map from $C$ to $A$ along leaves of $\wlu$ defines an orientation 
reversing map from $A$ onto itself. This map must admit
a fixed point, hence there is a leaf $U$ of $\wls$ fixed by $g$ and 
intersecting $\mathcal L$. 
Now $g^2$ fixes $U$ and $A$ and hence leaves invariant the orbit
$U \cap A$. This produces 2 distinct periodic orbits in $\pi(A)$,
contradiction.

Hence, $g$ fixes $\alpha$ and $\beta$. Keeping the notation above, 
we have $g(A)=A$ and $g(B)=B$ (where $B$ is the $g$ invariant
unstable half-leaf 
of $\wu(\alpha)$ in the boundary of $\mathcal L$). It follows
that $g$ preserves the orientation of $\oo$. It therefore preserves the cyclic 
ordering along vertices of ${\mathcal G}(\alpha)$. It follows as in item 2 that
$g$ acts trivially on ${\mathcal G}(\alpha)$.
\end{enumerate}

\end{proof}

The main result concerning non Hausdorff behavior in the leaf spaces
of $\wls, \wlu$ is the following:

\begin{theorem}{\cite{Fe2,Fe3}}\label{theb}
Let $\Phi$ be a pseudo-Anosov flow in $M^3$. 
Suppose that $F \not = L$
are not separated in the leaf space of $\wls$.
Then $F$ is periodic and so is $L$. More precisely,
there is a non-trivial element $g$ of $\pi_1(M)$ such that $g(F)=F$ and $g(L)=L$.
Moreover, let $\alpha$, $\beta$ be the unique $g$-fixed points in $F$, $L$, respectively.
Then, the chain of lozenges connecting $\alpha$ to $\beta$ is s-scalloped (see figure \ref{pict}).
\end{theorem}

%

\begin{figure}
\centeredepsfbox{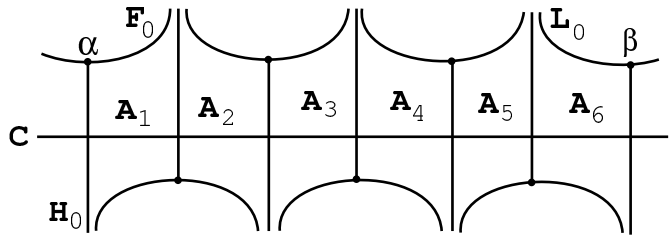}
\caption{
The correct picture between non separated
leaves of $\wls$.}
\label{pict}
\end{figure}

\vskip .2in
\noindent
{\bf {Non-Hausdorff trees}}
\vskip .1in

A {\em segment} is a set with a linear order which is 
isomorphic to an interval in $\rrrr$:
$[0,1], [0,1), (0,1)$ or $[0,0]$. 
Type $(0,1)$ is called an {\em open} segment
and type $[0,0]$ is a {\em degenerate} segment.
A {\em closed} segment is one of type either $[0,0]$ or $[0,1]$, ie. admitting a minimal and a maximal element.
A {\em half open} segment is one of type $[0,1)$, where we also consider the reverse linear
order.
A {\em subsegment} $C$ is a subset of a segment $I$  so that
if $x, y$ are in $C$ and $z$ in $C$ satisfies
$x < z < y$, then $z$ is also in $C$.
With the induced linear order,  $C$ is also 
a segment.
If a set $Z$ is a union of segments, then given $x$ in $Z$,
a {\em prong} at $x$ is a segment $I$ in $Z$
of type $[0,1)$ or $[0,1]$
with $x \in I$ corresponding to $0$.
A subprong of a prong $I$ at $x$ is a subsegment of $I$ of
type $[0,1)$ with $x$ corresponding to $0$.
Two prongs $I_1, I_2$ at $x$ are distinct if $I_1 \cap I_2 = \{ x \}$,
or equivalently they do not share a subprong at $x$.

\begin{define}{(non Hausdorff tree)}{\cite{Fe5}}
A non Hausdorff tree is a space $\hp$ satisfying:

1) \ $\hp$ is a union of open segments,


2) \ $\hp$ is arcwise connected $-$ for each $x, y \in \hp$, there
is a finite chain of segments $I_1, ..., I_n$ with
$x \in I_1, y \in I_n$ and $I_i \cap I_{i+1} \not = \emptyset$
for any $1 \leq i < n$,

3)  \ Points separate $\hp$ in the following way  $-$
for any $x \in \hp$ and $I_1, I_2$ distinct prongs at $x$ the
following happens:
Given $y_1 \in I_1 - \{ x \}, \ y_2 \in I_2 - \{ x \}$,
then any finite chain of segments 
 from $y_1$ to $y_2$ (as in (2) above) must contain
$x$ in at least one of the segments.

If $I_1, I_2$ are two segments with
$I_1 \cap I_2$ a single point which is an endpoint
of both $I_1$ and $I_2$, then given compatible
orders in $I_1, I_2$ we extend them to an order
in $I_1 \cup I_2$, which is then 
a segment of $\hp$.
\end{define}

A priori there may be infinitely
or even uncountably many distinct prongs at $x$.

\begin{define}{(topology of $\hp$ $-$ \cite{Fe5})}{}
We say that a subset $A$ of $\hp$ is open in $\hp$ if for any
$x \in A$ the following happens: for any prong $I$ at
$x$, there is a subprong $I'$ at $x$ ($I' \subset I$)
so that $I' \subset A$. 
\end{define}

Equivalently $A$ is open if for any open segment $S$ and
$x$ in $A \cap S$, there is an open subsegment $S'$
containing $x$ and contained in $A$.

It follows from condition 3)  of non Hausdorff trees that
if $I_1$ and $I_2$
are two segments, then $I_1 \cap I_2$ is either empty or
is a subsegment of both $I_1, I_2$, which may be a point.
A point $x \in \hp$ is {\em regular} if given any two
open segments $I_1, I_2$ with $x \in I_1 \cap I_2$, then
$I_1 \cap I_2$ is an {\underline {open}} segment in $\hp$.
Otherwise $x$ is {\em singular} and $\hp$ is ``treelike" in $x$.
Equivalently a point is regular if there are only tw{}o distinct
prongs at $x$.

\vskip .1in
It is easy to check that if $V$ is an interval in $\rrrr$ with
the standard topology and $f: V \rightarrow \hp$ is an
order preserving bijection to a segment in $\hp$, then
$f$ is a continuous map.

Given $x \not = y$ then for any prong at $y$
there is a subprong disjoint from $x$,
hence contained in $\hp - \{ x \}$. 
It follows that 
$\hp - \{ x \}$ is an open set in $\hp$ and therefore
points are closed in $\hp$, that is, $\hp$ satisfies
the $T_1$ property of topological spaces \cite{Ke}.
In general $\hp$ does not satisfy the
Hausdorff property $T_2$ \cite{Ke}.
Given $x \in \hp$ and $I$ a prong at $x$ let

$$A_I \ = \ \{ \ y \in \hp - \{ x \} \ \ | \ \  
{\rm there \ is \ a \ segment \ path} \ \ \gamma \subset \hp - \{ x \}
\ \ {\rm from} \ y \ {\rm to \ some \ point \ in} \ I \ \}.$$

By the above remark, $A_I$ is arcwise connected.
If $I, J$ are prongs at $x$ which share a subprong
then it is easy to see that $A_I = A_J$.
If $I, J$ are distinct prongs at $x$ then
$I \cup J$ is a segment of $\hp$ with
$x$ in the interior of the segment.
If there is a segment path $\gamma \subset \hp - \{ x \}$
 from some $y \in A_I$
to some $z \in A_J$ then one constructs a segment path $\gamma$
contained
in $\hp - \{ x \}$ from some $y' \in I$ to
some $z' \in J$. This contradicts condition (3)
of the definition of non Hausdorff tree.
Hence $A_I \cap A_J = \emptyset$.

In addition given $y \in A_I$ and $J$ a prong at $y$,
there is a subprong $J' \subset \hp - \{ x \}$.
Clearly $J' \subset A_I$. This implies that any
$A_I$ is open in $\hp$ and hence 
$A_I$ is also closed in $\hp - \{ x \}$.
Each $A_I$ is path connected hence connected,
so the collection

$$\{ A_I \}, \  I \  {\rm distinct \ prongs \ at} \  x \ \ \ (1)$$ 

\noindent
is the collection of
connected components of $\hp - \{ x \}$.

In addition suppose that $A_I, A_J$ are distinct,
but there is a path $\alpha$ in $\hp - \{ x \}$ from a point
in $A_I$ to a point in $A_J$ (notice here we consider
a general path). Then since $A_I, A_J$ are path connected,
it follows that $A_I \cup A_J \cup \alpha$ is path 
connected and hence connected in $\hp - \{ x \}$
contradicting the fact that (1) is the family of connected components
of $\hp - \{ x \}$.
It follows that the collection (1) is also the
collection of path components of $\hp - \{ x \}$.

\vskip .1in
\noindent
{\underline {Conclusion}}: distinct prongs at $x$ are in one to
one correspondence with components (or path components) 
of $\hp - \{ x \}$.
For instance $x$ has $p$ prongs if and only if
$\hp - \{ x \}$ has $p$ components.

\vskip .05in
Given $x, y \in \hp$ which are not separated from each other
in $\hp$ we write $x \approx y$.
One says that
$z$ {\em separates} $x$ from $y$ if $x, y$ are in distinct
components of $\hp - \{ z \}$. 
Given any two $x, y \in \hp$ there is a continuous
path $\alpha(t), 0 \leq t \leq 1$ from $x$ to $y$.
Define 

$$(x,y) \ \ = \ \ \{ \ z \in \hp \ \ | \ \ z \ {\rm separates}  \
{\rm from} \ y \ \} \ \ \ {\rm and} \ \ \ 
[x,y] \ = \ (x,y) \cup \{ x \} \cup \{ y \},$$

\noindent
The first is the {\em open block} of $\hp$ with
endpoints $x, y$ and the second is 
the {\em closed block} of $\hp$ with endpoints $x, y$.
In \cite{Fe5} it is proved that
$[x,y]$ is the intersection of all
continuous
paths in $\hp$ from $x$ to $y$.

We remark that when $x, y$
are the endpoints of a segment $I$ of $\hp$, the
notation $[x,y]$ also suggests the segment
$I$ from $x$ to $y$ (there is a unique such segment).
In fact $I$ and $[x,y]$ are the same \cite{Fe5}.
We will also
use the notation $(x,y]$ for half open segments.

As $\hp$ may not be Hausdorff it may be that
$[x,y]$ is not connected.
It turns out that
$[x,y]$ is a union of finitely many closed segments of $\hp$
homeomorphic to either $[0,0]$ or $[0,1]$:


\begin{lemma}{(\cite{Fe5})}{}
For any $x, y \in \hp$ then there are $x_i, y_i \in \hp$ with:
$$[x,y] \ = \ \bigcup_{i=1}^n \ [x_i,y_i], \  \ x_1 = x, \  y_n = y,$$

\noindent
a disjoint union,
where $[x_i,y_i]$ are closed segments in $\hp$.
In addition $y_i \approx x_{i+1}$ for any
$1 \leq i \leq n-1$
and some or all segments $[x_i,y_i]$ may be
degenerate, that is, points.
\label{segme}
\end{lemma}

%
%
%
%

There is a natural pseudo distance in $\hp$:
$d(x,y) =  \# ({\rm components} \ [x,y]) -1$,
see \cite{Ba5,RSS}.
So $d(x,y) = 0$ means there is a segment from $x$ to $y$.
Also $d(x,y)$ is the minimum number of non immersed points
of any path from $x$ to $y$. 

\vskip .1in
We now consider group actions on non Hausdorff trees.
Let $\gamma$ be a homeomorphism of $\hp$. We say that 
$\gamma$ {\em separates points} if $\gamma(x)$ is separated from
$x$ for any $x \in \hp$, that is,
they have disjoint neighborhoods  in
$\hp$. In particular $\gamma$ acts
freely on $\hp$. 
In \cite{Ba5}, the first author constructed
a fundamental axis $\aaa(\gamma)$ if $\gamma$ separates
points in $\hp$ and $\hp$ has no singularities.
In that case $\hp$ is a simply connected $1$-dimensional
manifold and hence  is orientable.
%
%
%
%
%
%

\begin{define}{(fundamental axis)}{\cite{Fe5}}\label{def:axis}
Let $\gamma$ be a homeomorphism of
a non Hausdorff tree $\hp$ so that $\gamma$ has no fixed points.
The fundamental axis of $\gamma$, denoted by $\aaa(\gamma)$ is

$$\aaa(\gamma) \ \ = \ \ 
\{ \ x \in \hp \ \ | \ \ \gamma(x) \in [x,\gamma^2(x)] \ \ \},$$

\noindent
or equivalently $\gamma(x)$ separates $x$ from $\gamma^2(x)$.
\end{define}

If $\gamma(x)$ is not separated from $x$ in $\hp$, we
say that $x$ is an {\em almost invariant} point under $\gamma$.
In \cite{Fe5} the following easy fact is proved:
Let $\gamma$ be a homeomorphism of a non Hausdorff
tree $\hp$ without fixed points.
Then $x \in \aaa(\gamma)$ if and only if
there is a component $U$ to $\hp - \{ x \}$
so that $\gamma(U) \subset U$.
The main result is:

\begin{theorem}{(\cite{Fe5})}{}\label{th:axis}
Let $\gamma$ be a homeomorphism of a non Hausdorff tree  $\hp$
without fixed points. Then $\aaa(\gamma)$ is non empty.
\label{axis}
\end{theorem}

Clearly $\aaa(\gamma)$ is invariant under
$\gamma$. 
Also applying $\gamma^{-2}$ then 
$\gamma^{-1}(x)$ separates $x$ from $\gamma^{-2}(x)$
and so $\aaa(\gamma) = \aaa(\gamma^{-1})$.

\begin{proposition}{}{}
For any $x \in \aaa(\gamma)$, then
$\aaa(\gamma) = 
\cup _{i \in {\bf Z}} [\gamma^i(x), \gamma^{i+1}(x)]$.
\label{charac}
\end{proposition}

%
%
%

\vskip .07in
\noindent
{\bf {Remark:}} \ In general it is not true that if 
$\gamma$ acts freely on $\hp$, then powers of 
$\gamma$ also do. For example let
$\gamma$ have an almost invariant point
$v$ with $\gamma(v) \not = v$, but $\gamma^2(v) = v$.
In this case
$\aaa(\gamma)$ is an open  segment which is not properly
embedded in $\hp$.


\vskip .05in
Let $x \in \aaa(\gamma)$. If $d(x,\gamma(x)) = 0$,
then $x, \gamma(x)$ are connected by a segment in $\hp$.
Since $\gamma(x)$ separates $x$ from $\gamma^2(x)$
it follows that $[x,\gamma(x)] \cup [\gamma(x),\gamma^2(x)]
= [x,\gamma^2(x)]$ is a segment of $\hp$.
It follows that $\aaa(\gamma)$ is an open segment 
of $\hp$, hence homeomorphic
to $\rrrr$. 
If $d(x,\gamma(x)) > 0$, then
$x$ and $\gamma(x)$ are
connected by a chain of closed segments.
It is easy to see that

$$\aaa(\gamma) \ = \ \bigcup _{n \in {\bf Z}} \ [z_i,w_i],$$

\noindent
where $w_i$ is not separated from $z_{i+1}$.
Then $\gamma$ acts as a translation on the set of
segments, that is, there is $k \in {\bf Z}$,
so that $\gamma([z_i,w_i]) = [z_{i+k},w_{i+k}]$
for any $i \in {\bf Z}$.
We abuse notation and say that $\gamma$ acts on ${\bf Z}$.

Notice that if $\gamma$ acts freely and $\gamma$ leaves
invariant an open segment $I$ of $\hp$, then
$\aaa(\gamma) = I$. This is because for any $z \in I$,
$\gamma(x)$ separates $x$ from $\gamma^2(x)$ 
(free action on $I$), so
$I \subset \aaa(\gamma)$.
But $\aaa(\gamma) = \cup_{n \in {\bf Z}} 
[\gamma^n(x),\gamma^{n+1}(x)]$ so $I = \aaa(\gamma)$.
Finally it is also not hard to prove the following:
Let $\gamma, \alpha$ be two commuting homeomorphisms of
$\hp$ which act freely. Then $\aaa(\gamma) = \aaa(\alpha)$,
see \cite{Fe5}.

\section{Actions and pseudo-Anosov flows}

Let $\Phi$ be a pseudo-Anosov flow in $M^3$.
The foliations $\ls, \lu$ have the following local models:
at a non singular point $y$ there is a ball neighborhood
$U$ 
of $y$ in $M$ homeomorphic to ${\bf D}^2 \times [0,1]$ where
the leaves of (say) $\ls$ are of the form ${\bf D}^2 \times 
\{ t \}$.
Near a singular $p$ prong orbit the picture is the
same as a $p$-prong singularity of a pseudo-Anosov
homeomorphism of a surface times an interval. 
For example consider the germ near zero of the foliation 
of the plane whose leaves are the fibers of
the complex map $z \to Re(z^{p-2})$.
This foliation has a $p$-prong singularity
at the origin. The $3$-dimensional picture is obtained
by multiplying this by an interval. 
Similarly for
$\lu$. 
Let $C$ be an interval in $\rrrr$.

\begin{define}{(transverse curves)} Let $\tau: C \rightarrow M$
be a continuous curve. 
Then $\tau$ is transverse to $\ls$ if the following happens:
given $t$ in  $C$ there is a small neighborhood $Z$ of 
$\tau(t)$ where $\tau$ is an injective map to the set of local
sheets of $\ls$.
The same definition works for $\lu, \wls, \wlu$.
\end{define}

Equivalently the curve is always crossing local leaves.
The foliations $\ls, \lu$ blow up to essential laminations.
Hence in $\mi$ being transverse to $\wls$ is equivalent
to $\tau$ inducing an injective map in the leaf space of $\wls$.
For non singular points this is the usual notion of transversality.

We establish some notation.
Let 

$$\hhs \ \ = \ \ {\rm the \ leaf \ space \ of} \ \ \wls
\ \ \ {\rm and} \ \ \ \nu_s: \mi \rightarrow
\hhs  \ \ {\rm the \ projection \ map}.$$

\noindent
Similarly define $\hhu$ and $\nu_u$.
The results below which will be proved for $\hhs$, obviously work also for $\hhu$.

\begin{lemma}{}{} $\hhs$ has a natural structure as a non Hausdorff
tree, where the segments in $\hhs$ are projections of transversals
to $\wls$. Similarly for $\hhu$.
\end{lemma}

\begin{proof}{}
We prove properties (1)-(3) of the definition of non Hausdorff tree.
Given $x$ in $\hhs$ let $p$ in $\nu_s^{-1}(x)$ and $\tau$ an
open transversal to $\wls$ containing $p$. Then $\nu_s(\tau)$ is
an open segment containing $x$. This proves (1).
Let $x, y$ in $\hhs$ and choose $p$ in $\nu_s^{-1}(x)$, 
$q$ in $\nu_s^{-1}(y)$. Connect $p, q$ by a path in
$\mi$ and perturb it slightly to be a concatenation of
transversals. This can be done because it can be done locally.
Hence $x, y$ are connected by a finite collection of segments
in $\hhs$ and this proves (2).

Finally let $I_1, I_2$ be segments in $\hhs$ intersecting only in $x$.
Let $l_1, l_2$ be transversals to $\wls$ with $I_i = \nu_s(l_i)$, 
$i = 1, 2$.
We can assume they share a point $p$ in $\nu_s^{-1}(x)$.
Any two transversals to $\wls$ entering the same component
of $\mi - \ws(p)$ will have subtransversals intersecting the
same leaves of $\wls$ because of the local picture.
Therefore $l_1 - \{ p \}$, 
$l_2 - \{ p \}$ are contained in different components of $\mi - \ws(p)$.
Let now $y_k \in I_k - \{ x \}$, $k = 1, 2$.
Let $J_i, 1 \leq i \leq n$ be a concatenation of segments from $y_1$ to
$y_2$ in $\hhs$. 
There are transversals $\tau_i$ to $\wls$ with
$\nu_s(\tau_i) = J_i$. 
Let $q_1$ in $\tau_1 \cap \nu_s^{-1}(y_1)$ and 
$q_2$ in $\tau_n \cap \nu_s^{-1}(y_2)$.
Since $J_i$ and $J_{i+1}$ intersect we can
connect a point in $\tau_i$ to a point in $\tau_{i+1}$ by a path
in a leaf of $\wls$.
The concatenation of parts of $\tau_i$ and paths in leaves of $\wls$ produces
a path from $q_1$ to $q_2$ in $\mi$. 
Since $\ws(p)$ separates
$\mi$ and $q_1, q_2$ are in different components of the complement,
then this path has to intersect $\ws(p)$.
If it intersects $\widetilde W^s(p)$ 
in a path in $\ws(p)$ then the endpoints of this 
path are in some $\tau_i$ and hence its projection, which is $x$ is
in $J_i$. This proves (3).
\end{proof}

We have two topologies in $\hhs$: the quotient topology from $\nu_s$ 
and the non Hausdorff tree topology. These are the same:

\begin{lemma}{}{}
The quotient topology in $\hhs$ (from $\nu_s: \mi \rightarrow \hhs$) is
the same as the non Hausdorff tree topology in $\hhs$.
\end{lemma}

\begin{proof}{}
Let $A \subset \hhs$ be an open set in the quotient topology and $x$ in $A$.
Let $I$ be a prong at $x$. Then $I = \nu_s(\tau)$ for some
transversal $\tau$ to $\wls$ starting in some $p \in \nu_s^{-1}(x)$.
Since $\nu_s^{-1}(A)$ is open in $\mi$ and $p$ is in $\nu_s^{-1}(A)$
there is a non degenerate subtransversal $\tau'$ of $\tau$ starting
at $p$ and contained in $\nu_s^{-1}(A)$. Let $I' = \nu_s(\tau')$.
Then $I'$ is a prong at $x$ which is a subprong of $I$.
In addition $I'$ is contained in $A$. Therefore $A$ is open
in the non Hausdorff tree topology.

Conversely suppose that $A$ is open in the non Hausdorff tree
topology. By way of contradiction suppose that there is $p$ in
$\nu_s^{-1}(A)$ which is not in the interior of $\nu_s^{-1}(A)$.
Then we can find a sequence $(p_n)_{n \in {\bf N}}$ in $\mi$
converging to $p$ and with $p_n$ not in $\nu_s^{-1}(A)$ for any
$n$.
It follows that $p_n \not \in \ws(p)$ for any $n$ as 
$\nu_s^{-1}(A)$ is $\wls$ saturated.
Up to a subsequence assume there is a component $Z$ of
$\mi - \ws(p)$ containing $p_n$ for every $n$.
Here the condition of finitely many prongs
at singular points is used.
Let $\tau$ be a transversal to $\wls$ starting at $p$ and
entering the component $Z$.
Let $x = \nu_s(p)$ and $I = \nu_s(\tau)$. Then $I$ is a prong at
$p$ and since $A$ is open in the non Hausdorff tree topology,
there is a subprong $I'$ at $x$ with $I'$ contained in $A$.
Let $\tau'$ be the subtransversal of $\tau$ corresponding to
$I'$.  For $n$ sufficiently large $\ws(p_n)$ intersects $\tau'
\subset \nu_s^{-1}(A)$. Hence $p_n$ is in $\nu_s^{-1}(A)$. This
contradiction shows that 
$\nu_s^{-1}(A)$ is open in $\mi$. Therefore $A$ is open in
the quotient topology.
\end{proof}

\noindent
{\bf {Remark}} $-$ A variation of the proof works for non Hausdorff
trees $\hp$ which are ``leaf spaces" of lifts of essential laminations.
The difference is that it is very possible that
there are singularities $\hp$ which have infinitely
many prongs.

\vskip .1in
We say that two leaves $L, F$ of $\wls$ are non separated from
each other if there are $p$ in $L$, $q$ in $F$ and
a sequence of leaves $(L_n)$ of $\wls$ having points $p_n, q_n$ in
$L_n$ with $(p_n)$ converging to $p$ and $(q_n)$ converging
to $q$. 
We call this condition (I) for $L, F$.
Up to subsequence we may assume that $(L_n)$ is
a nested sequence of leaves of $\wls$.
By throwing out a few initial terms in $(p_n), (q_n)$,
this is equivalent to the existence of transversals 
$\tau_L, \tau_F$ to $\wls$ with $\tau_L$ starting
at $p$, $\tau_F$ starting at $q$ with $\tau_L$ containing
all $p_n$ as above and $\tau_L$ containing all $q_n$.
Project to $\hhs$: let

$$x = \nu_s(p), \ \ y = \nu_s(q), \ \ x_n =  \nu_s(p_n), \ \ 
y_n = \nu_s(q_n), \ \ I = \nu_s(\tau_L), \ \ 
J = \nu_s(\tau_F).$$

\noindent
Here $I, J$ are segments in $\hhs$, $I$ is a prong at $x$
and $J$ is a prong at $y$. Also $x_n = y_n$.
If $I_n$ is the subsegment 
of $I$ from $x_1$ to $x_n$ and $J_n$ the subsegment
of $J$ from $y_1$ to $y_n$ then $I_n = J_n$ and therefore
$I - \{ x \} = J - \{ y \}$. Conversely if $x, y$ have
prongs $I, J$ so that $I - \{ x \} = J - \{ y \}$ it is easy
to show that $L = \nu^{-1}(x)$ and $F = \nu^{-1}(y)$ are leaves
of $\wls$ non separated from each other.

We claim that 
condition (I) is also equivalent to condition (II): $L, F$ do not
have disjoint, open, $\wls$ saturated neighborhoods in $\mi$.
In other words $x, y$ do not have disjoint open neighborhoods
in $\hhs$.
Clearly condition (I) implies condition (II).
Conversely suppose that condition 
(II) holds. If $x = y$ then clearly condition 
(I) holds. Suppose then $x, y$ are distinct.
We proved before that for any $z$ in $\hhs$, then
two points are in the same
path component of $\hhs - \{ z \}$ if and only if they are connected
by a segment path in $\hhs$ which does not contain $\{ z \}$
and these path components are open in $\hhs$.
By condition (II) it follows that
for any $z$ in $\hhs - \{ x, y \}$, the points $x, y$ are
in the same component of $\hhs - \{ z \}$.
Hence $(x,y)$ is empty.
By lemma 3.5, page 71 of \cite{Fe5}, there are prongs
$I$ at $x$ and $J$ at $y$ so that
$I - \{ x \} = J - \{ y \}$. 
This is condition (I).

If any of these 2 conditions holds for $x, y$ in $\hhs$ we write 
$x \sim y$.

For $f$ in $\pi_1(M)$ let $Fix(f)$ be those $x$ in $\hhs$ with $f(x) = x$.
Let $Fix^{\sim}(f)$ be the set of $x$ in $\hhs$ with $x \sim f(x)$.
Considering the action of $f$ on the orbit space $\oo$, let $B(f)$ the
set of $u$ in $\oo$, fixed by $f$.

\begin{lemma}{}{}
Let $\Phi$ be a pseudo-Anosov flow and $f$ in $\pi_1(M)$. 
Then $Fix^{\sim}{f}$ is a closed subset of $\hhs$.
\end{lemma}

\begin{proof}{}
Let $x$ not in $Fix^{\sim}(f)$, so $x \not \sim f(x)$. 
Then $x$ and $f(x)$ have disjoint open neighborhoods $U, V$ in $\hhs$.
By continuity of $f$, there is a smaller open neighborhood $W$ of $x$ so that
$f(W)$ is contained in $V$. Hence any $y$ in $W$ satisfies $y \not \sim f(y)$
and $(Fix^{\sim}(f))^c$ is open.
\end{proof}

\noindent
{\bf {Remark}} $-$ In general $Fix(f)$ is not closed $-$ a sequence
$(x_n)$ in $Fix(f)$ may converge to $x$ which is only in $Fix^{\sim}(f)$.

The following will be useful later:

\begin{lemma}{}{}
If $f$ is in $\pi_1(M)$ and $f$ is not the identity, then 
$Fix^{\sim}(f)$ is countable.
\label{count}
\end{lemma}

\begin{proof}{}
First we show that $Fix(f)$ is countable.
Let $L$ in $\wls$ with $f(L) = L$.
Then there is a periodic orbit in $\pi(L)$. If $L_1, L_2$
are in $Fix(f)$ then their periodic orbits are connected
by a chain of lozenges by theorem \ref{chain}.
In addition the orbit space $\oo \cong \rrrr^2$ is countably compact.
If $Fix(f)$ were uncountable, then $B(f)$ would be uncountable
and there would be accumulation points in $B(f)$. This is disallowed
because any two points in $Fix(f)$ are connected by a chain of lozenges.

Now let $N = \{ x \in \hhs, \ {\rm so \ that} \ x \ {\rm is \ non \
separated \ from \ some} \ y \in \hhs \}$.
We will prove that $N$ is countable, hence $Fix^{\sim}(f)$ is countable.
Assume by way of contradiction that $N$ is uncountable.
The space $\hhs$ is a union of countably many open segments and
we fix one such countable collection.
For each $x$ in $\hhs$, let $I_x$ be one such segment in the countable
family containing $x$.
If $N$ is uncountable, then there is an open segment $I$ in 
$\hhs$ containing
uncountably many elements of $N$.
Choose an order in $I$.
For each $z$ in $I \cap N$, there is 
$y$ distinct from $z$ with $y \sim z$.
Suppose wlog that for uncountably many such $z$ the corresponding $y$
is non separated from the $z$ in their positive sides, with respect to
the order in $I$.
For any such $z, z'$ in $I \cap N$,
let $y, y'$ be non separated from  them respectively.
 We claim that $I_y, I_{y'}$
are different. Suppose for simplicity that $z < z'$ in $I$.
Here $z' \sim y'$ and non separated on their positive sides,
so $I_{y'}$ does not contain $z'$ or any point in $I$ smaller
than $z'$. But by construction $I_y$ contains $y$, so $I_y, I_{y'}$ are
different.
Hence all such $I_y$ are different, contradicting the fact that
there are only countably many of these.
This finishes the proof of the lemma.
\end{proof}

\section{Pseudo-Anosov flows in Seifert fibered spaces}
\label{seifconj}

This section is devoted to proving the following result:

\begin{theorem}{}{}
If $\Phi$ is a pseudo-Anosov flow in $M^3$ which is a Seifert
fibered space, then up to finite covers, $\Phi$ is topologically
equivalent to a geodesic flow on a closed hyperbolic surface.
\label{Seifert}
\end{theorem}

\begin{proof}{}
The new mathematical result is a reduction of the
proof to the non singular
case. The smooth Anosov case was originally proved 
in \cite{Ba1}.
We also give an improved proof of the Anosov case,
which may be useful in other contexts.

If necessary lift to a double cover so that the Seifert fibration
is orientable, hence 
the center of $\pi_1(M)$ is non-empty 
(it contains for example the homotopy
class of the regular fibers). 
Let $h$ be in the center of $\pi_1(M)$.
The cyclic subgroup $< h >$ is a normal subgroup of $\pi_1(M)$.
The proof splits in two cases, depending on whether $Fix^{\sim}(h)$ is
empty or not.

\vskip .1in
\noindent
{\bf {Case 1}} $-$ $Fix^{\sim}(h)$ is non empty.

We show that this cannot happen.
Notice that if $x \sim y$ in $\hhs$ and $g$ is in $\pi_1(M)$ then
$g(x) \sim g(y)$.
Let $g$ in $\pi_1(M)$ and $x$ in $Fix^{\sim}(h)$. Then 
$g^{-1} h g(x) = h(x) \sim x$, so 
$hg(x) \sim g(x)$ and $g(x)$ is in $Fix^{\sim}(h)$.
By lemma \ref{count} \ \   $Fix^{\sim}(h)$ is countable.
Therefore $Fix^{\sim}(h)$ is a countable, closed, $\pi_1(M)$
invariant subset of $\hhs$.
Consider the union $Z$ of the leaves $L$ in $\wls$ with
$\nu_s(L)$ in $Fix^{\sim}(h)$. 
This set $Z$ is closed, $\wls$ saturated, $\pi_1(M)$ invariant and transversely
countable.
It projects to a sublamination of $\ls$ which is transversely
countable. Let ${\mathcal L}$ be a minimal sublamination of $\pi(Z)$.
Any sufficiently small  transversal to a minimal lamination
intersects it in either a closed interval, a Cantor set or a point.
The first two are disallowed by the transverse countability
condition. The last option implies that there is an isolated
leaf in $\ls$, which is 
not possible for pseudo-Anosov flows. This shows that case 1 cannot happen.

\vskip .2in
\noindent
{\bf {Case 2}} $-$ $Fix^{\sim}(h)$ is empty.

By theorem \ref{axis}, $h$ has a non empty axis $\ca(h) \ = \ 
\{ x \in \hhs \ | \ h(x) \ {\rm separates} \ x \ {\rm from} \ h^2(x) \}$.
This axis has a linear order where $h$ acts as a translation. 
Clearly, for every $g$ in $\pi_1(M)$:

$$g\ca(h)=\ca(ghg^{-1})=\ca(h)$$

\noindent
hence $\ca(h)$ is $\pi_1(M)$-invariant.

Either $\ca(h)$ is an infinite
segment or a countable union of disjoint closed segments:

$$\ca(h) \ = \ \cup_{i \in {\bf Z}} [x_i, y_i] \ = \ 
\cup_{i \in {\bf Z}} B_i \  \ (*)$$

\noindent
where $y_i \sim x_{i+1}$.
We show that the second option cannot happen.
Suppose by way of contradiction that $\ca(h)$ is of form (*). 
Every $g$ in  $\pi_1(M)$ permutes the components $B_i$,
preserving or reversing the order on the set ${\bf Z}$ of labels.
Hence there is a morphism $\pi_1(M) \to \mbox{Aut}({\bf Z})$, whose kernel 
is the subgroup made of elements $g$ such that $g x_i = x_i$ for all $i$, ie. a
trivial or cyclic normal subgroup. Since $\mbox{Aut}({\bf Z})$ is the diedral group,
containing a cyclic subgroup of index $2$, it follows that $\pi_1(M)$ contains a
finite index subgroup isomorphic to ${\bf Z}$ or ${\bf Z} \oplus {\bf Z}$,
which is not possible for
an irreducible Seifert fibered space without boundary.
We conclude that 
$\ca(h)$ cannot be an infinite collection of disjoint closed segments.

\vskip .08in
Therefore $\ca(h)$ is a real line parametrized as
$\ca(h) = \{ l_t, t \in \rrrr \}$. 
If $\ca(h)$ is not properly embedded in $\hhs$,
then $(l_t)$ converges to a point $x$ in
$\hhs$ as $t$ converges to infinity
(and maybe other points as well).
But then since $\ca(h)$ is invariant under $h$, this implies that
$h(x) \sim x$, which is not allowed in Case 2.

Next we show that $\ca(h)$ is all of $\hhs$. Again suppose it
is not and let $l$ be a point of $\hhs$ not in $\ca(h)$.
Since $\ca(h)$ is connected (as it is a line), then $\ca(h)$
is contained in a single component of $\hhs - \{ l \}$.
Let $B$ be another component
of $\hhs - \{ l \}$. Let $L = \nu^{-1}_s(l)$.
It was proved in \cite{Fe7} that any complementary component of $L$
covers $M$. This implies that given $x$ in $\ca(h)$, there is 
$g$ in $\pi_1(M)$ with $g(x)$ in $B$, which is disjoint from
$\ca(h)$.
This contradicts the $\pi_1(M)$ invariance of $\ca(h)$.

We conclude that $\hhs$ is homeomorphic to $\rrrr$ and 
similarly $\hhu$ is also homeomorphic to $\rrrr$. Therefore 
there are no singularities of $\Phi$ and 
$\ls, \lu$ are $\rrrr$-covered.

\vskip .15in
Since there is no singularity, the flow is actually (topologically)
Anosov. The result was then proved in \cite{Ba1}.
We present a different proof here, which improves arguments in \cite{Ba1}  and which
follows arguments in the unpublished reference \cite{Ba7}.

If there is a leaf of $\wls$ intersecting all leaves
of $\wlu$, then proposition \ref{susp} shows that 
$\Phi$ is 
a product pseudo-Anosov flow. The manifold then
would have solv geometry and could not be  Seifert fibered,
contradiction.

It follows from \cite{Fe1,Ba1} that $\Phi$ has the skewed type:
the orbit space $\oo$ is homeomorphic to an infinite strip in 
$\rrrr^2$ bounded by parallel lines, say with slope one. 
The stable foliation is the foliation by
horizontal segments and the unstable foliation is the
foliation by vertical segments (see figure~\ref{band}).

\begin{figure}
\centeredepsfbox{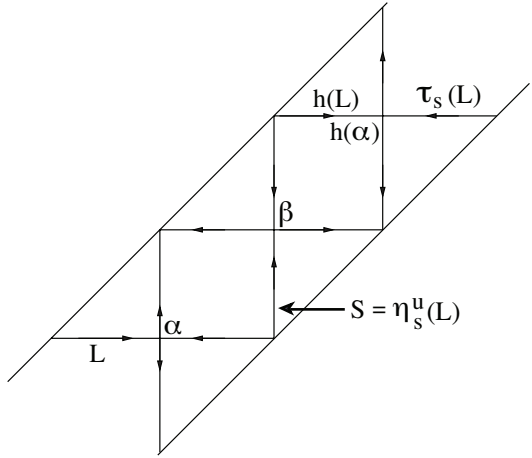}
\caption{Orbit space of skewed type.}
\label{band}
\end{figure}


Put a transverse orientation to $\wls$ positive with increasing $y$
and to $\wlu$ positive with increasing $x$ (where $(x,y)$ are the cartesian
coordinates in $\oo$).
For each stable leaf $L$, there is in the positive side of $L$
a unique unstable leaf $S$ which makes a perfect fit with $L$ $-$
in this model it is equivalent to $S$ sharing an endpoint with $L$.
This produces a $\pi_1(M)$ equivariant map $\eta^u_s$ 
from $\hhs$ to $\hhu$, which is a homeomorphism (\cite{Ba1,Fe1}).
Similarly for each $S$ in $\wlu$ there is a unique $E$ of $\wls$
in the positive side of $S$ and sharing an endpoint with $S$.
The composition $L \rightarrow S \rightarrow E$ is a translation
$\tau_s$ in $\hhs$ and $\hhs/\tau_s$ is a circle
$S^1_s$.
Similarly one has $\tau_u$ which is increasing from 
$\hhu$ to $\hhu$  and a circle $S^1_u = \hhu/\tau_u$.
Both $\tau_s$ and $\tau_u$ are $\pi_1(M)$ equivariant 
homeomorphisms (\cite{Ba1,Fe1}), so $\pi_1(M)$
acts on $S^1_s$ and $S^1_u$.
We denote the first action by

$$\xi_s: \ \pi_1(M) \ \ \rightarrow \ \ {\rm \operatorname{Homeo}}(S^1_s).$$

In addition
the map $\eta^u_s: L \rightarrow S$ as above is also equivariant by the action
of $\pi_1(M)$ and hence induces a canonical homeomorphism from
$S^1_s$ to $S^1_u$ with inverse denoted by $\zeta$.
So we can identify $S^1_s \times S^1_u$ with $S^1_s \times S^1_s$ 
by $(z,w) \rightarrow (z,\zeta(w))$. This induces an 
action of $\pi_1(M)$ on $S^1_s \times S^1_s$.

For every orbit $\beta$ of $\wwp$, there are unique leaves
$L$ of $\wls$ and $G$ of $\wlu$ so that $\beta = L \cap G$.
Using $L$ and $G$, the
orbit $\beta$ generates a point in $S^1_s \times S^1_u$
and hence a point $(p,q)$ in $S^1_s \times S^1_s$. 
We say that $\beta$ projects to $(p,q)$.
This defines a map 

$$\nu: \ \oo \ \ \rightarrow \ \ S^1_s \times S^1_s.$$

\noindent
The projection $(p,q)$ is not in the 
diagonal $\Delta$: points in the diagonal correspond to
$L$ in $\wls$ and $S$ in $\wlu$ so that 
$S = \eta^u_s (\tau_s)^n (L)$ for some integer $n$. 
In particular $L$ and $S$ do not intersect and neither does
$S$ intersect $(\tau_s)^m(L)$ for any integer $m$.
Conversely if $(p,q)$ is in $S^1_s \times S^1_s - \Delta$, then 
one can lift $p$ to a leaf $L$ of $\wls$ and $q$ lifts
to a stable leaf, which after the identification
$S^1_s$ with $S^1_u$ produces $S$ in $\wlu$ with $S \cap L$ not empty.

Note that if $g$ acts trivially on $\hp^s$ then $g$ is the identity
in $\pi_1(M)$. This follows for instance  because the set of fixed
points of non trivial elements of $\pi_1(M)$ is discrete in $\hp^s$.

\vskip .1in
\noindent
{\bf {Claim 1}} $-$ $h$ acts trivially on $S^1_s$.

Let $\widetilde \alpha$ be a lift of a periodic orbit $\alpha$
associated to
a covering translation $g$. 
Then
$g^2 h(\widetilde \alpha) = h g^2 (\widetilde \alpha) = h(\widetilde \alpha)$,
so $\widetilde \alpha$ and $h(\widetilde \alpha)$ are connected
by a chain of $n$ lozenges by Theorem \ref{chain}.
Replacing $g^2$ by $g^{-2}$ if necessary, we can assume that
$\widetilde \alpha$ is an attracting fixed point of the restriction of $g^2$ to
the stable leaf $L$ through $\widetilde \alpha$. Then $h(\widetilde \alpha)$
is also an attractive fixed point of the restriction of $g^2$ to $h(L)$.
It follows (see fig.~\ref{band}) that $n$ is even.
In the figure $\beta$ is connected to $\alpha$ by one lozenge
and $h(\alpha)$ is connected to $\alpha$ by a chain
of 2 lozenges.
Therefore $h(L) = (\tau_s)^i(L)$ for $i = n/2$.

This implies that the projections 
to $S^1_s$ of periodic leaves are fixed points of $\xi_s(h)$.
Since periodic leaves are dense, we conclude that $\xi_s(h)$ is the identity map
on $S^1_s$. The claim is proved.


\vskip .1in
Recall that $h$ was any element of the center of $\pi_1(M)$.
Here $\pi_1(M)$  cannot be ${\bf Z}^3$ because $M$ has a pseudo-Anosov
flow.
It follows that the center of $\pi_1(M)$ is a cyclic subgroup \cite{He,Ja-Sh}. 
From now, we assume that $h$ generates
the center; and we denote by $l$ the integer such that 
when acting on $\hhs$, then $\tau_s^l=h$. In order to simplify the presentation,
we identify in the sequel $\hhs$ with $\rrrr$ in
a way that $\tau_s$ is the translation
$x \mapsto x+1$.

Let now $f$ in the kernel of $\xi_s$. When acting on $\hhs$,
$f(x) = x + j$ for some $j$ in ${\bf Z}$.
In addition given any $g$ in $\pi_1(M)$ and considering the action on $\hhs$,
it follows that for any $x$ in $\hhs$, for any $i$ in ${\bf Z}$,
then $g(x + i) = g(x) + i$. 
Now, for any $g$ in $\pi_1(M)$, again when considering the action on
$\hhs$ we have

$$g^{-1} f^{-1} g f (x) \ = \ g^{-1} f^{-1} g(x + j)  \ = \ 
g^{-1} f^{-1} (g(x) + j) \ = \ g^{-1} g(x) \ = \ x.$$

\noindent
Therefore $g^{-1} f^{-1} g f$ acts trivially on $\hhs$ and
is the identity in $\pi_1(M)$. 
Hence $f$ is in the center of $\pi_1(M)$ which
is $<h>$.

\vskip .05in
\noindent
{\bf {Conclusion}}: 
$\operatorname{ker} \xi_s \ = \ < h > \ = \ \ {\rm center \ of } \ \ \pi_1(M)$.
\vskip .05in

Let $H = < h >$ and 
$Q = \pi_1(M)/H$.
Since $H$ is the kernel of $\xi_s$, there is an induced 
action $\overline \xi_s$ of $Q$ on $S^1_s$.
Given $g$ in $\pi_1(M)$ let $\overline g$ be its image in $Q$.
By the conclusion above the action $\overline \xi_s$ is faithful.


We now think of $S^1_s$ as the ideal boundary of the hyperbolic disc
${\bf H}^2$ and $(p,q)$ as the hyperbolic geodesic in ${\bf H}^2$
connecting these endpoints.

\begin{lemma}{}{}
The action $\overline \xi_s$ of $Q$ on $S^1_s$ is a convergence group action.
\end{lemma}

\begin{proof}{} First we prove the following fact:

\vskip .1in
\noindent
{\bf {Claim 2}} $-$
Two arbitrary orbits $\beta_1, \beta_2$ of $\wwp$ are connected by
a chain of lozenges if and only if $\beta_1, \beta_2$ project
to either the same point of $S^1_s \times S^1_s - \Delta$ or
one projects to some point
$(p,q)$ and the other projects to $(q,p)$.
In the first case they are connected by an even number of lozenges
and in the second case they are connected by an odd number 
of lozenges.

Suppose first that $\beta_1, \beta_2$ are connected by a chain
of lozenges.
The first lozenge in the chain has a stable side $L$ containing
$\beta_1$.
There is an   unstable side
$S$ of the lozenge making a perfect fit with $L$. 
The other corner $\beta$ of the lozenge is contained in 
$S$. 
Suppose wlog that $S$ is in the positive side of $L$.
Then $S = \eta^u_s(L)$. In addition $\wu(\beta_1), \ws(\beta)$
also make a perfect fit and 

$$\wu(\beta_1)  \ \ = \ \ 
\eta^u_s (\tau_s^{-1}(\ws(\beta)).$$

\noindent
So if $\beta_1$ projects to $(p,q)$ then $\beta$ projects
to $(q,p)$. 
Following the lozenges in the chain proves that $\beta_2$ projects
to either $(p,q)$ or $(q,p)$.
Using these arguments one sees that 
$\beta_1$ and $\alpha = \tau_s(L)
\cap \tau_u(\widetilde W^u(\beta_1))$ are connected by a chain
of two lozenges.

Conversely suppose that $\beta_1$ and $\beta_2$ both project
to $(p,q)$. Let $F = \widetilde W^s(\beta_1), \ G = \widetilde W^u(\beta_1)$ 
and let also
$E = \widetilde W^s(\beta_2), \ S = \widetilde W^u(\beta_2)$.
Since the projections of both $\beta_1$ and $\beta_2$ have
the same point $p$ as first coordinate, there is $n$ in ${\bf Z}$
so that $E = \tau^n_s(F)$. Similarly there is $m$ in ${\bf Z}$ with
$S = \tau^m_u(G)$.
In the collection $\{ \tau^i_u(G), i \in {\bf Z} \}$, there is only
one element intersecting $\tau^n_s(F)$ and that is $\tau^n_u(G)$.
It follows that $n = m$. In addition

$$\beta_2 \ \ = \ \ \tau^n_s(F) \cap \tau^n_u(G).$$

\noindent
As explained above $\beta_1$ and $\tau_s(F) \cap \tau_u(G)$
are connected by a chain of two lozenges and by induction
$\beta_1$ and $\beta_2$ are connected by a chain with an
even number of lozenges.
The case that $\beta_1$ projects to $(p,q)$ and
$\beta_2$ projects to $(q,p)$ 
is very similar and is left to the reader.
This proves claim 2.
\vskip .1in

Let $\alpha$ be an arbitrary  closed orbit of $\Phi$, 
let $\widetilde \alpha$ be a lift to $\mi$, which is invariant
under $g$ in $\pi_1(M)$, with $g$ associated to $\alpha$ in
the positive direction.
Let $(p,q)$ in $S^1_s \times S^1_s - \Delta$ 
be $\nu(\widetilde \alpha)$.
Recall that $h$ in $\pi_1(M)$ represents the fiber of the Seifert
fibration.
Since $h$ acts trivially on $S^1_s$, then  claim 2 implies that
$\widetilde \alpha$ and $h(\widetilde \alpha)$ are
connected by a chain of lozenges with an even number of lozenges
\cite{Fe1}.
Therefore the set of 
orbits in the complete chain of lozenges
from $\widetilde \alpha$ 
is finite modulo the action by $< h >$ and this set
projects to a finite set $V$ of orbits of $\Phi$
in $M$. But $\alpha$ is closed, so $V$ is a finite set of
closed orbits and hence discrete in $M$. Hence 
$\pi^{-1}(V)$ is a discrete, 
$\pi_1(M)$ invariant set of orbits of $\wwp$.
We conclude that
$\nu(\Theta(\pi^{-1}(V)))$ is a discrete set
in $S^1_s \times S^1_s - \Delta$.
It is also $\pi_1(M)$ invariant.
This is the ``orbit'' of $(p,q)$ under the action of $\pi_1(M)$.

Now given $\alpha, \widetilde \alpha, g$ as above,
let $L = \widetilde W^s(\widetilde \alpha)$.
Then $g(L) = L$ and since $g$ is associated to the
positive direction of $\alpha$ then $L$ is a 
contracting fixed point of $g$ acting on $\hhs$.
In the same way $S = \widetilde W^u(\alpha)$ is also
fixed by $g$ and it is a repelling fixed point
of $g$ acting on $\hhu$ and hence $p$ is
the attracting fixed point of $g$ acting on 
$S^1_s$ and $q$ is the repelling fixed point.
There are no other fixed points.

\vskip .08in
In order to prove the convergence group property
for the action $\overline \xi_s$ of $Q$ on $S^1_s$, we now
consider a sequence $b_n$ of distinct elements of $Q$ and
let $g_n$ in $\pi_1(M)$ with $b_n = \overline g_n$.
In the arguments below we abuse notation and also denote by
$\xi_s$ the action of $\pi_1(M)$ on
$S^1_s \times S^1_s$ $-$ the context makes clear which one is
being used.

\vskip .05in
Consider a closed orbit $\alpha$ as above, with a given
lift $\widetilde \alpha$,  corresponding points $p, q$ in
$S^1_s$ and $L = \ws(\widetilde \alpha)$.
Suppose first that
up to subsequence 

$$\xi_s(g_n)((p,q)) \ = \ (p,q) \ \ \ {\rm or} \ \ \ 
\xi_s(g_n)((p,q)) \ = \  (q,p) \ \ \ {\rm for \ all} \ \ n.$$

\noindent
Notice that it does not matter if we consider 
$\xi_s(g_n)$ or $\overline \xi_s (\overline g_n)$.
First a reduction: if $\xi_s(g_n)((p,q)) = (q,p)$ for all $n$, then
replace $\widetilde \alpha$ by $g_1(\widetilde \alpha)$ and 
$g_n$ by $g_n g^{-1}_1$. The new collection satisfies
$\xi_s(g_n)((p,q)) = (p,q)$ for all $n$.
Claim 2 implies that
for every $n$, $g_n(\widetilde \alpha)$ is connected
to $\widetilde \alpha$ by a chain of lozenges, with
an even number of lozenges. 
For each $n$ there is $a_n$ so that
$g_n(L) \ = \  \tau^{a_n}_s(L)$.
Recall the integer $l$ above so that $h = \tau^l_s$ when acting
on $\hhs$.
There are $b_n$ and $c_n$ in ${\bf Z}$ with 
$0 \leq c_n < l$ and $a_n = b_n l + c_n$.
Up to another subsequence we assume that $c_n$ is
constant. Again up to taking $g_1(\widetilde \alpha)$ instead
of $\widetilde \alpha$ and $g_n g^{-1}_1$ instead
of $g_n$ we may assume that $c_n = 0$ for all $n$.
The above facts imply that
for each $n$ there is $i_n$ in
${\bf Z}$ so that $h^{i_n} g_n(\widetilde \alpha) 
= \widetilde \alpha$
(in fact $i_n = -b_n$).
Therefore $h^{i_n} g_n = f^{j_n}$, for some
$j_n$ in ${\bf Z}$ where $f$ is
a generator of the isotropy group of $\widetilde \alpha$
in the forward direction.
Notice that $\xi_s(h^{i_n})$ acts as the identity on $S^1_s$
(and also on $S^1_s \times S^1_s - \Delta$).
If there is a subsequence $(j_{n_k})$ which is constant, then
the formula

$$g_{n_k} \ \ = \ \ h^{-i_{n_k}} f^{j_{n_k}}$$

\noindent
shows that all $\xi_s(g_{n_k})$
act in exactly the same way
on $S^1_s$.
Then $\overline \xi_s (\overline g_{n_k})$ is constant
and since $\overline \xi_s$ is faithful, then the sequence
$(\overline g_{n_k})$ is also constant 
$-$ contradiction to hypothesis.
So up to subsequence we may assume (say) that $j_n$ converges to infinity 
(as opposed to converging to minus infinity) when 
$n \rightarrow \infty$.
Then 

$$\xi_s(g_n) \ = \ \xi_s(h^{-i_n} f^{j_n}) \ = \ \xi_s(f^{j_n})$$

\noindent
and $p$ is the
sink for the sequence $\xi_s(g_n)$
acting on $S^1_s$ and
$q$ is the source.
This proves the convergence group property in this case.

\vskip .1in
From now on we assume up to subsequence that $\xi_s(g_n)((p,q)) \not = (p,q),
(q,p)$ for
all $n$. 
In fact by the same arguments we can assume that all $\xi_s(g_n)((p,q))$ are distinct.
Since the orbit of $(p,q)$ under $\pi_1(M)$ is discrete in 
$S^1_s \times S^1_s - \Delta$,
then up to subsequence $\xi_s(g_n)((p,q))$ converges to a point $(z,z)$ in 
$S^1_s \times S^1_s$.
These arguments work for {\underline {any}}
closed orbit $\alpha$.

\vskip .1in
We now show that $\xi_s(g_n)$ has a subsequence with the source/sink behavior.
Fix an identification of $S^1_s$ with the unit circle ${\bf S}^1$.
Since $\Phi$ is $\rrrr$-covered,
then the set of closed orbits is dense \cite{Ba1}.
Find $(p_1,q_1)$  corresponding to a periodic orbit,
very close to $(-1,1)$ and not disconnecting
these two points in ${\bf S}^1$.
By the above arguments,
up to subsequence $\xi_s(g_n)((p_1,q_1))$ converges to a single
point $(z,z)$ in ${\bf S}^1 \times {\bf S}^1$. Therefore one interval $I_1$ of
${\bf S}^1$ defined by 
$(p_1,q_1)$ converges to $z$ under $\xi_s(g_n)$.
The interval $I_1$ has length close to half the length of the circle
${\bf S}^1$. We work by induction assuming that
an interval $I_i$ has been produced.
Let $J_i$ be the closed complementary interval 
to $I_i$. Find a periodic
point $(p_i,q_i)$ so that:
$q_i$ is in $J_i$ and almost cuts it in half and $p_i$ is
in the interior of $I_i$ (switch $p_i$ and $q_i$ if necessary).
We already know that $\xi_s(g_n)(p_i))$ converges to $z$.
As before up to another subsequence
one of the intervals defined by $(p_i, q_i)$ converges to
a point under $\xi_s(g_n)$, which then must be $z$ as $p_i$ is in $I_i$.
Adjoin this interval to $I_i$ to produce $I_{i+1}$ which
converges to $z$ under $\xi_s(g_n)$. Let $J_{i+1}$ be the 
closed complementary interval.
Since each step roughly reduces the
size of the remaining interval by a factor of $1/2$, then the intervals
$J_i$ converge to a single point $w$. Use a diagonal process
and obtain a sequence $\xi_s(g_{n_k})$ with source $w$ and sink $z$.
This finishes the proof of the convergence group property.

Notice that as we mentioned before, we denoted by $\xi_s$
the action on both $S^1_s$ and $S^1_s \times S^1_s - \Delta$.
\end{proof}

\noindent
{\underline {Convention}} $-$ We lift to a double cover if necessary
so that $\wls$ is transversely orientable.
Every orientation preserving
convergence group acting on the circle is equivalent in
$\operatorname{Homeo}^+({\bf S}^1)$ to a Fuchsian group \cite{Ga,Ca-Ju}.
Let $\Gamma$ be $\overline \xi_s(Q)$.
Hence $\Gamma$ is equivalent to a Fuchsian group $T$.
Here $O = {\bf H}^2 / T$ is a hyperbolic 2-dimensional 
orbifold.

We have a  conjugation $\psi: S^1_s \rightarrow {\bf S}^1$ 
between the action of $\Gamma$ on $S^1_s$ and a Fuchsian action
$T$ on ${\bf S}^1$. 
Lift $\psi$ to a homeomorphims $\widetilde \psi:  \hhs \rightarrow \rrrr$.
Let $g$ in $\pi_1(M)$ and 
we also think of $g$ as acting on $\hhs$. Then 

$$\psi \circ \overline \xi_s(\overline g) \circ \psi^{-1} \ = \ 
\psi \circ \xi_s(g) \circ \psi^{-1}$$

\noindent
is the ideal map of a Moebius transformation and
hence $\widetilde \psi g (\widetilde \psi)^{-1}$ is a projective transformation
of $\rrrr$.
This shows that the foliation  
$\Lambda^s$ is transversely projective.
As shown by the first author in \cite{Ba1}, this implies
that the flow $\Phi$ is up to a finite cover,
topologically equivalent
 to
a geodesic flow in the unit tangent bundle of a 
hyperbolic surface. 
This finishes the proof of theorem \ref{Seifert}.
\end{proof}



\vskip .1in
\noindent
{\bf {Remark:}} One may ask whether  theorem \ref{Seifert} can
be improved to remove the finite covers condition, perhaps
by considering geodesic flows in unit tangent bundles of
hyperbolic orbifolds. But this is not possible in general,
because one can unwrap the fiber direction. 
We explain this:
suppose $\Phi'$ is the geodesic flow in $T_1 S$, where $S$ 
is a closed hyperbolic orbifold. Let $\alpha$ be a closed 
orbit of $\Phi'$, that is, it comes from 
 a closed geodesic $\gamma$ in $S$, where
for simplicity we assume that $\gamma$ does not pass
through a singularity of $S$.
Then the vectors in $T_1 S$ projecting to $\gamma$ in $S$
form a torus $T'$ in
$T_1 S$ and there are 
{\underline {exactly}} two closed orbits in $T'$ corresponding
to the two directions along $\gamma$.
To get a counterexample start with $M = T_1 R$ where again
for simplicity $R$ is a non singular hyperbolic surface.
Let $M_1$ be a finite cover of $M$ obtained by unwrapping
the fiber direction some number $n$ of times.
Then $M_1$ is Seifert fibered and the geodesic flow in $T_1 R$
lifts to an Anosov flow in $M_1$.
Any torus in $T$ in $M_1$ projects to a torus in $T_1 R$ and this is
homotopic to a torus over a closed geodesic of $R$, but 
traversed $n$ times in the fiber direction. 
This implies that the original torus is homotopic to one
which has $2n$ closed orbits $-$ and hence cannot be a 
torus of the geodesic flow of a hyperbolic orbifold.
Hence the Anosov flow in $M_1$ cannot be topologically
equivalent to a geodesic flow, but is equivalent
to a finite cover of a geodesic flow.

\vskip .2in
\noindent
{\bf {EXAMPLES and COUNTEREXAMPLES}}


Recall that in a one prong pseudo-Anosov flow we allow the existence
of one prongs.
One prong pseudo-Anosov flows can behave completely differently from
pseudo-Anosov flows. In particular it is well known that there are one prong 
pseudo-Anosov flows in ${\bf S}^2 \times {\bf S}^1$, so the manifold $M$
need not be irreducible and the universal cover need not be $\rrrr^3$.

Here we introduce 2 new classes of examples of one prong pseudo-Anosov flows.

\vskip .1in
\noindent
1) Let $R$ be a closed hyperbolic surface with an order 2  symmetry 
$\sigma$ which is an isometric
reflection along a non separating 
simple closed geodesic $\alpha$ of $R$.
Let $M_1$ be the unit tangent bundle of $R$ and $\Phi_1$ be the 
geodesic flow in $M_1$. The isometry $\sigma$ sends geodesics 
of $R$ to geodesics
and preserves the geodesic flow. It induces a map $\sigma_*$ in $M_1$ which 
has order $2$. Let $M$ be the quotient of $M_1$ by the map $\sigma_*$. 
The map $\sigma_*$ 
does not act freely: the fixed points correspond exactly to the tangent
vectors to $\alpha$ $-$ there are two closed orbits $\alpha_1, \alpha_2$ of
$\Phi_1$ which are fixed pointwise by $\sigma_*$. 
These correspond to the 2 directions
in $\alpha$. Hence $M$ is an orbifold, but admitting
a natural manifold structure so that the projection map 
$M_1 \to M$ is an order 2  branched covering map.
The flow $\Phi_1$ induces a flow
$\Phi$ in $M$ because $\sigma$ sends geodesics to geodesics.
The stable/unstable foliations of $\Phi_1$ are invariant
under $\sigma_*$ so induce stable/unstable foliations
of $\Phi$.
The stable leaf of $\Phi_1$ through
$\alpha_1$ folds in two, producing a  one prong singularity of $\Phi$ and
similarly for $\alpha_2$. The flow $\Phi$ is an example of 
a one prong pseudo-Anosov flow.
Alternatively the 
manifold $M$ is obtained as follows: let $R_1, R_2$ be the 
closures of the 2 components of $R - \alpha$. The unit tangent bundle of
$R_1$ is homeomorphic to $R_1 \times {\bf S}^1$, with boundary a torus $Z$
with 2 closed curves corresponding to $\alpha_1$ and $\alpha_2$.
The map $\sigma_*$ identifies one complementary annulus of $\alpha_1, \alpha_2$ in $Z$ to
the other one with no shearing. This is obtained by a Dehn filling of $Z$ where 
$\{ t \} \times {\bf S}^1$ is the meridian.
Therefore $M$ is homeomorphic to the union of 
$N_1 = R_1 \times {\bf S}^1$ and a solid
torus. 
This is almost a graph manifold: it is the union of Seifert fibered spaces,
but $M$ is not irreducible:
Take a non peripheral arc $l$ in $R_1$. Then $l \times {\bf S}^1$ is an annulus
in $R_1 \times {\bf S}^1$ which is 
capped off with 2 discs in the solid torus to produce a sphere which 
is non separating in $M$ and hence clearly does not bound a ball
in $M$.

\vskip .1in
\noindent 
{\bf {Remark}} $-$ This example and the next work whenever the 
hyperbolic surface $R$ admits an isometric reflection along a
collection of simple closed geodesics $\{ \alpha_i \}$. For simplicity
of exposition we describe the examples in 1) and 2) with a single
geodesic $\alpha$.

\vskip .1in
2) The second class of examples is obtained by a modification of example 
1) in order to be
in a Seifert fibered manifold.
The modification is that the glueing of the annuli in 
$\partial N_1$ is done with
a shearing. 
The notation is the same as in example 1): $R$ is the hyperbolic
surface with a geodesic $\alpha$ of symmetry and $R_1, R_2$ the
closures of the components of $R - \alpha$.
The unit tangent bundle of $R$ is $M_1$ and $N_1, N_2$ are
the restrictions to vectors in $R_1$ and $R_2$ respectively.
We use 2 tori: 
$\partial N_1 = T_1$ and $\partial N_2  = T_2$. These are glued
to form $M_1$. Put coordinates
$(\theta_1, \theta_2)$ in $T_1$, $(a_1, a_2)$ in $T_2$ as follows:
$T_1$ consists of the unit vectors along $\alpha$. 
Parametrize $\alpha$ by arc length parameter
$t$ where $0 \leq t \leq l_0$ and $l_0$ is the length of $\alpha$. 
Let $\theta_1 = 2 \pi t/l_0$.
Let $\theta_2$ be the angle between the unit tangent vector 
to $\alpha$ and the vector $v$, where
$\theta_2 = 0$ corresponds to the direction of $\alpha_1$.
Also $\theta_2 = \pi$ corresponds to $\alpha_2$ and 
$0 < \theta_2 < \pi$ are the vectors
exiting $N_1$ and entering $N_2$.
Put coordinates $(a_1,a_2)$ in $T_2$ so that the glueing map
to create $M_1$ is $\eta: T_1 \rightarrow T_2$ given
by $a_1 = \theta_1, a_2 = \theta_2$ 
(essentially the same coordinates). Notice that
vectors with $0 < a_2 < \pi$ are entering $N_2$
and vectors with $\pi < a_2 < 2 \pi$ are entering $N_1$.

In $N_1$ we consider the restriction of the geodesic flow of $R$. 
We collapse
$\partial N_1 = T_1$ to an annulus as follows. Let $A_1$ be the strip
$0 \leq \theta_2 \leq \pi$ in $T_1$ and let $A_2$ be the strip
$\pi \leq \theta_2 \leq 2 \pi$ in $T_1$. We glue $A_1$ to $A_2$ by

$$f(\theta_1, \theta_2) \ =  \ 
(\theta_1 + 2n\theta_2, 2 \pi - \theta_2) \ \ \ \ \ (*)$$

\noindent
Let $M$ be the quotient of $N_1$ by this glueing
and let $\Phi$ be the induced flow from the geodesic flow in $N_1$.
Notice that the flow in $N_1$ is outgoing in the interior of
$A_1$ and incoming in the interior of $A_2$.
In addition, the angle between flow lines and $T_1$ depends only
on $\theta_2$ and not on $\theta_1$ (by definition) and so by formula $(*)$
this produces a flow $\Phi$ in $M$ which is smooth outside of the
closed orbits $\alpha_1, \alpha_2$.
Here we abuse notation and continue to call $\alpha_1, \alpha_2$ their
projections to $M$.

Let $A$ be the annulus which is the quotient of $A_1, A_2$ by
the glueing. Let $M_2$ be the double branched cover
of $M$ obtained by double branched cover (opening up)
 along $A$. This $M_2$
can be cut along the torus $T$ which is the preimage
of $A$. The closure of the 2 complementary components of $T$
are homeomorphic to $N_1$ and $N_2$ and still denoted
by $N_1, N_2$. We think of $N_1$ as the unit tangent
bundle of $R_1$. We can also think of $N_2$ as the
unit tangent bundle of $R_2$ $-$ this is because
$N_2$ under the branched cover is another copy of
$N_1$, which is isometric to $N_2$ 
by the map $\sigma_*$ induced by the symmetry
$\sigma$ of the surface $R$. Let $T_1,T_2$ be the
corresponding boundaries of $R_1,R_2$, with the
corresponding coordinates $(\theta_1,\theta_2)$
and $(a_1,a_2)$ as above. Therefore $M_2$ is
obtained by a certain glueing of map $g$ from $T_1$ to $T_2$.

We first extend the map $f$ to an involution on the entire torus $T_1$:
in $A_2$ (which is the region $\pi \leq \theta_2
\leq 2\pi$), the map $f$ has the same formula
$f(\theta_1,\theta_2) = (\theta_1 + 2n\theta_2, 2\pi-\theta_2)$.
Clearly $f$ is an involution in $T_1$.

\vskip .1in
\noindent
{\bf {Claim}} $-$ In order to obtain the flow 
$\Phi$ in $M$, the glueing from $T_1$ to $T_2$ in
the $(\theta_1,\theta_2)$, $(a_1,a_2)$ coordinates is given by:

$$g: T_1 \rightarrow T_2, \ \ \ \ 
g(\theta_1, \theta_2) = (\theta_1 + 2n \theta_2, \theta_2).$$

In order to prove the claim we
need to show that when restricted to
the annulus $A_1$ then $f = \sigma_* g$.
Recall that $\sigma_*$ restricted to $T_2$ (which is identified
with $T$) has the form
$\sigma_*: T_2 \rightarrow T_1$, 
$\sigma_*(a_1,a_2) = (a_1, 2\pi - a_2)$.
It is now clear that  $f = \sigma_* g$ in $A_1$. 
By the extension of $f$ to $A_2$, this also holds
in $A_2$.
This proves the claim.

\vskip .1in
Let $\Phi_2$ be the lift of the flow $\Phi$ to $M_2$.
This flow $\Phi_2$ is the geodesic flow in $R_1$ when restricted
to $N_1$ and the the geodesic flow
of $R_2$ when restricted to $N_2$. 
The glueing is given by the map $g$ described
above.
The map $g$ is a shearing. In a very nice result, Handel and Thurston
\cite{Ha-Th} studied exactly this example and they 
proved that the flow $\Phi_2$ in $M_2$  is an Anosov flow which is volume
preserving.
Therefore this flow has stable and unstable foliations which
project to stable/unstable foliations of $\Phi$: this is because
if 2 orbits in $M_2$ are asymptotic then their projections
to $M$ are asymptotic and vice versa. The projection
from $M_2$ to $M$ is locally injective and smooth except along $\alpha_1$ and
$\alpha_2$, where it is 2 to 1. Hence the stable/unstable foliations in 
$M$ are non singular except possibly at $\alpha_1, \alpha_2$.
Since the projection is 2 to 1 and stable leaves go to stable leaves,
then along the stable leaf of $\alpha_1$ the stable leaf folds
in two and similarly for the unstable leaf and likewise 
for $\alpha_2$. Therefore $\Phi$ is smooth everywhere
except at $\alpha_1, \alpha_2$ which are one prong
singularities. We conclude that $\Phi$ is a one prong
pseudo-Anosov flow.


Finally $M$ can be thought as a Dehn filling of $N_1$ along
$T_1$. We determine the new meridian.
%
Under the map $f$ from $A_1$ to $A_2$, 
the segment $\theta_1 = 0, \ 0 \leq \theta_2 \leq \pi$ \ in $A_1$ is
glued to the the segment
$(2n \theta_2, 2 \pi - \theta_2), \ 0 \leq \theta_2 \leq \pi$ \ in $A_2$. 
This last segment goes from $(0,2 \pi)$ to $(2n \pi, \pi)$ linearly.
It follows that this is the new meridian which is then the 
$(-n,1)$ curve. 

When
$n = 0$, this is exactly the same construction as in the first example
which makes the fiber in $N_1$ null homotopic. When $n\not=0$,
the curve which becomes null homotopic
is not $\{ p \} \times {\bf S}^1$.
It follows that the resulting manifold $M$ is Seifert fibered.


\vskip .1in
\noindent
{\bf {Conclusion}} $-$ 
If one allows $1$-prongs, then Seifert fibered manifolds
can admit one prong pseudo-Anosov flows with singularities as opposed
to what happens  with pseudo-Anosov flows.
Theorem A does not hold for one prong pseudo-Anosov flows.


\vskip .1in
 This poses the
following questions: Suppose that $\Phi$ is a one prong
pseudo-Anosov flow in $M$ Seifert fibered (closed). 
Can one show that there are no $p$-prongs with $p \geq 3$?
Can one show that $\Phi$ has a branched cover to an Anosov flow
in a Seifert manifold?

\vskip .1in
\noindent
{\bf {Remark:}} 
With this description of geodesic flows we now mention
the following, which will be extremely useful later on
in the article.
Here is an explicit example of a Klein bottle in a manifold with
an Anosov flow. Let $\Phi$ be the geodesic flow of a nonorientable hyperbolic
surface $S$ and $\alpha$ an orientation reversing simple geodesic.
Let $A$ be the unit tangent bundle of $\alpha$ and $\alpha_1, \alpha_2$,
the two orbits of $\Phi$ associated to the two directions of $\alpha$.
Consider tubular neighborhoods of 
of $\alpha_1$, $\alpha_2$. These are solid tori,
and $A$ in these neighborhoods
wraps around each of these periodic orbits twice producing a
M\"obius band, which contains the periodic orbit, and with boundary
a closed curve homotopic to the double of the periodic orbit. 
It follows that the closure
of $A$ is the union of an annulus (outside the solid tori) and two M\"obius strips
and therefore $A$ 
is a Klein bottle. This is a typical example of {\em Birkhoff-Klein bottle},
see formal definition in section \ref{sec:birk}.
A tubular neighborhood of this Klein bottle if homeomorphic to the twisted
line bundle over the Klein bottle. 

\section{Pseudo-Anosov flows in manifolds with virtually solvable fundamental group}

In this section we first do a detailed analysis of maximal subgroups of 
$\pi_1(M)$ stabilizing a given chain of lozenges.
Conversely given a subgroup of $\pi_1(M)$ isomorphic to ${\bf Z}^2$ we 
analyse the uniqueness of chains of lozenges invariant under this subgroup.
These results are foundational for understanding
 any ${\bf Z}^2$ subgroup of $\pi_1(M)$
and they are fundamental for the analysis of pseudo-Anosov flows
in manifolds with virtually solvable fundamental groups.
The results are later used for other results in this article.
We also expect that these
results will be useful for further study of pseudo-Anosov flows
in toroidal manifolds.

In this section let $K$ denote the Klein bottle.
We first need a result from $3$-dim topology.
Let $F$ be a compact surface with a free involution $\tau$.
Then $M = (F \times I)/(x,t) \sim (\tau(x),1-t)$ is 
a {\em twisted $I$-bundle} over the surface
$F' = F/x \sim \tau(x)$ and $F$ is the associated $0$-sphere
bundle, see \cite{He}, page 97.

\begin{lemma}{}{}
Let $N$ be an irreducible, compact $3$-manifold with finitely generated
fundamental group which is torsion free and has a finite index subgroup 
isomorphic to ${\bf Z}^2$. Then $N$ is either an I-bundle or a twisted I-bundle 
over a surface of zero Euler characteristic. In particular 
$\pi_1(N)$ is isomorphic to 
either ${\bf Z}^2$ or $\pi_1(K)$.
In addition if $N$ is orientable, then either $N = T^2 \times I$ or
$N = (T^2 \times I)/(x,t) \sim (\tau(x),1-t)$ is a twisted $I$-bundle
over the Klein bottle $T^2/x \sim \tau(x)$ which is one sided in $N$.
\label{fund}
\end{lemma}

\begin{proof}{}
Suppose first that $N$ is closed. Then take a finite cover $N'$
with $\pi_1(N')  = {\bf Z}^2$. Since the finite cover is irreducible, this is not
possible \cite{He}. Hence $\partial N$ is not empty.
Suppose that boundary of $N$ is compressible. By the loop theorem
\cite{He} there is a curve in $\partial N$, not null homotopic in 
$\partial N$,
but bounding an embedded disc $D$ in $N$. Cutting along D, shows that 
$\pi_1(N)$ is either a free product or an amalgamated free product along
a trivial group, hence a free product with ${\bf Z}$.
In either case the free product would either not contain a ${\bf Z}^2$ (it would
be infinite cyclic) or would contain a free group of rank $\geq 2$, in which case 
it could not contain ${\bf Z}^2$ with finite index.
	Hence $\partial N$ is incompressible. If it has a component of genus
$\geq 2$ then as above it would have a rank $2$ free subgroup, again contradiction.
If it has a component which is a projective plane, then $\pi_1(N)$ has elements
of order $2$, contrary to hypothesis. Since $N$ is irreducible, no component
of $\partial N$ is a sphere, as $\pi_1(N)$ is not trivial. We conclude that
every boundary component of $N$ is either a torus or a Klein bottle.

Let $F$ be one such component. Because $F$ is incompressible and $\pi_1(N)$ has
a finite index subgroup isomorphic to ${\bf Z}^2$, 
then $\pi_1(F)$ has finite index in $\pi_1(N)$.
By theorem 10.5 of \cite{He}, either i) $\pi_1(N) = {\bf Z}$, or
ii) $\pi_1(N) = \pi_1(F)$ with $N \cong F \times I$  or
iii) $\pi_1(F)$ has index $2$ in $\pi_1(N)$ and $N$ is a twisted $I$-bundle over
a compact manifold $F'$, with $F$ the associated $0$-sphere bundle.
In our situation case i) cannot happen. In case ii) $\pi_1(N)$ is either
${\bf Z}^2$ or $\pi_1(K)$ and we are done.
In case iii) $\pi_1(N)$ is isomorphic to $\pi_1(F')$ as there is 
a deformation retract from $N$ to $F'$.
Here $F'$ is a closed surface which has a double cover either the torus
or the Klein bottle. Hence again $F'$ is the torus or the Klein bottle and
we also conclude that $\pi_1(N)$ is either ${\bf Z}^2$ or $\pi_1(K)$.
The last stament is easy given the above.
This finishes the proof of the lemma.
\end{proof}

Note that both the torus and the Klein bottle have double covers
homeomorphic to themselves. The manifolds in question above can be either
orientable or not.
It is easy to construct a compact manifold $N$ which is a twisted $I$-bundle
over the Klein bottle (with quotient surface a Klein bottle). This
manifold has boundary a Klein bottle and an orientation double cover $N_2$
which is a twisted $I$-bundle over the torus (with quotient surface
a Klein bottle, which is one sided in $N_2$). Finally $N$ has
an order $4$ cover homeomorphic to $T^2 \times I$.

\begin{lemma}{}{}
Suppose that ${\mathcal C}$ is a bi-infinite chain of lozenges. Let $G$ be the
stabilizer of ${\mathcal C}$ in $\pi_1(M)$.
Then $G$ is isomorphic to a subgroup of $\pi_1(K)$.
In particular, $G$ is torsion free
and it contains a unique maximal abelian subgroup of index at most 2, which 
is either trivial, 
(infinite) cyclic or  isomorphic to ${\bf Z}^2$.\label{stachain}
\end{lemma}

\begin{proof}{}
The proof will reveal the structure of the stabilizer of $\mathcal C$ and
not just show that it is isomorphic to a subgroup of $\pi_1(K)$.
In this proof cyclic means infinite cyclic.
Let $\alpha$ be a corner in $\mathcal C$.
The chain $\mathcal C$ corresponds to a linear subtree $T_0$ of the tree ${\mathcal G}(\alpha)$.
It defines a homomorphism $\rho: G \to \mbox{Aut}(T_0)$.
%
The kernel $\mathcal K$ of $\rho$ stabilizes every corner
of $\mathcal C$, and thus, is either cyclic or trivial. 

Assume first that $G$ preserves the orientation on $T_0$.
Then $\rho(G)$ is a group of translations along $T_0$, ie. trivial or cyclic. 
In the former case, $G=\mathcal K$ is either trivial or
cyclic. 
In the latter case, if $\mathcal K$ is trivial then 
$G$ is isomorphic to $\rho(G)$ and hence trivial
or cyclic. If $\mathcal K$ is cyclic then
$G$ is an extension of ${\bf Z}$ by ${\bf Z}$.
It is an elementary fact that any such extension splits and hence
$G$ is either ${\bf Z}^2$ or $\pi_1(K)$.
We are done.

Hence from now on assume that some
element $g$ of $G$ reverses the orientation of $T_0$. 
Hence $g$ leaves either a vertex or an edge of
$T_0$ invariant.
Then, according to proposition \ref{pro:treefacts} item 4, $g$ preserves a corner $\alpha$.
That is $g$ does not act as a reversion on an edge.
Let $s$ be a generator of the $G$-stabilizer of $\alpha$ $-$ in particular this
stabilizer is not the identity and is isomorphic to 
${\bf Z}$. Then $s$ reverses the orientation of $T_0$ (otherwise
all elements in $G$ leaving $\alpha$ invariant would preserve orientation)
and $s^2$ is in $\mathcal K$. On the other hand, every element of $\mathcal K$ fixes $\alpha$ and 
preserves the orientation: it
must be a power of $s^2$, which therefore generates $\mathcal K$. 
As before there are two options for $\rho(G)$. One option is that $\rho(G) = \rho(<s>)$ and
therefore $G$ is generated by $s$ and is cyclic.
Otherwise $\rho(G)$ has at least one translation.
Select $h$ in $G$ such that $\rho(h)$ is a translation
along $T_0$ of minimal length. 
In this case it is easy to see that 
$s$, $h$ generate $G$.

By considering the action on the set of vertices of $T_0$ one sees that
$hsh$ preserves $\alpha$. It is also in $G$ so
$hsh = s^i$ where $i$ is odd. 
Similarly $h^{-1} s h^{-1} = s^j$, $j$ odd. 
Now we use $3$-manifold
topology.

Let $G'$ be the subgroup of $G$ preserving the orientation on $T_0$.
The first case of the proof shows that $G'$ has a subgroup of order $\leq 2$ isomorphic
to ${\bf Z}^2$, so $G$ has a subgroup of order $\leq 4$ isomorphic to ${\bf Z}^2$.
(we stress that we want a subgroup of order $2$ isomorphic to ${\bf Z}^2$, so more
work is needed).
Let $U$ be the cover of $M$ associated to $G$.
Then $U$ is irreducible and $\pi_1(U)$ is torsion free, because $\pi_1(M)$ is
torsion free $-$ as its universal cover is homeomorphic to $\rrrr^3$.
By Scott's core theorem \cite{He} there is a compact
core $N$ for  $U$. We can assume that no boundary component of $N$ is a sphere - by
attaching $3$ balls to such components, without affecting  the fundamental
group. Now apply the previous lemma to show that $G = \pi_1(N)$ is isomorphic to either
${\bf Z}^2$ or $\pi_1(K)$. 

Finally if $G$ is not abelian then $G$ is isomorphic to $\pi_1(K)$
and it is an elementary algebra fact that $G$ has a unique maximal
abelian subgroup of index $2$, which is isomorphic to ${\bf Z}^2$.
\end{proof}

\noindent
{\bf {Remark}} $-$ Consider the most complicated case ($G$ contains a ${\bf Z}^2$ 
subgroup). Using graphs of groups one can quickly produce
a short exact sequence $1 \rightarrow {\bf Z} \rightarrow G 
\rightarrow A \rightarrow 1$, where $A$ is either ${\bf Z}$ or $D_{\infty}$,
the infinite dihedral group. The difficulty occurs in the dihedral 
group case: in particular in our situation the exact sequence does
not split, for otherwise $G$ would have torsion. Hence the graph of
groups analysis gets more involved and has many possibilities.
\vskip .05in

%

Conversely:

\begin{lemma}\label{le:Z2}
Let $G$ be a subgroup of $\pi_1(M)$ isomorphic to ${\bf Z}^2$. Assume
that $\Phi$ is not product. Then $G$ preserves a bi-infinite chain of
lozenges. 
\end{lemma}

\begin{proof}
If $G \sim {\bf Z} \oplus {\bf Z}$ acts freely on the orbit space $\oo$, then 
it was proved in \cite{Fe5} that $\Phi$ is product, contrary to hypothesis.
Hence there is $g$ in $G$ with a fixed point in $\oo$.
If $g = (g')^n$ where $g'$ is in $G$ and $|n| > 1$, then
$g'$ also does not act freely on $\oo$ (Proposition \ref{pro:treefacts}, item 1.).
Hence we may assume that $g$ is indivisible in $G$.
Choose $h$ in $G$ so that $h, g$ form a basis of $G$.
Consider the tree ${\mathcal T}={\mathcal G}(g)$: since $G$ is 
abelian, then $G$ acts on ${\mathcal T}$.
If $f$ is an element
of $G$ admitting a fixed point  in $\mathcal T$, then 
some power of $f$ leaves invariant all vertices
of $\mathcal T$ and likewise for $g$.
It follows that
$g$ and $f$
admit a common power: $g^p=f^q$. Since $f, g$ are in $G \cong {\bf Z}^2$ then
$f, g$ generate a cyclic group. But  $g$ is indivisible in $G$, implying 
that $f$ is a power of $g$. Hence, $G/\langle g \rangle \sim {\bf Z}$ is a cyclic
group acting freely on the vertices of the tree $\mathcal T$. 
According to Proposition \ref{pro:treefacts}, item 4., an element
in $G/ <g>$ cannot 
reverse an edge of $\mathcal T$. It follows that $G/\langle g \rangle$ acts freely on
$\mathcal T$, and that there is an invariant axis for this cyclic group therein. 
It provides a bi-infinite
$G$-invariant chain of lozenges $\mathcal C$.
In particular the arguments show that $g$ fixes all the vertices in ${\mathcal C}$.
\end{proof}

\begin{define}\label{def:scallopedregion}{(\cite{Fe5})}
Let $\mathcal C$ be a s-scalloped bi-infinite chain of lozenges. The s-scalloped region
defined by $\mathcal C$ is the union of all lozenges in $\mathcal C$ with the 
half-leaves of $\wlu$
common to two adjacent lozenges in $\mathcal C$. One defines similarly u-scalloped regions.
A scalloped region is a s-scalloped or u-scalloped region; it is an open subset of $\oo$.
\end{define}

It may happen in the situation of lemma \ref{le:Z2} that the $G$-invariant chain
is not unique, but only in a very special situation:

\begin{lemma}\label{le:uniqueC}
Let $G$ be a subgroup of $\pi_1(M)$ isomorphic to ${\bf Z}^2$. Assume that
$G$ preserves two different chains of lozenges. Then, one these chains is
s-scalloped, and the other is u-scalloped. Moreover the associated
u-scalloped and
s-scalloped regions are the same, that is, they are the same
subset of the orbit space $\oo$.
\end{lemma}

\begin{proof}
In this proof we consider all objects in $\oo$.
Let  $\mathcal C$, ${\mathcal C}'$ be two different $G$-invariant 
chains of lozenges. 
First of all notice that the set of vertices of a chain is a linear
set isomorphic to a subset of ${\bf Z}$ and the stabilizer of any
vertex is at most ${\bf Z}$. Since $G \cong {\bf Z}^2$ there is
an element of $G$ acting freely on the set of vertices of ${\mathcal C}$
(or ${\mathcal C}'$). It follows that ${\mathcal C}, {\mathcal C}'$ are
bi-infinite chains of lozenges.

Let $g$
be an element of $G$ fixing every corner 
in $\mathcal C$, and let $f$ be an element of
$G$ fixing every corner of ${\mathcal C}'$ $-$ see proof of the previous lemma. 
Suppose first that $g$ and $f$ share a common non-trivial power:
$g^p=f^q$, $p,q \not= 0$. 
Since $G$ is abelian it acts on ${\mathcal G}(g^p)$ and also ${\mathcal G}(g)
\subset {\mathcal G}(g^p)$, so ${\mathcal C}$ is an invariant axis for
$G$ acting on ${\mathcal G}(g^p)$. Similarly $\mathcal C'$ is 
a $G$-invariant axis in ${\mathcal G}(f^q)$. Since these trees are the same,
it now follows that ${\mathcal C} = {\mathcal C}'$, contradiction.


Hence, replacing $G$ by a finite index subgroup if necessary, one can assume
that $f$, $g$ form a basis of $G \sim {\bf Z}^2$.

Let $\beta$ be a corner of $\mathcal C'$. We claim that $\beta$ cannot be 
in $\mathcal C$ or in one of its boundary sides. Suppose not. There is $h$ non trivial
in $G$ fixing $\beta$ and therefore fixing every corner
of $\mathcal C'$. As $h$ leaves $\mathcal C$ invariant, then
$\beta$ has to be a corner of $\mathcal C$.
This would produce an element in $G$ fixing every corner of $\mathcal C$ and
every corner of $\mathcal C'$
and hence some powers of $f$ and $g$ coincide.
The previous paragraph shows this is impossible.
Let now $c$ be a path in $\oo$
joining $\beta$ to an element $\delta$ in the union
of the lozenges in $\mathcal C$, and disjoint from the corners of $\mathcal C$. 
We assume that $c$ avoids the singular orbits in $\oo$.
Notice that the union of corners of $\mathcal C$ forms a discrete set in $\oo$.
Consider the intersection $V$ between $c$ and the union
of stable and unstable half-leaves contained in the boundary of 
the lozenges of $\mathcal C$. By the above this intersection is non
empty. Assume first that $V$ is finite. Let $\gamma$
be the first element of $V$ met while traveling along $c$ from $\beta$ to $\delta$.
Then $\gamma$ lies on the boundary of a lozenge $C$ of $\mathcal C$, let's say the
boundary component is a stable
half leaf $L$ containing a corner $\alpha$ of $\mathcal C$. 
Let $C'$ be the other lozenge in $\mathcal C$
admitting also $\alpha$ as a corner: there is a half leaf  $K$, contained in the boundary
of $C'$ and such that the union $L \cup K \cup \alpha$ is an embedded line in
$\oo$, which moreover disconnects $\mathcal C$ from $\beta$. 
In addition this properly embedded line is unique with these properties.
Since $\mathcal C$ and $\beta$
are $f$-invariant, it now follows that $L \cup K \cup \alpha$
is $f$-invariant, and hence $f(\alpha)=\alpha$, where $\alpha$ is a 
corner of $\mathcal C$.
Contradiction.

Therefore, $V$ is not finite: it admits an accumulation point $\gamma$. 
Since $\wls$ and $\wlu$ are transverse outside the singular points,
$\gamma$ is an accumulation point of a sequence $F_n \cap c$, 
where the $F_n$ are leaves in the boundary of lozenges in
${\mathcal C}$. In addition we may assume that all 
$F_n$ have all the same type, for example all $F_n$ are
leaves of $\wls$. 
Let $L$ be the leaf of $\wlu$ through $\gamma$: it intersects all the $F_n$
for $n$ sufficiently big. It follows that $\mathcal C$ contains an infinite u-scalloped subchain. 
Since $\mathcal C$ is $G$-invariant, 
the entire chain $\mathcal C$ has to be a bi-infinite u-scalloped chain. 
Hence it defines
a u-scalloped region $U$.

Similarly, ${\mathcal C}'$ has to be scalloped, and defines a scalloped region $U'$. 

Now the key point is the following: in \cite{Fe5} the following facts are shown:
i) We can choose $h$ in $G$
acting freely on $\oo$;
ii) The leaves of $\wls$ (respectively $\wlu$) intersecting $U$ defines 
a $G$-invariant subline $I^s$ in $\hhs$
(respectively a $G$-invariant subline $I^u$ in $\hhu$);
iii) Every leaf in $I^s$ intersects every leaf in $I^u$,
and this intersections occurs in $U$;
iv) Every point in $U$ is the intersection of a leaf in $I^s$ and
a leaf in $I^u$.

Similarly, the open scalloped region $U'$ provide $G$-invariant sublines $J^s$, $J^u$ in $\hhs$, $\hhu$, such that
every leaf in $J^s$ intersects every leaf in $J^u$ at a point in $U'$. But since $h$ acts freely, $h$-invariant lines in
$\hhs$, $\hhu$ are unique \cite{Fe5}. Thus, $I^s=J^s$ and $I^u=J^u$. The equality $U=U'$ follows.

If the chain ${\mathcal C}'$ was u-scalloped, as $\mathcal C$, 
then it would be equal to $\mathcal C$ since it defines the same
scalloped region. Hence, ${\mathcal C}'$ is s-scalloped. The lemma follows.
\end{proof}

\begin{corollary}
\label{cor:uniqueC}
Let $G$ be a subgroup of $\pi_1(M)$ isomorphic to ${\bf Z}^2$ and $h$ an element of
$\pi_1(M)$ such that $hG'h^{-1}=G'$, where $G'$  is a finite index subgroup of $G$. 
Then $h$ preserves any $G$-invariant chain of lozenges.
\end{corollary}

\begin{proof}
Let $\mathcal C$ be a $G$-invariant chain of lozenges. Then, $\mathcal C$ is $G'$-invariant,
and $h(\mathcal C)$ is $hG'h^{-1} = G'$-invariant. According to lemma~\ref{le:uniqueC},
if $\mathcal C$ is not scalloped, then $\mathcal C$ is the unique $G'$-invariant chain: hence we have
$h(\mathcal C)=\mathcal C$. 
If not, $\mathcal C$ is scalloped, 
for example suppose that ${\mathcal C}$ is s-scalloped. 
Again by lemma \ref{le:uniqueC}, ${\mathcal C}$  is the unique
s-scalloped $G'$-invariant chain, and since $h(\mathcal C)$ is also s-scalloped, the equality
$h(\mathcal C)=\mathcal C$ follows.
\end{proof}

As a corollary of these results, we get the description of 
pseudo-Anosov flows in manifolds with virtually 
solvable fundamental group (theorem B).

\begin{theorem}{}{}
Let $\Phi$ be a pseudo-Anosov flow in $M^3$ with
$\pi_1(M)$ virtually solvable. Then $\Phi$ has no singularities
and is product.
In particular $\Phi$ is topologically equivalent
 to a suspension
Anosov flow.
\label{solva}
\end{theorem}

\begin{proof}{}
First notice that the fact that each leaf of $\wls$ intersect
every leaf of $\wlu$ is invariant up to taking finite covers
and so is the existence of singularities. Hence we can take
finite covers at will.
Up to a finite cover, one can assume that $\pi_1(M)$ is solvable.
Notice that as $M$ has a pseudo-Anosov flow
then $M$ is irreducible.
Since $\pi_1(M)$ is solvable, 
classical $3$-manifold topology results \cite{He}
imply that $M$ fibers over the circle with
fiber a surface $S$ which has solvable fundamental
group. The surface $S$ can only be the torus or the
Klein bottle $K$. Up to another  finite cover 
one can assume that $S$ is actually the torus. 

Assume that $\Phi$ is not product. Then, according to lemma \ref{le:Z2},
$\pi_1(S)$ preserves a chain of lozenges. Since $\pi_1(S)$
is normal in $\pi_1(M)$, it follows from Corollary~\ref{cor:uniqueC} that
this chain of lozenges is $\pi_1(M)$-invariant. 
According to lemma \ref{stachain}, $\pi_1(M)$ is a finite index extension
of ${\bf Z}$ or ${\bf Z}^2$.
This contradicts the fact that $M$ fibers over the
circle with fiber $T^2$.
This finishes the proof.
\end{proof}

\section{$\pi_1$-injective tori in optimal position}
\label{sec:birk}

Given a $\pi_1$-injective torus, we look for a representative
in its homotopy class which is in optimal position $-$ this means
that it is a union of Birkhoff annuli, which have very important
dynamical meaning. If the initial torus is embedded we want to
study when the optimal position torus is also embedded. This 
is tremendously important if one wants to cut the manifold
along the tori which separate pieces in the torus decomposition.

We first study under 
which conditions a chain of lozenges $\mathcal C$
may admit a corner $\alpha$ such that for some element $g$ of
$\pi_1(M)$ the image $g(\alpha)$ is contained in a lozenge of $\mathcal C$.
Later on we explain how this concerns the intersections of corner
orbits in the Birkhoff annuli with the interior of the annuli.

\begin{define}{}{}
Let $\mathcal C$ be a chain of lozenges.
If for any element $g$ of $\pi_1(M)$ and for every corner $\alpha$ of
$\mathcal C$ then the orbit $g(\alpha)$ is not in 
the interior of a lozenge in ${\mathcal C}$, then
$\mathcal C$ is called {\em simple}. The chain $\mathcal C$ is called
a {\em string of lozenges} if no corner orbit
is singular and consecutive lozenges are never adjacent.\label{def:string}
\end{define}

\begin{proposition}{}{}
Let $G$ be a subgroup of $\pi_1(M)$ isomorphic to ${\bf Z}^2$ and 
let ${\mathcal C}$ be a
$G$-invariant chain of lozenges.
Suppose that $\mathcal C$ is not simple.
Then ${\mathcal C}$ is a string of lozenges. In addition
$G$ is contained in the fundamental group of
a free Seifert fibered piece.
\label{toriem}
\end{proposition}

\begin{proof}{}
The chain $\mathcal C$ is not simple. 
Therefore there is a corner orbit $\alpha$ of $\mathcal C$
and an element $g$ in $\pi_1(M)$ so that 
$g(\alpha)$ is in the interior
of a lozenge in $\mathcal C$. 
Once and for all in this proof the orbit $\alpha$ and 
the transformation $g$ are fixed. We stress that the element
$g$ is not used all the time in this proof, but whenever
it is, it refers to this fixed element.
Notice that $g$ is NOT in
$G$ as $G$ preserves $\mathcal C$ and its corners.

\vskip .1in
\noindent
{\bf {Proof that ${\mathcal C}$ is a string of lozenges}}.
We denote by $\{ \alpha_i, i \in {\bf Z} \}$ the corners 
of $\mathcal C$ and by $\{ C_i, i \in {\bf Z} \}$ the
lozenges of ${\mathcal C}$, so that $\alpha_i$, 
$\alpha_{i+1}$ are the corners of $C_i$ for each integer $i$.
Moreover, we assume wlog $\alpha=\alpha_0$. 
By assumption there is an integer $k$ so that 
$\beta = g(\alpha)$ belongs to $C_k$.
We will prove that both corners $\alpha_k$, $\alpha_{k+1}$ of $C_k$ are in 
the interior of lozenges in $g({\mathcal C})$.
Since the orbit $\beta$ is in the interior of a lozenge, then
$\beta$ is non singular and $\ws(\beta),\wu(\beta)$
define exactly 4 quadrants in $\mi$. Two of the quadrants 
contain the corners of $C_k$. Let $W$ be one of
the remaining quadrants. It contains a perfect
fit between two sides of the lozenge $C_k$. Wlog assume that
$S = \ws(\alpha_k)$ and $U = \wu(\alpha_{k+1})$.

We claim that $W$ does not
contain a lozenge 
with corner in
$\beta$. Suppose not and call this lozenge $D_1$.
Then $D_1$ has 2 sides in $\wu(\beta)$ and $\ws(\beta)$.
There is a unstable  side of $D_1$, call it $E$ which is 
contained in an unstable leaf
and makes a perfect fit with $\ws(\beta)$.
Since $\ws(\beta)$ intersects $U = \wu(\alpha_{k+1})$ transversely,
it follows that $S$ separates $U$ from the lozenge 
$C_k$. Therefore $E$ cannot intersect any leaf
which makes a perfect fit with $\wu(\beta)$.
This is a contradiction and proves the claim.

It follows that the 2 quadrants defined by $\beta$ 
which contain respectively $\alpha_k$ and $\alpha_{k+1}$, also contain
lozenges in $g({\mathcal C})$. 
Let $D_2, D_3$ be these
lozenges which are in $g({\mathcal C})$ and have a corner
in $\beta$ (where $g$ is the fixed element in this proof). 
Since $\widetilde W^s(\alpha_{k+1})$ intersects $\widetilde W^u(\beta)$ and
$\widetilde W^u(\alpha_{k+1})$ intersects $\widetilde W^s(\beta)$, the
definition of lozenges implies that $\alpha_{k+1}$ is in
the interior of one of these lozenges $-$ say $D_3$. 
As in the argument above
it now follows that the other corners of $D_2, D_3$
are in the interior of $C_{k-1}, C_{k+1}$ respectively.
This procedure can be iterated indefinitely and in both directions. It now follows
that {\underline {all}}
$g(\alpha_i)$ are in the interior
of lozenges in ${\mathcal C}$.
In particular this implies that 
each $g(\alpha_i)$ (and consequently
the same for the orbits ${\alpha_i}$) is non singular and 
$C_i$, $C_{i+1}$ are not adjacent. 
This shows that ${\mathcal C}$ is a string of lozenges.

\vskip .1in
In order to conclude, we have to show that up to conjugation $G$ is contained 
in the fundamental group of a free Seifert piece. Let $H$ be the stabilizer of
$\mathcal C$ in $\pi_1(M)$, and let $H_0$ be the
maximal abelian subgroup of $H$ (see lemma \ref{stachain} which shows that
$H_0$ has index $\leq 2$ in $H$).
Then $G \subset H_0$; hence we can assume $G=H_0$, ie. that $G$ has index at most two
in $H$. 

\vskip .1in
We stress the following very important fact:
the above arguments show that for any corner $\gamma$ of
${\mathcal C}$ there are exactly 2 lozenges which have 
corner $\gamma$.
The remaining quadrants of $\gamma$ do NOT have
lozenges with corner $\gamma$. As a corollary, we obtain
that the tree ${\mathcal G}(\alpha)$ coincides with $\mathcal C$. Similarly,
${\mathcal G}(\beta)=g({\mathcal C})$.
In particular $\mathcal C = {\mathcal G}(\alpha)$ is a simplicial linear tree.

\vskip .1in
\noindent
{\bf {Claim 1}} $-$ One can assume that the manifold $M$ is orientable.

Suppose that $M$ is not orientable and let $M_2$ be the orientation double
cover of $M$, with lifted flow $\Phi_2$.
Let $l^s$ be the set of stable leaves either
intersecting a lozenge in ${\mathcal C}$ or containing
a corner orbit in ${\mathcal C}$. This set is order
isomorphic to the reals $\rrrr$.
Similarly define $l^u$. One can use the
arguments above to show that
$l^s, l^u$ are invariant under $g$. 
This is because every
$g(\alpha_i)$ is in the interior of a lozenge in $\mathcal C$ $-$
so the arguments above show that if
$q$ is any corner of ${\mathcal C}$, then 
$g(q)$ is also in the
interior of a lozenge in ${\mathcal C}$. This implies the
$g$ invariance of $l^s, l^u$.
If $g$ preserves the order in $l^s$ then the arguments
above imply that $g$ also preserves the order in $l^u$:
this is because
one can order $l^s, l^u$ so that ``high elements" in 
$l^s$ intersect high elements in $l^u$. Since intersection
is preserved by the action of $g$ the statement follows.
This implies that $g$ preserves orientation in $\oo \cong
\rrrr^2$. If on the other hand $g$ reverses order in
$l^s$, the same argument shows that $g$ also reverses
order in $l^u$ and hence $g$ again preserves orientation
in $\oo$. Since clearly $g$ preserves the flow direction
it follows that in any case $g$ preserves orientation
in $M$. Therefore $g$ is an element of $\pi_1(M_2)$.

Similarly, one proves for every element $a$ of $G$ that 
if $a$ reverses the orientation of $l^s$, it also reverses the
orientation of $l^u$: $G$ is contained in $\pi_1(M_2)$.
Now if $P_2$ is a free Seifert piece whose fundamental group
contains $G$, then $P=p(P_2)$ is a free Seifert piece in $M$
whose fundamental group contains $G$. Hence we may assume that
$M = M_2$ in the statement of the proposition. Claim 1 is proved.

Notice that it is not true that any ${\bf Z}^2$ subgroup
of any $3$-manifold group consists entirely of orientation
preserving elements. For instance consider the twisted
$I$-bundle over the torus $T^2$. The one sided torus
in the middle has orientation reversing elements. Glue two
copies of this to produce examples in closed $3$-manifolds.

\vskip .1in
\noindent
{\bf {Assumption}} $-$ From now on we can assume
that $M$ is orientable.

Since $g$ preserves $l^s$,
there are two options:
Case I) $g$ preserves orientation in $l^s$. Then there is $k$ in ${\bf Z}$ so that
$g(\alpha_i)$ is always in the interior of $C_{k+i}$,  \ \ 
Case II) $g$ reverses  orientation in $l^s$. 
We will reindex the $\alpha_i$. This forces a reindexing of the $C_i$ as
$C_i$ is always the lozenge with corners in $\alpha_i$ and $\alpha_{i+1}$.
Here there are two possibilities: First if $k$ is even, then 
shift the $i's$ (by $i \rightarrow i - k/2$) so that
$g(\alpha_i)$ is in the interior
of $C_{-i}$ for all $i$.
If on the other hand $k$ is odd, then first shift the $i's$ by
$i \rightarrow i -  (\frac{k+1}{2})$ which results in 
$g(\alpha_i)$ is in $C_{-i-1}$, then do a reflection $i \rightarrow -i$,
which results in $g(\alpha_i)$ is in $C_{-i}$. So regardless of $k$
even or odd in case II), we can adjust the indices so that
$g(\alpha_i)$ is always in $C_{-i}$.

Then up to choosing a new
$\alpha_0$ and perhaps changing $i$ to $-i$,
it follows that $g(\alpha_i)$ is in the interior
of $C_{-i}$ for all $i$.

\vskip .1in
\noindent
{\bf {Claim 2}} $-$  There is an element $h_0$ of $G$ such that the
centralizer $Z(h_0)$\ (in $\pi_1(M)$) \ is not abelian.

Let $f$ denote a generator of the stabilizer in $G$ of every $\alpha_i$, and 
let $h$ be an element of $G$ acting freely on $\mathcal C$: there is an integer $p$ so that
$h(\alpha_i)=\alpha_{i-p}$, $h(C_i)=C_{i-p}$. 

Assume first that we are in Case I).
For every integer $i$,  $g(\alpha_{pi})$ is contained in $C_{k+pi}$,
hence all the $h^ig(\alpha_{pi})$ lie in $C_k$. On the other hand,
one can produce as in \cite{Ba2} a $f$-invariant proper embedding of 
$[0,1] \times \rrrr$ into $\mi$, so that 
$\{ 0, 1 \} \times \rrrr$ maps into the corner
orbits of $C_k$, $(0,1) \times \rrrr$ maps
into the interior of the lozenge and transversely
to $\wwp$.
The image of this embedding projects 
to an embedded annulus $\hat{A}$ in $\mi/\langle f \rangle$, which 
itself projects to an {\em immersed}
annulus $A$ in $M$, transverse in its interior to the 
flow $\Phi$. The key point is that $A$ is compact, hence
the periodic orbit $\pi(\beta)$ intersects $A$ only a finite number of times.
It follows that 
$\pi(\beta) = \pi(\alpha_0) = \pi(\alpha)$ 
admits only finitely many lifts in $\mi/\langle f \rangle$
intersecting $\hat{A}$. In other words, there must be distinct positive integers $i, j$ and an
integer $q$ such
that:

$$h^i g (\alpha_{pi}) \ = \ f^q (h^j g (\alpha_{pj}))$$


\noindent
Let

$$\alpha' = \alpha_{pj} = h^{-j}(\alpha) 
\ \ \ \ \ {\rm so} \ \ \ \ \ \alpha_{pi} = h^{-i}(\alpha) = h^{j-i}(\alpha')$$

\noindent
Hence:
$$h^igh^{j-i}(\alpha') \ = \ 
f^q h^j g(\alpha')$$

\noindent
So there is $n$ for which $h^i g h^{j-i} = f^q h^j  g s^n$,
where $s$ is the stabilizer 
{\em in $\pi_1(M)$} of $\alpha_{pj} = \alpha'$. 
Let $m = i-j$. Since $f$ and $h$ are both in $G$ they commute, so the 
last equation
implies 

$$g^{-1} f^{-q} h^m g \ = \ s^n h^m$$

\noindent
Notice that $s$ preserves ${\mathcal G}(\alpha) = \mathcal C$.
This is because $\mathcal C$ is a string of lozenges and also the very important fact
mentioned above.
Hence $s$ belongs to $H$. 
Let $h_0 = (s^n h^m)^2$ and $v = (f^{-q} h^m)^2$.
Also since $m$ is not zero then $h_0$ is not the identity.
The equation above implies that 
$g^{-1} v g  = h_0$.
Since $H_0$ has index $\leq 2$ in $H$ then $h_0$ is in $H_0$.
 
%

We conclude that
$h_0$ is a non trivial element 
of $G$ whose centralizer $Z(h_0)$  contains $G$, but also $g^{-1}Gg$.

\vskip .1in
Now suppose we are in Case II) and we want to achieve the same conclusion.
This is similar to Case I) and some details are left to the reader.
Here $g(\alpha_{pi})$ is in $C_{-pi}$ and 
$h^{-i}(C_{-pi}) = C_0$.
As in case I) there are $i, j$ positive and distinct and $q$ integer
to that

$$h^{-i} g(\alpha_{pi}) \ = \ f^q  h^{-j} g(\alpha_{pj}),$$

\noindent
So if $\alpha' = \alpha_{pj}$ then $h^{-i} g h^{j-i} = f^q h^{-j} g s^n$,
with $s$ as above, leading finally to 

$$g^{-1} ( f^{-q} h^{-m}) g \ = \ s^n h^m, \ \ \ \ {\rm where} \ \ \ \ m = i - j \not = 0$$

\noindent
Here take $h_0 = (s^n h^m)^2$ non trivial in $H_0$ and let $v = (f^{-q} h^{-m})^2$.
So as before $g^{-1} v g = h_0$, so again $h_0$ is a non trivial element of $H_0$
whose centralizer contains $G$ and also $g^{-1} G g$.

\vskip .1in
Now assume by way of contradiction that 
$Z(h_0)$ is abelian.
According to lemma \ref{le:uniqueC}, since the chain $\mathcal C$ is not scalloped,
it is the unique $G$-invariant chain of lozenges. Since $g^{-1}Gg$ is a subgroup
of $Z(h_0)$, it commutes with $G$ as $Z(h_0)$ is abelian. It follows that $\mathcal C$ is $g^{-1}Gg$-invariant.

But a similar argument shows that $g^{-1}(\mathcal C)$ is the unique $g^{-1}Gg$-invariant
chain of lozenges. Hence $g^{-1}(\mathcal C)=\mathcal C$. This is a contradiction since
$\beta = g(\alpha)$ is not a corner of $\mathcal C$. 
This finishes the proof of claim 2.

\vskip .1in
Since $Z(h_0)$ is not abelian,
lemma VI.1.5 of \cite{Ja-Sh} shows that there is a 
Seifert fibered piece $P$ of the torus decomposition
of $M$ \cite{Ja-Sh,Jo,Ja} so that $Z(h_0) \subset \pi_1(P)$.
The hypothesis of lemma VI.1.5 of \cite{Ja-Sh} require
i) $M$ is irreducible, ii) $M$ is orientable, iii) $M$ has
an incompressible surface. Condition i) holds because
$M$ has a pseudo-Anosov flow \cite{Fe-Mo}. 
Condition ii) holds because of Claim 1. As for condition iii) we know
that $\pi_1(M)$ has a ${\bf Z}^2$ subgroup. 
Work of Gabai \cite{Ga} or Casson and Jungreis \cite{Ca-Ju}
implies that either $M$ has an embedded incompressible
torus or $M$ is a small Seifert fibered space. But it $M$ is Seifert
fibered, then theorem \ref{Seifert} shows that the fiber in $M$
acts freely on $\oo$ and we are done. 
So we can assume that
condition iii) also holds.
An example of a non simple chain of lozenges 
in Seifert fibered spaces is the following:
let $\Phi$ be a geodesic flow,  $\gamma$ a non simple geodesic
and $T$ the torus associated to $\gamma$ with corresponding chain ${\mathcal C}$.
Then $\mathcal C$ is not simple.

In order to conclude, we just have to show that $P$ is a free piece. Assume this is not the case:
let $t$ be the fiber of a Seifert fibration in $P$ admitting fixed points in $\oo$.

\vskip .1in
\noindent
{\bf {Claim 3}} $-$ For any $\nu$ in $\pi_1(P)$,
$\nu({\mathcal C}) = {\mathcal C}$.

Since $G \subset \pi_1(P)$, for every $a$ in $G$ we have $ata^{-1}=t^{\pm 1}$. Let
$G'$ be the subgroup of $G$ made of elements $a^2$ where $a$ is an
arbitrary element of $G$.
Then $G'$ 
is isomorphic to ${\bf Z}^2$ (it has index $4$ in $G$) and $G'$ is
contained in the centralizer $Z(t)$. The chain $\mathcal C$
is the unique $G'$-invariant chain of lozenges  (lemma \ref{le:uniqueC}). But since
$G' \subset Z(t)$, the chain $t(\mathcal C)$
is $G'$-invariant, hence equal to $\mathcal C$.
Then $t$ has a fixed point which is a corner of $\mathcal C$
and so ${\mathcal G}(t) \subset {\mathcal G}(\alpha)$.

Consider now the action of $G'$ on the tree ${\mathcal G}(t)$. 
Since ${\mathcal G}(t)$ is contained in a linear tree and $G'$ is
isomorphic to ${\bf Z}^2$, 
there is an element $b$ of $G'$ acting freely on ${\mathcal G}(t)$. 
Since ${\mathcal G}(t) \subset {\mathcal G}(\alpha) = \mathcal C$ and the last one is a simplicial
linear tree, it now follows that ${\mathcal G}(t) = \mathcal C$.
Claim 3 follows since ${\mathcal G}(t)$
is obviously $\pi_1(P)$-invariant.

\vskip .1in
The fundamental group $\pi_1(P)$ contains $Z(h_0)$ which itself contains
$g^{-1}Gg$: it follows that $g^{-1}Gg$ preserves $\mathcal C$. We have already observed,
while proving that $Z(h_0)$ is not abelian (claim 2), that this is impossible. This contradiction proves
that $t$ acts freely on $\oo$.
This finishes the proof of proposition \ref{toriem}.
\end{proof}

\noindent
{\bf {Remark}} - The same arguments as in the section ``Proof that
${\mathcal C}$ is a string of lozenges" of the above
proposition prove the following:
suppose that ${\mathcal C}$ is a connected infinite collection
of lozenges and their corners so that no two lozenges
intersect (that is, their interiors are disjoint). 
Suppose that there is a corner $p$ in ${\mathcal C}$
and an element $g$ of $\pi_1(M)$ so that
$g(p)$ is in (the interior of) a lozenge in ${\mathcal C}$. Then
$C$ is a string of lozenges with their corners.


\begin{define}\label{def:birk}
A {\em Birkhoff annulus} is an immersed annulus in $M$ so that 
each boundary component is a periodic orbit of the flow, and such that
the interior of the annulus is transverse to the flow.
If the interior is embedded, then the
annulus is called {\em weakly embedded.} 
\end{define}

The interior of a Birkhoff annulus is
transverse to the flow, and hence is also transverse to
the weak foliations $\Lambda^s$, $\Lambda^u$. They therefore induce
foliations on the annulus denoted by  $l^s$, $l^u$. These foliations can both be extended 
to the boundary of the annulus as foliations tangent to the boundary.
A singular orbit with $p$ prongs (here again we use that for
pseudo-Anosov flows $p \geq 3$) induces a singularity of 
$l^s$ (or $l^u$) in the interior of the annulus having 
negative index $1-p/2$. Since the Euler characteristic of
the annulus is zero, 
Poincar\'e-Hopf index formula implies that the interior
of the annulus intersects no singular orbits. 

\begin{define}\label{def:elementary}
A Birkhoff annulus is {\em elementary} if $l^s,$ $l^u,$ do not have closed leaves
in the interior.
\end{define}

Observe that in the definition of weakly embedded Birkhoff annuli,
we did not require the whole annulus to be embedded: it may wrap
around each periodic orbit in its boundary, an arbitrary (finite) number of times.
Notice however that the boundary cannot intersect the interior, as otherwise
points near the boundary would produce self intersections in the interior.

Let $\Upsilon: A \hookrightarrow M$ be a Birkhoff annulus (embedded or not). It lifts as an 
immersion $\tilde{\Upsilon}: \tilde{A} \sim {\bf R} \times [0, 1] \hookrightarrow \mi$ such
that ${\bf R} \times \{ 0 \}$, and ${\bf R} \times \{ 1 \}$ are orbits of $\wwp$, and
such that the image by $\tilde{\Upsilon}$ of $ {\bf R} \times (0, 1)$ is transverse to $\wwp$:
we call 
$\widetilde \Upsilon: \widetilde A \hookrightarrow \mi$ a {\em Birkhoff band.}
Moreover, this image is invariant under the action of 
the cyclic subgroup ${\Upsilon}_\ast(\pi_1(A)) \sim {\bf Z}$. 
Finally, if $\Upsilon: A \hookrightarrow M$ is elementary, 
every orbit of $\mi$ intersects the image of the interior in at most one point, and
the projection in $\oo$ is a ${\Upsilon}_\ast(\pi_1(A))$-invariant 
lozenge union its two corners but without the sides (\cite[Proposition 5.1]{Ba2}).
This set is neither closed nor open in $\oo$.

Conversely, and as we already mentioned in the proof of Proposition \ref{toriem}, Claim 2, 
every lozenge in $\oo$ invariant by a cyclic subgroup of $\pi_1(M)$ is the projection
in $\oo$ of an embedded Birkhoff band in $\mi,$ that 
projects in $M$ to an elementary Birkhoff annulus. Moreover, if the lozenge
is {\em simple}, ie. if its interior contains no iterate of its corner, then the
Birkhoff annulus can be selected weakly embedded (\cite[Theorem D]{Ba2}).


\vskip .1in

More generally, let $\mathcal C$ be a string of lozenges invariant under a subgroup $G$
of $\pi_1(M)$ isomorphic to ${\bf Z}^2$. Then, there is a cyclic subgroup $H$ of $G$
fixing every lozenge in $\mathcal C$. We lift all the lozenges to $\mi$, so that the lift of every two
successive lozenges share a common $H$-invariant orbit. This can be done in a $G$-equivariant way.
We also lift the entire corner orbits of the lozenges. The union is a set which 
is $G$ invariant in $\mi$ and 
projects in the quotient of 
$\mi$ by $G$ to an embedded torus. This this torus projects 
to an immersed torus in $M$ which is a union of elementary Birkhoff annuli.

\begin{define}
A {\em Birkhoff torus} is an immersion $\Upsilon: T \to M$ of a torus $T$, such that $T$ is an union
of distinct annuli $A_i$ for 
which every restriction $\Upsilon: A_i \to M$ is an elementary Birkhoff annulus.
In addition we require the following: if $A_i$ and $A_{i+1}$ are two consecutive
annuli abutting the common closed orbit $\gamma$, then locally near
$\gamma$ the annuli  $A_i$ and $A_{i+1}$ are
in distinct quadrants defined by $\gamma$.

Similarly, a {\em Birkhoff-Klein bottle} is an immersion of the Klein bottle whose
image is an union of elementary Birkhoff annuli. We have the same restriction
on abutting annuli as in the tori case.
\end{define}

Notice the restriction to elementary Birkhoff annuli.

In the sequel, a {\em closed Birkhoff surface} means 
a Birkhoff torus or a Birkhoff-Klein bottle.
A Birkhoff surface is an union of Birkhoff annuli. 
It contains a finite number of periodic orbits of $\Phi$,
called the {\em tangent orbits,} 
and is transverse to $\Phi$ outside these periodic orbits.

The reason for the added condition about quadrants is
the following: 
Without it we could have started with 
say an embedded Birkhoff annulus $A$ and flow $A$ forward slightly
to an annulus $A'$ which is disjoint from $A$ in the interior.
The union $T = A \cup A'$ is a torus which is NOT incompressible
as it bounds an obvious solid torus. 
We now explain the added condition on quadrants. Let $\gamma$ be
a closed orbit in the boundary of an elementary Birkhoff annulus
$A$. 
Denote this boundary component of $A$ by $\partial _1 A$.
The stable and unstable leaves of $\gamma$ define quadrants:
in a neighborhood of $\gamma$ they are 
the components of the complement of the union of the local sheets
of $W^s(\gamma)$ and
$W^u(\gamma)$ near $\gamma$.
The interior of the annulus $A$ cannot intersect these local sheets of 
$W^s(\gamma)$ or $W^u(\gamma)$. This is because the interior is
transverse to $\Phi$ and $A$ is compact $-$ an intersection
would create a closed curve intersection. This is disallowed
because all Birkhoff annuli are elementary.
Therefore near $\partial _1 A$, the annulus 
$A$ enters a unique quadrant defined by $\gamma$.
In the universal cover this means that the lift
$\widetilde A$ enters a well defined, unique  lozenge with corner
$\widetilde \gamma$.

\begin{define}
A closed Birkhoff surface $\Upsilon: S \to M$ is called {\em weakly embedded} if the Birkhoff annuli $\Upsilon: A_i \to M$ are all weakly embedded, 
with interiors two-by-two disjoint.

If moreover $\Upsilon: S \to M$ is an embedding, then the closed Birkhoff surface is {\em embedded.}
\end{define}

As explained above, the condition that interiors are embedded and two by two disjoint implies that none 
of the tangent periodic orbits of $\Upsilon(S)$
intersects any interior of the annuli.

\begin{proposition}\label{pro:Csimple}
Let $\mathcal C$ be a string of lozenges in $\oo$ 
invariant under a subgroup $G$
of $\pi_1(M)$ isomorphic to ${\bf Z}^2$ or $\pi_1(K)$. Then $\mathcal C$ is the projection in $\oo$
of the lift to $\mi$ of a closed Birkhoff surface $\Upsilon: S \to M$. More precisely,
$\Upsilon: S \to M$ is the composition $\hat{p} \circ \hat{\Upsilon}$ of an embbeding $\hat{\Upsilon}: S \to \widehat{M}$ 
and the covering map $\hat{p}: \widehat{M} \to M$, where
$\widehat{M}$ is the quotient of $\mi$ by $G$.

Moreover, if $\mathcal C$ is simple,
ie. if no element of $\pi_1(M)$ maps a corner of $\mathcal C$ in the interior of
a lozenge of $\mathcal C$, then the closed  Birkhoff surface can be selected weakly embedded.
\label{simple}
\end{proposition}

\begin{proof}
The first part has been explained before in the case where $G$ is abelian, and
is easily generalized to the case $G\sim\pi_1(K)$: the matter is to find a fundamental domain
of the action of $G$ on the set of lozenges in $\mathcal C$, to lift each lozenge in this fundamental
domain to a Birkhoff band, and then to lift all other lozenges in $\mathcal C$ as Birkhoff bands in
a $G$-equivariant way.

Assume now that the chain is simple. Every lozenge in it is simple.
Then the closed Birkhoff surface is an union of weakly embedded Birkhoff annuli, whose
interiors are all disjoint from the tangent periodic orbits. Since the chain
is simple, we can prove, using the technics in 
\cite[\S~7]{Ba2} that through some isotopy along the flow, the interiors of the elementary annuli can be made
disjoint from each other, that is, the Birkhoff surface is weakly embedded.
\end{proof}

All of these results in \cite{Ba2} were stated and proved for smooth Anosov flows.
However, exactly the same techniques work for general pseudo-Anosov flows.

More generally, using the results above, then 
according to lemma~\ref{le:Z2}:

\begin{lemma}{}{}\label{le:kleinbirkhoff}
 Let $G$ be a subgroup of $\pi_1(M)$ isomorphic
to ${\bf Z}^2$.
Suppose that the pseudo-Anosov flow $\Phi$ is not product. Then $G$ is the image
$\Upsilon_\ast(\pi_1(T))$ of the fundamental group of a
Birkhoff torus $\Upsilon: T \to M$.
%
\label{immers}
\end{lemma}

From the conclusion of theorem F it is easy to construct many
weakly embedded Birkhoff surfaces that are not homotopic
to an embedded Birkhoff surface.
Observe that weakly embedded closed Birkhoff surfaces may 
fail to be embedded for various reasons:

I)  every Birkhoff subannulus may be non-embedded, wrapping around one or both of the
tangent periodic orbit in its boundary. It means that some element $g$ of $\pi_1(M)$ (corresponding
to the periodic orbit)  is not
in $G$, but $g$  preserves a corner in $\mathcal C$ (where $\mathcal C$ is 
the $G \cong {\bf Z}^2$ invariant chain of lozenges).

II) an element of $\pi_1(M)$ may map a corner $\alpha$ of $\mathcal C$ to another corner $\beta$ of $\mathcal C$
which is not in the $G$-orbit of $\alpha$, ie. a tangent periodic orbit can be the boundary of more than
two Birkhoff subannuli.
This is the case in the Bonatti-Langevin example (\cite{Bo-La}).

III) even an element $g$ of $\pi_1(M)$ not in $G$ 
could map a lozenge in $\mathcal C$ to another lozenge in ${\mathcal C}$. 
At the Birkhoff surface level this implies the existence of 
two different elementary Birkhoff annuli (in the torus) sharing the same boundary components
and homotopic one to the other along the orbits of $\Phi$. 
This situation typically arises in Proposition~\ref{simple} if 
 $G$ is a finite
index subgroup of a bigger group preserving the chain $\mathcal C$. 

\vskip .1in
\noindent
{\bf {Remark:}} Let us first stress out that possibility
I) can certainly happen.
For example let $\Phi$ be
the geodesic flow in the unit tangent bundle of an 
{\underline {orientable}}
hyperbolic surface and let $T$ be the set of unit vectors
along a simple closed geodesic.
Let $\gamma$ be one closed orbit in $T$. Put coordinates
in the torus $\partial N(\gamma)$ 
so that $(0,1)$ is the meridian and $(1,0)$ is the trace
of say the stable foliation. The construction here is more general, the key
fact used is that the trace of the stable foliation
intersects the meridian once.
Do Dehn surgery on $\gamma$ so that the new meridian
is $(1,n)$ where $n$ is an integer $> 1$.
Isotoping the old torus slightly to 
a torus $T'$ avoiding $\gamma$ we see that it
survives the Dehn surgery. After Dehn surgery $T'$
is homotopic to a 
Birkhoff torus, with Birkhoff annuli  which wrap
$n$ times around the orbit $\gamma$.
Since it is a Birkhoff torus, it is $\pi_1$-injective
and so is $T'$. This gives the desired examples.
In fact the surgery procedure can be done by blowing up
the orbit $\gamma$ into a boundary torus and then
blowing back using the new meridian information \cite{Fr}.
Therefore the new Birkhoff torus can be taken as the 
result of the original Birkhoff torus under this procedure.

A Birkhoff torus is $\pi_1$-injective because of the following:
a closed curve is homotopic to either a closed orbit in
the Birkhoff torus or to a curve transverse to say the
stable foliation in the torus. In the first case the curve
represents a power of a closed orbit, which is not
null homotopic \cite{Fe-Mo}. In the second case, the curve 
is transverse to the stable foliation 
{\underline {in the torus}}. 
The condition that consecutive annuli abutt the closed
orbit from distinct
quadrants of that closed orbit implies that 
this curve is also transverse to the stable foliation
{\underline {in the manifold}}.
It now follows from the theory of essential laminations
that this curve 
is not null homotopic in the manifold \cite{Fe3,Ga-Oe}.

\vskip .1in
The notion of weakly embedded tori is sufficient to 
analyse the relationship between (possible) singular
orbits of the flow and the torus decomposition of $M$.

\begin{proposition}{}{}
Let $\alpha$ be a singular orbit of a pseudo-Anosov flow $\Phi$ in $M$.
Then $\alpha$ is homotopic into a piece of the torus decomposition of
$M$.
\label{singint}
\end{proposition}

\noindent
{\bf {Remarks}} $-$ 1)  Clearly this is not true for regular periodic orbits:
for example there are (non Seifert) graph manifolds with Anosov flows
which are transitive $-$ for example the flows constructed
by Handel and Thurston \cite{Ha-Th}, which are actually
volume preserving. 
Then there are dense orbits and hence  periodic orbits which are not
homotopic into any Seifert fibered piece.
2) If $M$ is atoroidal, the lemma is trivial.
3) Notice that $\alpha$ may be homotopic
into several pieces. If that happens then $\alpha$ is homotopic
into a torus $T$ which  is boundary of two pieces.
Finally, if $\alpha$ is homotopic into a third piece, then $\alpha$ has
to be homotopic through a piece $P$. The piece $P$ cannot be
atoroidal, because an atoroidal piece is a cusped hyperbolic
manifold and consequently acylindrical. Hence $P$ is Seifert
and since $\alpha$ is homotopic to two distinct boundary
components of $P$, it follows that $\alpha$ and the fiber
of $P$ have common powers. In that case the piece $P$ has
to be a periodic piece. Then $\alpha$ is homotopic into
any other piece $P_1$ which intersects $P$. Notice that $\alpha$
cannot be homotopic into any additional piece: otherwise
$\alpha$ would be homotopic
through a second Seifert piece $P'$. That would force the
fibers in $P$ and $P'$ to be the same, which is impossible.

\begin{proof}{}
Let $T_1, ..., T_a$ be the cutting tori in a torus decomposition of $M$ $-$ with
complementary components $P_1, ..., P_b$, which are either Seifert fibered or 
atoroidal. By a small isotopy assume that $\alpha$ is
transverse to the collection $\{ T_i \}$.
Fix a lift $\widetilde \alpha$ to $\mi$ and 
let $g$ in $\pi_1(M)$ be associated to $\alpha$ so that
$g(\widetilde \alpha) = \widetilde \alpha$.
Consider the collection of all lifts of the $\{ T_i \}$.

%

\vskip .1in
\noindent
{\bf {Case 1}} $-$ Suppose that $\widetilde \alpha$ eventually stops
intersecting lifts of the $\{ T_i \}$.

Since $\alpha$ is closed, this shows that $\alpha$ is contained in
a component of the complement of $\{ T_i \}$.

\vskip .1in
\noindent
{\bf {Case 2}} $-$ Suppose that $\widetilde \alpha$ keeps
intersecting a fixed lift $\widetilde T$ in
points $p_k = \wwp_{t_k}(p_0)$ where $t_k$ converges to infinity.

Let $V$ be the tree, whose vertices are the components
$\mi -$ (lifts of $\{ T_i \}$) and edges are the 
lifts of $\{ T_i \}$.
Then $\pi_1(M)$ acts on $V$. 

By transversality, the intersection of $\alpha$ and $\{ T_i \}$ is
finite.
Up to subsequence we may assume that $\pi(p_k)$ is a single point.
The projection to $M$ of $\wwp_{[t_k, t_{k'}]}(p_0)$ is
the orbit $\alpha$ being traversed a number $n$ of times.
This shows that $\alpha^n$ is freely homotopic into some
$T_i$. 
It follows that 
$g^n$ preserves
an edge in $V$ and so does not act freely on $V$. 
Therefore $g$ also does not act freely on $V$.
There are two options:
If $g$ acts as an inversion in the tree  $V$, then it fixes an
edge associated to a lifted torus
$\widetilde T_*$ and then $\alpha$ is homotopic
into the torus $T_* = \pi(\widetilde T_*)$. Then we are done.
Otherwise $g$ fixes a vertex
in $V$ and hence $\alpha$ is homotopic into a piece
of the Seifert fibered decomposition.

\vskip .1in
\noindent
{\bf {Case 3}} $-$ $\widetilde \alpha$ intersects 
distinct lifts
$\widetilde T^j, j \in {\bf N}$ of elements in $\{ T_i \}$.

By the proof of case 2, it follows that the assumption of case 2 does
not hold. Therefore $\widetilde \alpha$ eventually stops intersecting
any single lift $\widetilde T$ of the $\{ T_i \}$.
In addition if distance between $\widetilde \alpha$ and any single
lift $\widetilde T$ does not converge to infinity as time goes to 
infinity then: up to subsequence we may may assume there are
$p_k$ in $\widetilde \alpha$ with $d(p_k, \widetilde T)$ bounded.
We may then assume that $\pi(p_k)$ converges in $M$ and up to a
small adjustment and subsequence we may assume that 
$\pi(p_k)$ is constant. 
In addition $p_k$ is a bounded
distance from $z_k$ in $\widetilde T$. Up to another subsequence
assume that $\pi(z_k)$ converges in $M$ and since $\pi(T)$ is compact,
we may assume that $\pi(z_k)$ is constant. 
The projection of $\wwp_{[t_k,t_{k'}]}$ is 
$\alpha$ being traversed $n$ times. 
The projection of an arc in $\widetilde T$ from $z_k$ to $z_{k'}$ is
a closed curve in $T$. Up to another subsequence assume that the
geodesic arcs from $z_k$ to $p_k$ have images in $M$ which are
very close. This produces a free homotopy from $\alpha^n$ and
a closed curve in $\pi(T)$. Now the proof is exactly as
in Case 2.

Hence assume that $d(p_k, \widetilde T)$ converges to infinity for
any fixed lift $\widetilde T$.
If $\widetilde \alpha$ keeps returning to the same component
of $\mi -$ (lifts of $\{ T_i \}$), then some power of $\alpha$
preserves this component and an argument as in case 2 finishes
the proof.

Finally we can assume that $\widetilde \alpha$ crosses $\widetilde T^j$ for
each $j$ and eventually switches from one component of 
$\mi - \widetilde T^j$ to the other. There is a smallest
separation distance $a_0 > 0$ between any two lifts of
$\{ T_i \}$.
Homotope each $T_i$ to a Birkhoff Torus,  union of 
Birkhoff annuli $\{ B_m \}$ and
lift these homotopies to $\mi$. Each point is moved at most a constant
$a_1$. 
Fix $j$ and let $j'$ vary.
The fact that $d(\widetilde T^j, \widetilde T^{j'})$
goes to infinity means that $\widetilde \alpha$ has to cross some lift 
$\widetilde B_m$ of some Birkhoff annulus
$B_m$ and cannot be contained in $\widetilde B_m$.
But this is a contradiction because the orbits intersecting the interior
of a Birkhoff annulus are never singular.
This finishes the proof of lemma \ref{singint}.
\end{proof}

\noindent
{\bf {Remark}} One can also prove this by using group actions on trees
more extensively: the element $g$ either fixes a point, or $g$
has an inverted
edge or acts freely on $V$. To use this, further work is needed,
for instance in the first case, one needs to 
find a vertex $p$ in $V$, fixed by $g$ so that $\widetilde
\alpha$ has a point in the region associated to $p$ and similarly
for the other cases.

\begin{theorem}{}{}
Suppose that $M$ is orientable and
that $\Phi$ is not product.
Let $T$ be an embedded, incompressible
torus in $M$. Then either 1) $T$ is isotopic to an embedded
Birkhoff torus, or
2) $T$ is homotopic to a weakly embedded Birkhoff 
torus and contained in a periodic Seifert
fibered piece, or
3) $T$ is isotopic to the boundary of the tubular neighborhood of an embedded Birkhoff-Klein bottle
contained in a free Seifert piece.
\label{embedd}
\end{theorem}

\begin{proof}{}
Using proposition~\ref{pro:Csimple},
let $\Upsilon_0: T \to M$ be an immersed Birkhoff torus homotopic to $T$
and let $T_* = \Upsilon_0(T)$.
Let ${\mathcal C}$ be the chain of lozenges invariant under $\pi_1(T)$ and associated
with the torus $T^*$ (a priori there could be two $\pi_1(T)$-invariant chains, if they
are scalloped).
The proof of this proposition has similarities with 
that of proposition \ref{toriem}, but notice that some of
the conclusions 
are {\underline {opposite}}.


\vskip .1in
\noindent
{\bf {Step 1}} $-$ {\underline {Claim}}:
 $\mathcal C$ is simple.

Suppose this is not true.  Then there is a corner orbit
$\alpha$ in ${\mathcal C}$ and $f$ in $\pi_1(M)$ with $f(\alpha) = \beta$
intersecting the interior of a lozenge in ${\mathcal C}$.
Let $g$ be a generator of the isotropy group of $\beta$. Let 
$\widetilde T_*$ be the lift of $T_*$ to $\mi$ which 
is invariant under $\pi_1(T)$. Similarly let $\widetilde T$ be the
lift of $T$ invariant under $\pi_1(T)$.
Then $\beta$ intersects $\widetilde T_*$ in a single point
$p$.  Let $\beta^+$, $\beta^-$ be the two rays of $\beta$ defined by $p$. 
Notice that $\widetilde T_*$ is embedded and separates $\mi$.
Hence $\beta^+$ and $\beta^-$ are in distinct
components of $\mi - \widetilde T_*$.
In addition $\widetilde T$ also separates $\mi$.

Assume that $\beta^+$ and $\beta^-$ are not at bounded distance from $\widetilde T_*$:
for any $R > 0$, there are points $q_R^-$, $q_R^+$ in 
$\beta^-$, $\beta^+$, each at distance $>R$ from 
$\widetilde T_*$. But $T$ and $T_*$ are freely homotopic: there is some $R_0$ such that $\widetilde T$ is contained
in the $R_0$-neighborhood of $\widetilde T_*$: for any $R>2R_0$, any path joining a point $q^-$ to a point $q^+$ such that $d(q^\pm, q^\pm_{R}) < R$ 
must intersect $\widetilde T$. 

On the other hand, 
the closed orbit $\pi(\beta)$ is homotopic in $M$ to a curve 
in $T$. But $T$ is embedded and $M$ is orientable, so $T$ is two sided
and $\pi(\beta)$ is homotopic
to a curve disjoint from $T$ (it is crucial that $T$ is embedded
here!). Lift this to a homotopy from $\beta$ to a curve
$\beta_1$ disjoint from $\widetilde T$. The homotopies from
from $\beta$ to $\beta_1$
move points a bounded distance. Hence, there is a positive number $r$ such that
for every $R>0$, there are points $m^\pm_R$ on $\beta_1$
such that $d(m^\pm_R, q^\pm_{R}) < r$. Take $R>2R_0$, $R>r$: according to the above,
the segment in $\beta_1$ between $m^-_R$ and $m^+_R$ must intersect $\widetilde T$. Contradiction.

Therefore, one of the two rays (say $\beta^+$) is at bounded distance $\leq a_1$
from $\widetilde T_*$.
Consider a sequence of points 
$p_i =g^{n_i}(p)$ in $\beta^+$ which all
project to the same point $\pi(p_1)$ in $M$.

Let $q_i$ in $\widetilde T_*$ a distance $\leq a_1$ from $p_i$. Up
to subsequence assume that $\pi(q_i)$ converges in $M$.
Since $T_*$ is compact, we can assume that $\pi(q_i)$ is constant.
Now up to another subsequence assume that there are geodesic
segments $u_i$ in $\widetilde M$ from $p_i$ to $q_i$ so that
$\pi(u_i)$ converges in $M$.
Again by small adjustments we can assume that $\pi(u_i)$ is constant
for $i$ big. Consider the following closed curve in 
$\mi$: a segment in $\beta$ from $p_i$ to $p_k$, \  $k > i$, then
the segment $u_k$, then a segment in $\widetilde T_*$ from
$q_k$ to $q_i$ and finally a segment from $q_i$ to $p_i$
along $u_i$. Since $\pi(u_i) = \pi(u_k)$ this projects
to a free homotopy from a power of the loop $\pi(\alpha)$ to a closed
curve in $T_*$. In other words, $g^n(q_i)=q_k$ for some
$n$ in ${\bf Z}$. Hence for some $n$ different from $0$, $g^n$ leaves
$\widetilde T_*$ invariant.

But this implies that $g^n$ leaves ${\mathcal C}$ invariant. Since $g^n(\beta) = \beta$,
then $g^n$ leaves invariant the lozenge $C$ of ${\mathcal C}$ containing $\beta$
in its interior
and $g^n$ is not the identity. But then $g^n$ does not leave
invariant any orbit in the interior of $C$ $-$ contradiction to it
leaving $\beta$ invariant. This proves the claim.

\vskip .1in
Let $G = \pi_1(T)$.
According to proposition~\ref{pro:Csimple} we can choose the Birkhoff torus $\Upsilon_0: T \to M$
weakly embedded. As we already observed, if this Birkhoff torus is not embedded, some element
$g$ of the set $(\pi_1(M) - \pi_1(T))$ maps 
a corner of $\mathcal C$ to a corner of $\mathcal C$. Our strategy is to enlarge
$G$ to a bigger subgroup of $\pi_1(M)$, containing all these elements.

Let $\mathcal G$ denote the tree ${\mathcal G}(\alpha)$ where $\alpha$ is a corner in $\mathcal C$.
The chain 
$\mathcal C$ corresponds to a $G$-invariant line in $\mathcal G$. 
Let $H$ be the subset of $\pi_1(M)$
of those $h$ such that 
there is a vertex $\beta$ of $\mathcal G$ such that $h(\beta)$ is also
a vertex of $\mathcal G$. In particular ${\mathcal G}(\beta) = {\mathcal G}(\alpha) 
= {\mathcal G}(h(\beta))$.
Then, for every $h$ in $H$:

$$ h({\mathcal G}(\alpha)) = 
 h({\mathcal G}(\beta)) =  {\mathcal G}(h(\beta)) = 
{\mathcal G}(\alpha),$$

\noindent
hence $H$ is the stabilizer of $\mathcal G$. In particular $H$ is a 
subgroup of $\pi_1(M)$.



Let $H_0$ the subgroup of $H$ acting trivially on $\mathcal G$: $H_0$ is a cyclic normal
subgroup of $H$, generated by an element $h_0$. 
Let $H'$ be the centralizer of $H_0$ (or $h_0$) 
in $H$: it is a normal subgroup of $H$ of index at most $2$.


\vskip .1in
\noindent
{\bf {Step 2}} $-$ The case where $H'$ is abelian.

Here $H' \cong {\bf Z}^2$.
Since $G \cap H'$ has finite index in
$H'$ which is abelian, it follows that 
$\mathcal C$ is $H'$-invariant (Corollary~\ref{cor:uniqueC} ). Now since $H'$ is normal
in $H$, the same result shows that $\mathcal C$ is $H$-invariant.
By Lemma~\ref{stachain},
$H$ is isomorphic to ${\bf Z}^2$ or $\pi_1(K)$ $-$ since it contains
a ${\bf Z}^2$. 
This is the crucial conclusion in this case.

Apply Proposition~\ref{pro:Csimple} to $H$ using that $\mathcal C$ is
simple:
there is a weakly embedded closed Birkhoff
surface $\Upsilon_1: S \to M$ with $(\Upsilon_1)_*(\pi_1(S))=H$. 
It follows from the discussion following lemma
\ref{immers} that $\Upsilon_1: S \to M$ is an embedding, since any element 
of $\pi_1(M)$ mapping a corner of $\mathcal C$ to a corner of $\mathcal C$ lies in $H$.

Suppose first that $S$ is a torus, that is, $H$ is isomorphic to ${\bf Z}^2$.
If $S$ is one sided, then $M$
is non orientable, contrary to hypothesis. Therefore there is a neighborhood
$N$ of $S$ homeomorphic to $S \times [0,1]$.
As the initial embedded torus $T \subset M$
is homotopic into $N$, it now follows from classical $3$-dim topology \cite{He}
that $T$ is homotopic and in fact  isotopic to the embedded Birkhoff torus
$\Upsilon_1(S)$. 
In other words, $T$ is isotopic to
an embedded Birkhoff torus: we are done (case 1) of the statement of the proposition).

Consider now the case where $S$ is the Klein bottle. Since
$M$ is oriented, $\Upsilon_1(S)$
admits a tubular neighborhood $U$ in $M$
diffeomorphic to the non-trivial line bundle over $K$. The boundary of $U$
is an embedded torus $T'$. As above $T$ is homotopic into $U$ and
has to be homotopic and in fact isotopic to $T'$.

Now observe that $U$ is a Seifert submanifold which is not a product
of surface cross interval. It follows that $U$ is contained in a Seifert piece $P$ 
of the torus decomposition of $M$ (that is $S$ is not in the boundary of
two intersecting atoroidal/hyperbolic pieces). If $P$ is periodic then proposition
\ref{simple} implies that $T$ is homotopic to a 
weakly embedded Birkhoff torus $-$ this is case 2) of the 
statement of the proposition. 
If $P$ is not periodic then we are in case 3).
We are done in this case.

\vskip .1in
\noindent
{\bf {Step 3}} $-$ The case where $H'$ is not abelian.

Since $H'$, contained in the centralizer of $H_0$, is not abelian, 
lemma VI.1.5 of \cite{Ja-Sh} shows that there is a Seifert fibered piece
$P$ of the torus decomposition of $M$ so that $H' \subset \pi_1(P)$. 
Let $t$ be a representative of the regular fibers
of $\pi_1(T)$. Then the centralizer $Z(t)$ of $t$ in $\pi_1(P)$ 
(the characteristic subgroup) has index at most 2 in $\pi_1(P)$.
Since $H' \subset \pi_1(P)$, the centralizer $Z(t)$ contains a finite index subgroup $G''$
of $G$. Hence, according to corollary~\ref{cor:uniqueC}, $t$ 
preserves the chain $\mathcal C$: in particular $t$ belongs to $H$.

Assume that $P$ is periodic, ie. that $t$ can be selected acting non-freely on $\oo$. Then, 
according to proposition~\ref{pro:Csimple}, $T$ is homotopic to a weakly embedded Birkhoff torus,
contained (up to homotopy) in $P$. We are in case 2) of the proposition.
Notice that in general there may be identifications in the boundary orbits
as already described.
A priori any of problems I), II) or III) described after lemma \ref{immers}
may occur.

The last case to consider is the case where
$t$ acts freely on $\oo$. Then $\mathcal C$ represents the axis of $t$
in the tree ${\mathcal G}$. 
When $t$ acts freely it may not leave invariant a unique chain of lozenges,
for example as happens in the geodesic flow case.
However the key fact here is that $H$ preserves $\mathcal G$ and then 
$\mathcal C$ is the unique axis of $t$ in $\mathcal G$.
Since $H'$ is contained in $\pi_1(P)$, 
some finite index normal subgroup $H''$ of $H$
is contained in $Z(t)$. 
For any $g$ in $H''$, then $tg(\mathcal C) = g(\mathcal C)$ so by the 
uniqueness above, $g(\mathcal C) = \mathcal C$,
or $\mathcal C$ is preserved by $H''$. Since $H''$
is normal in $H$
then again it follows that $H$ preserves 
$\mathcal C$.
Now we conclude almost as in step 2: if $H$ is isomorphic to ${\bf Z}^2$ then
$H'$ is abelian, contradiction to assumption in case 3).
If $H$ is isomorphic to $\pi_1(K)$ then 
$T$ is isotopic to the boundary of a tubular neighborhood of an embedded Birkhoff-Klein bottle
contained in $P$, which can be periodic (case 2) or free (case 3)).
\end{proof}

\noindent
{\bf {Remark:}}
We remark that tori homotopic to a double cover of a Birkhoff-Klein bottle appearing
in step 2 and 3 actually occur in the free case and in the periodic case too. 
The periodic case occurs for example in the Bonatti-Langevin flow \cite{Bo-La}.
An example of the free case was described in the remark at the end of 
section \ref{seifconj}.

\vskip .1 in
\noindent
{\bf {Remark:}} \ The hypothesis of orientability for $M$ in 
proposition \ref{embedd} occurs because several results for
torus decompositions and maps of Seifert spaces into manifolds
are only
clearly stated in the literature for orientable manifolds,
for example \cite{Ja-Sh}.

\section{Periodic Seifert fibered pieces}\label{sec:perio}

This section is devoted to the proof of theorem F
$-$ in particular we assume that $M$ is orientable.
Let $P$ be a (non trivial) Seifert fibered piece
of a $3$-manifold $M$ with a pseudo-Anosov flow $\Phi$. 
We will analyse here only the case
that the regular fiber $h_0$ in $\pi_1(P)$ does not act freely on $\oo$, that is
$P$ is a periodic piece.
By theorem \ref{Seifert} this implies that $P$ is not all of $M$.
We start by constructing a canonical tree of lozenges associated to $P$.
First consider the action on $\oo$: there is $\alpha$ in $\oo$ with
$h_0(\alpha) = \alpha$. Let ${\mathcal T}$ be the fat tree ${\mathcal G}(\alpha)$.
Given $g$ in $\pi_1(P)$, then \ $g h_0 g^{-1} = h_0^{\pm 1}$ \ so
$h_0 g(\alpha) = g(\alpha)$ \ and $g(\alpha)$ is in ${\mathcal G}(\alpha)$.
It follows that $\mathcal T = {\mathcal G}(\alpha)$ is a
$\pi_1(P)$-invariant tree. The kernel of the $\pi_1(P)$-action on
${\mathcal T}$ is a normal cyclic 
subgroup $H_0$ of $\pi_1(M)$, which contains a non-trivial power $h_0^n$
of $h_0$ (cf. proposition~\ref{pro:treefacts}).

Notice that there is at least a ${\bf Z} \oplus {\bf Z}$
in $\pi_1(P)$ so there are elements in $\pi_1(P)$ acting
freely on $\mathcal T$.
We now go through several steps to produce a normal
form of the flow in $P$.

\vskip .1in
\noindent
{\underline {Pruning the tree ${\mathcal T}$}} 

We first construct a subtree of ${\mathcal T}$ which is still $\pi_1(P)$-invariant
and has no vertices of valence one. Given
$g$ in $\pi_1(P)$ acting freely on ${\mathcal T}$ let $A(g)$ be the axis of
$g$ in ${\mathcal T}$.
Let now ${\mathcal T}'$ be the union of all axes
$A(g)$, for all $g$ in $\pi_1(P)$ acting freely on $\oo$.
Clearly ${\mathcal T}'$ is $\pi_1(P)$ invariant and has no vertices
of valence one, since they are all in axes. All that is left to
prove is that ${\mathcal T}'$ is connected and hence a subtree.

Let $c_0, c_1$ in ${\mathcal T}'$ so that there are $f, g$ in $\pi_1(P)$
with $c_0$ in $A(f)$, $c_1$ in $A(g)$.
If $A(f), A(g)$ intersect, then there is a path in ${\mathcal T}'$
from $c_0$ to $c_1$.
Suppose then that they do not intersect.
There is a well defined bridge in
${\mathcal T}$ from $A(f)$ to $A(g)$ denoted by
$[x,y]$ $-$ it
is a closed segment intersecting $A(f)$ only in the extremity $x$
and intersecting $A(g)$ only in $y$. Let $z = f^{-1}(x)$.
Consider the element $gf$ which is in $\pi_1(P)$.
Then $x$ separates $z$ from $y$ and so separates $z$ from
$g f(x)$ which is in $g A(f)$.
Also $g f(z) = g(x)$ separates $x$ from $g f(x)$ which is in
$g A(f)$. It follows that $z, x, gf(z)$ and $g f(x)$ are all
distinct and linearly ordered in a segment contained in ${\mathcal T}$.
Hence $gf$ acts freely on ${\mathcal T}$ and $x, gf(x)$ are in
$A(gf)$. In particular $x$ and $y$ are in $A(gf)$ contained
in ${\mathcal T}'$ so there is
a path in ${\mathcal T}'$ from $c_0$ to $c_1$.
This shows that ${\mathcal T}'$ is connected.

\vskip .1in
\noindent
{\underline {Weakly embedded Birkhoff annuli}}

Suppose there is a vertex $q$ of ${\mathcal T}'$ and an element $g$ of 
$\pi_1(M)$ (not necessarily in $\pi_1(P)$) and a lozenge 
$C$ in ${\mathcal T}'$ with $g(q)$ intersecting the interior of $C$.
The lozenge $C$ is part of an axis $A(f)$ for some $f$ in 
$\pi_1(P)$. 
Here $g(q)$ intersects $C$. By the remark
after 
proposition \ref{toriem}, it follows that 
the subtree ${\mathcal T}'$ 
is a string of lozenges. 
Then $f, h_0^{2n}$ generate a ${\bf Z} \oplus {\bf Z}$ subgroup of $\pi_1(M)$
preserving this string of lozenges. Moreover, 
$q$ is a vertex of ${\mathcal T}'$ and $g(q)$ is in the interior of
$C$. Proposition \ref{toriem} again implies that the piece
$P$ has to be a free piece $-$ contrary to assumption in this case.

We conclude that each lozenge in ${\mathcal T}'$ corresponds to a weakly embedded
elementary Birkhoff annulus in $M$. 
We want to show that the union of the Birkhoff annuli can be adjusted
to be embedded in the interiors. 

\vskip .1in
\noindent
{\underline {Weakly embedded union of Birkhoff annuli}}

As in the proof of Proposition~\ref{embedd}, we consider the stabilizer $H$ in $\pi_1(M)$ of
${\mathcal T}'$. The action of $H$ on ${\mathcal T}'$
is not faithful, since the kernel contains 
a non trivial group.
Therefore, $H$ contains an infinite cyclic normal subgroup, but also contains
$\pi_1(P)$. It follows that $H=\pi_1(P)$, since $P$ is a maximal Seifert piece.
In addition the same arguments show that the stabilizer in $\pi_1(M)$ of $\mathcal T$
is also $\pi_1(P)$.

Suppose that $g$ in $\pi_1(M)$ maps a 
vertex $\alpha$ of ${\mathcal T}'$ to a vertex of ${\mathcal T}'$. 
Hence it also sends a vertex of $\mathcal T$ to a vertex of $\mathcal T$.
In that case we already observed during the proof of proposition~\ref{embedd} that 
$g$ stabilizes ${\mathcal T}$ and hence belongs to
$\pi_1(P)$.

Consider the quotient of the tree ${\mathcal T}'$ by $\pi_1(P)$. It is a graph, that we denote by $A$.
Since it is a graph, the fundamental group of $A$ is a free 
group, and since 
$\pi_1(P)$ is finitely generated, then the fundamental group
of $A$ has finite rank. Moreover, 
by construction, $A$ does not contain an infinite ray (since every element of ${\mathcal T}'$ lies
on the axis of some element of $\pi_1(P)$ acting freely on ${\mathcal T}'$). 
It follows that $A$ is a finite graph.

Consider a
fundamental domain of the action of 
$\pi_1(P)$ on ${\mathcal T}'$. We lift every lozenge of this fundamental domain to a Birkhoff band in $\mi$,
and then lift all other lozenges in ${\mathcal T}'$ in a $\pi_1(P)$-equivariant way. It projects to
an union of weakly embedded Birkhoff annuli in $M$. Once more,
we can then use cut and paste techniques of
\cite{Ba2} to have the union of the Birkhoff annuli
to be embedded in the interior of the annuli $-$ with
possible identifications in the boundary orbits.

\vskip .1in
\noindent
{\underline {Flow adapted neighborhoods of periodic pieces}}

Let $B$ be the union of the weakly embedded elementary Birkhoff annuli as in the
previous item. 
Let U be the neighborhood of $B$ obtained by taking a tubular neighborhood of every
periodic orbit contained in $B$ (the ``tangent periodic orbits''), attaching to them tubular
neighborhoods of the elementary Birkhoff annuli.
Topologically, this corresponds to the following: start with a
finite collection of solid tori and attach several handles
diffeomorphic to $[-1,1] \times [-1,1] \times {\bf S}^1$, where in each handle, $\{ 0 \} \times [-1,1] \times {\bf S}^1$
is contained in the corresponding weakly embedded Birkhoff annulus. Handles attached to a given solid torus (corresponding
to one of the tangent periodic orbits) are pairwise disjoint. 
One can perform a Dehn surgery on $U$ along tangent periodic orbits so that now the handles are attached
along longitudes of the solid tori: we get a $3$-manifold $U'$ which is clearly a circle bundle over
a surface with boundary $\Sigma$. Moreover, $\Sigma$ retracts to the graph $A.$ 

It follows that $U$ is diffeomorphic
to a Seifert manifold, obtained by Dehn surgeries around fibers in $U'$ above vertices of $A$.
More precisely, there is a Seifert fibration $\xi: U \to \Sigma_*$ where $\Sigma_*$ is an orbifold, whose
singularities correspond to vertices of $A$; singular fibers are tangent periodic orbits where attached
Birkhoff annuli wrap non trivially. 

Now observe that $U$ is the projection of a ``tubular neighborhood'' in $\mi$ of Birkhoff bands, which
is homeomorphic to the product of the tree ${\mathcal T}'$ by ${\bf R}$. This neighborhood is therefore
simply connected, and $U$ is an incompressible Seifert submanifold with fundamental group isomorphic
to $P$. Therefore, $P$ is isotopic to $U$. This achieves the proof of Theorem F.

\vskip .1in
\noindent
{\bf {Remark}} $-$ The only periodic orbits contained in $U$ correspond
to the projections of the vertices of ${\mathcal T}'$.

Here is why:
The interiors of the finitely many Birkhoff annuli in question are 
transverse to $\Phi$ and so orbits intersecting these interiors 
exit $U$ if $U$ is sufficiently small.
The other orbits are in the solid tori neighborhoods.
If these neighborhoods are small enough then the only orbits entirely contained
in them are the core orbits.

In particular a singular orbit $\gamma$ cannot intersect the interior
of the Birkhoff annuli, hence either $\gamma$ is one of the periodic
orbits in $U$ or can be chosen disjoint from $U$ if $U$ is small.
Previously we had proved that a singular orbit is homotopic into
a piece of the torus decomposition. In a graph manifold, if a singular
orbit is homotopic into a free piece $Z$, we conjecture that 
it must be homotopic into the boundary of the piece $Z$.

\section{New classes of examples of pseudo-Anosov flows 
in graph manifolds}
\label{sec:examples}

In section \ref{seifconj} we
described  some new examples of (one prong) pseudo-Anosov
flows.
In this section we will describe two new
classes of examples, which are extremely interesting:
The first class consists of actual pseudo-Anosov flows. The examples
in the second
class, which is a much larger class,  may have one prongs.

\vskip .15in
\noindent
1) Consider the class of examples 1) of section \ref{seifconj}. Each example had
a 
$2$-fold branched cover which is the geodesic flow in $T_1 S$, where
$S$ is closed, hyperbolic and has a reflection along finitely
many geodesics.
For simplicity we assume here that $S$ has a single closed geodesic $\alpha$ of symmetry.
Let $N$ be the quotient manifold.
In $N$, there is a quotient annulus $C$ which is the branched quotient
of the unit tangent bundle of $\alpha$.
Now for any integer $n > 0$ we can do the $n$-fold branched cover of $N$ 
along $C$. If $n = 2$ this recovers the original geodesic flow.
Otherwise the boundary of $C$ lifts to 2 closed orbits which are
$n$-prongs. Let $M_n$ be this $n$-fold cover and $C'$ be the lift
of the annulus $C$. The set $C'$ cuts $M_n$ into Seifert fibered
pieces $-$ each a copy of $T_1 S'$, where $S'$ is one component of $S$
cut along $\alpha$ (notice both components of $S - \alpha$ are isometric by the symmetry
along $\alpha$).
Each of these components is a component (up to isotopy) of the torus
decomposition of $M_n$.
In each of these components the fiber acts freely on the orbit space,
so these are free pieces.
There is one additional Seifert component which is a small neighborhood of
$C'$. There is a planar graph $X$ which has 2 vertices (corresponding
to the 2 directions on the geodesic $\alpha$) and $n$ edges from one
vertex to the other. The set $C'$ is homeomorphic to $X \times {\bf S}^1$.
This is a Seifert fibered piece of $M_n$, where the fiber corresponds
to a periodic orbit $-$ this is a periodic piece.

This highlights an important fact: there are examples of 
graph manifolds $M$  supporting 
pseudo-Anosov flow $\Phi$,
so that in the torus 
decomposition of $M$ there are periodic pieces glued to free pieces.

\vskip .2in
\noindent
2) The next class of examples will be on graph manifolds where all pieces
are periodic. It is much more involved and much more interesting.

In the previous section we proved that the periodic Seifert pieces can
be obtained as neighborhoods of unions of Birkhoff annuli.
Here we will introduce standard models for certain neighborhoods
of Birkhoff annuli and then use them to produce many examples.

\vskip .1in
\noindent
{\underline {Model of neighborhood of an embedded Birkhoff annulus}}

Let $I = [-\pi/2, \pi/2]$. Let $N = I \times {\bf S}^1 \times I$ with coordinates 
$(x,y,z)$. Think of ${\bf S}^1$ as $[0, 1]/0 \sim 1$. Convention: the increasing
or positive direction in ${\bf S}^1$ corresponds to increasing in $[0,1]$.

For every positive real number $\lambda$, 
we consider the $C^{\infty}$ vector field $X_\lambda$ defined by:

\begin{eqnarray*}
\dot{x} & = & 0 \\
\dot{y} & =& \lambda\sin(x)\cos^2(z) \\
\dot{z} & = & \cos^2(x) + \sin^2(z)\sin^2(x)
\end{eqnarray*}

Let $\psi_\lambda$ be the local flow in $N$ generated by $X_\lambda$. It has 
the following properties:

\begin{itemize}

\item it preserves the fibration by circles $(x, y, z) \mapsto (x, z)$.

\item
There are only 2 closed orbits: 

$$\alpha_1 \ = \ \{ -\pi/2 \} \times {\bf S}^1 \times \{ 0 \}, \ \ \ \ 
\alpha_2 \ = \ \{ \pi/2 \} \times {\bf S}^1 \times \{ 0 \}.$$

\noindent
In $\alpha_1$ the flow is decreasing the $y$ coordinate
(in the flow forward direction) and 
in $\alpha_2$ the flow is increasing the $y$ coordinate.
Hence as oriented orbits, $\alpha_1$ is freely homotopic in $N$ to $(\alpha_2)^{-1}$.

\item
The flow is incoming and perpendicular to the boundary $I \times {\bf S}^1 \times \{ -\pi/2 \}$
and outgoing and perpendicular to $I \times {\bf S}^1 \times \{ \pi/2 \}$.
The flow is tangent to $\partial I \times {\bf S}^1 \times I$.

\item 
The annuli $x =$ constant are flow saturated. 

\item
The orbits in $\{ -\pi/2 \} \times {\bf S}^1 \times \{ -\pi/2 \}$ enter $N$ and 
spiral towards 
$\alpha_1$ in the negative $y$ direction. Hence in $N$, $W^s(\alpha_1) = \{ -\pi/2 \} \times
{\bf S}^1 \times [-\pi/2, 0]$. 
In $\{ -\pi/2 \} \times {\bf S}^1 \times (0, \pi/2]$ the orbits spiral 
(flow backwards) to
$\alpha_1$ in the positive $y$ direction, so
$W^u(\alpha_1) = \{ -\pi/2 \} \times {\bf S}^1 \times [0, \pi/2]$.
We have a similar behavior (with the $y$ coordinate increasing when moving 
flow forward)
in $\{ \pi/2 \} \times {\bf S}^1 \times I$.

\item
The flow is invariant under the any rotation in the $y$ coordinate:
$(x,y,z) \rightarrow (x, y+a, z)$ where the $y$ coordinate is $mod \ 1$.
The flow is invariant under $(x,y,z) \rightarrow 
(-x, \ - y \ (mod \ 1), \ z)$.
This is symmetry (I).

\item
Let $F_0 = (-\pi/2, \pi/2) \times {\bf S}^1 \times \{ -\pi/2 \},
F_1 = (-\pi/2, \pi/2) \times {\bf S}^1 \times \{ \pi/2 \}$,
both parametrized by the $x, y$ coordinates.
In $(-\pi/2, \pi/2) \times {\bf S}^1 \times I$ all orbits enter $N$ in
$F_0$ and exit $N$ in $F_1$. An easy computation shows
that the variation of time spent between the entrance and the exit is:

$$\Delta t = \frac\pi{\mid \cos(x) \mid}$$

There is an induced homeomorphism $f: F_0 \rightarrow F_1$ given
by the exit point in the $x, y$ coordinates.
It has the form

$$f(x,y) \ \ = \ \ (x, y + a(x)),$$

\noindent
where the function $a(x)$ is $C^{\infty}$ and depends only on $x$. It can
also be computed:

$$a(x) = \lambda\pi[ \tan(x) - \tan(x/2)]$$

Observe that $a(0) = 0$. In fact, the
orbits in the center annulus have $y$ coordinate constant.
Also, $a(x)$ converges to minus infinity when
$x$ converges to $-\pi/2$ and $a(x)$ converges to infinity
when $x$ converges to $\pi/2$. In addition,
$a(-x) = - a(x)$.

Finally:
$$a'(x) = \lambda\pi[\frac12 + (\tan^2(x) - \frac12\tan^2(x/2)] \geq \frac{\lambda\pi}2$$

\end{itemize}

\vskip .2in
By the formula above, the map $f$ is a non linear shearing in the
$y$ direction. The bigger the $\lambda$ the stronger the shearing.

One canonical Birkhoff annulus associated to the block $N$ is $B = [-\pi/2, \pi/2] \times {\bf S}^1
\times {0}$. If $\alpha_1, \alpha_2$ are traversed in the positive flow
direction
then $B$ is a free homotopy from $\alpha_1$ to $(\alpha_2)^{-1}$.
The flow is transverse to $B$ outside of $\alpha_1, \alpha_2$.
The formulas above are convenient and give explicit models, but
they are not essential:
Up to topological equivalence any embedded Birkhoff annulus has
a neighborhood with this description.

\begin{figure}
\centeredepsfbox{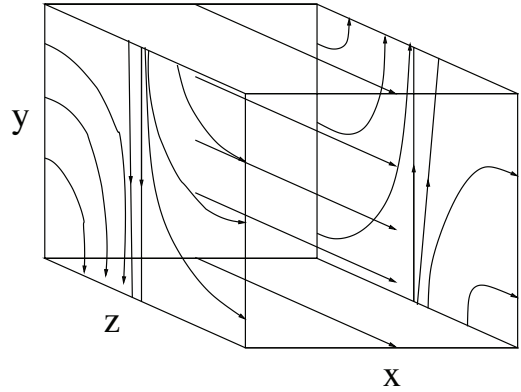}
\caption{Figure 4: The local flow $(N, \psi_\lambda)$. The top and bottom
are identified, that is, the vertical coordinate $y$ is defined modulo $1$.}
\label{box}
\end{figure}


\vskip .1in
\noindent
{\underline {Glueing the tangential boundaries of the blocks}} 

Observe that the same formula defines a vector field $\tilde{X}_\lambda$ 
on $\tilde{N} := {\bf R} \times {\bf S}^1 \times I$, which is $2\pi$-periodic
on the coordinate $x$. Actually, due to the invariance of $X_\lambda$ under the symetry (I),
we see that the transformation $\tau(x,y,z) = (x+\pi, -y, z)$ preserves $\tilde{X}_\lambda$.
The quotient of $\tilde{N}$ by the cyclic group generated by $\tau$ is a Seifert manifold $N_1$,
homeomorphic to the product $K \times I$, where $K$ is the Klein bottle. The induced
local flow has a single 1-prong singular orbit; $\tilde{N}$ has two boundary components,
one which is a incoming Klein bottle, the other an outgoing Klein bottle. 

More generally, we can take the quotient by the group generated by $\tau^k$ where $k$
is a positive integer. We get a Seifert 3-manifold $P_k$, diffeomorphic to $K \times I$ or ${\bf T}^2 \times I$
(according to the parity of $k$), with one incoming boundary component and one outgoing component,
containing exactly $k$ singular 1-prong periodic orbits.

Now, more generally, we can glue several copies of $(N, X_\lambda)$ in a much more involved way.
The blueprint encoding such a glueing will be a finite fat graph $X,$ ie. a graph embedded in a surface $\Sigma$
with boundary, such that $X$ is a retract of $\Sigma$ 
$-$ here, we do not require that $\Sigma$ be oriented. 





We moreover require the following conditions:

\vskip .1in
\begin{em}
Condition (I): the valence of every vertex is an even number.
\end{em}

\vskip .1in
\begin{em}
Condition (II): the set of boundary components of $\Sigma$ can be partitioned
in two subsets so that for every edge $e$ of $X$, the two sides of $e$ in $\Sigma$ lie
in different subset of this partition.
\end{em}


\vskip .1in
Use as labels ``incoming'' and ``outgoing''  for this partition of the set of boundary components of $\Sigma.$ 
Now every edge has an incoming side, and an outgoing side.

Given such information we construct a flow in a $3$-manifold.
Associate to every edge $e$ of $X$ a copy $N_e$ of $N$ as above.
Then, every {\em incoming} boundary component 
$c$ of $\Sigma$ corresponds to a cyclic sequence of edges $(e_1, e_2, ... , e_k)$. We glue all the associated
$N_{e_i}$ along the stable manifolds $\{ \pm\pi/2 \} \times {\bf S}^1 \times [-\pi/2, 0)$ in the same cyclic order;
more precisely, we map every point of coordinate $(\pi/2, y, z) \;\; (z <0)$ in $N_{e_i}$ to the point
of coordinate $(-\pi/2, -y, z)$ in the following copy $N(e_{i+1})$.
The result, for each boundary component $c$, is a Seifert $3$-manifold (with boundary and corner).
The Seifert $3$-manifold has
interior diffeomorphic to $P_k$ with the
unstable manifolds $\{ \pm\pi/2 \} \times {\bf S}^1 \times [-\pi/2, 0]$ 
and the incoming and outgoing boundaries removed. It has an incoming boundary component,
obtained by glueing
copies of closures of the incoming annulus $F_0$ for each $N_{e_i}$.
This boundary component is  diffeomorphic to the torus if $k$ is even
and to the Klein bottle if $k$ is odd. 
This manifold also has ``outgoing" annular components. Observe that up to diffeomorphism,
the result depends only on the 
cyclic order $(e_1, e_2, ... , e_k)$.

Next we do the similar glueing along outgoing boundary components, but now glueing the copies of $N$ 
along the unstable annuli. The result is a Seifert manifold $N(X)$, 
with incoming and outgoing components, but no tangential boundary components. Moreover,
to every vertex $v$ of $X$ corresponds  a tubular neighborhood of the 
periodic orbit which is homeomorphic to a solid torus.
The flow is obviously homeomorphic to a $p$-branched cover
of a tubular neighborhood of the singular orbit in $P_1$
$-$ here $2p$ is the valence of
$v$. 
This is a compact Seifert manifold.
Observe that $N(X)$ is orientable if and only if all $k$ are even.

By construction, $N(X)$ is equipped with a vector field $X_\lambda$ for every $\lambda > 0$. 
The boundary of $N(X)$ is an union of incoming components and outgoing components, which are tori or Klein bottles.
Due to the final process in the construction, this vector field is not smooth along the vertical
orbits corresponding to the vertices of $X,$ except if the valence of the vertex is 2 or 4, a special situation
where we can perform the glueing so that the vector field is smooth in the neighborhood of
the associated singular orbit. In particular, if all vertices have valence $4$, then there is no singular orbit.

This is exactly the case in the Bonatti-Langevin \cite{Bo-La} example, where the fat tree $X$
is a figure eight (with one vertex) embedded in a once-punctured M\"obius strip. 

\vskip .1in
\noindent
{\bf {Remark: }} 
Notice that $N(X)$ is a circle bundle over the surface $\Sigma$, with
fibers the vertical circles with constant $x$, $z$ components. Moreover, the local flow generated 
by $X_\lambda$ preserves this fibration, hence there is an induced vector field $\bar{X}_\lambda$ on
$\Sigma$. The vector field $\bar{X}_\lambda$ is Morse-Smale. Its singularities are the vertices of
$X$; it is transverse to $\partial\Sigma$. There is three types of non-singular trajectories of $\bar{X}_\lambda$:


-- trajectories in the stable line of a singularity, entering $\Sigma$,

-- trajectories in the unstable line of a singularity, exiting $\Sigma$,

-- trajectories joining two boundary components.

Observe that the data $(\Sigma, X)$ is equivalent to the data $(\Sigma, \bar{X}_\lambda)$ up to isotopy.

\vskip .1in
\noindent
{\underline {Glueing the transverse boundary components}} 

The next step is to glue outgoing boundary components to incoming boundary components.
Observe that these components are naturally isomorphic to boundary components in the manifolds
$N_k$, and thus admit natural coordinates $(x, y)$.

Let $T'$ be the union of the incoming boundary components and
let $T$ be the union of the outgoing boundary components. 
Let $\mu$ denote the line field in $T$ or $T'$ associated to $x$ being constant.
In order to perform the glueing, we have one obvious condition: there must be the same number
of outgoing and incoming tori, and the same number of outgoing and incoming Klein bottles! 

Under this condition,
we can select a map $A: T \rightarrow T'$ which is linear in the $x,y$-coordinates on each component.
The only assumption we will have is that $A$ does not preserve any 
of the line fields $\mu$.
Equivalently $A$ does not send any unstable manifold of the periodic
orbits to a curve isotopic into the stable manifold of a periodic
orbit. 

Given this condition
we first show that there are no components of $T$ which are Klein bottles. 
Suppose there is one such component denoted by $K_1$ to be glued to a component
$K_2$ of $T'$.
Notice that up 
to isotopy there are only two foliations by circles of the Klein bottle $K$.
One foliation has two circles which are orientation reversing and the nearby
leaves cover such a leaf two to one. The leaf space is a $1$-dim orbifold,
with two ``boundary" orbifold points of order $2$. This is type I.
The other
foliation comes from a product foliation by circles of the annulus
and glueing the boundaries by an orientation reversing homeomorphism.
This is type II.
Since there are only 
two such foliations up to isotopy and they are intrinsically different
(one has orientation reversing leaves and the other does not), then: any
homeomorphism between a Klein bottle $K$ and another $K'$ has to preverse each
type up to isotopy.

The construction of the flow shows that the line field $\mu$ 
induces foliations of type II in $K_1$ and $K_2$.
By the above explanation $A$ has to preserve the line field $\mu$ up
to isotopy, which we do not want. 
Hence we have a necessary condition:

\vskip .1in
\noindent
{\bf {Conclusion}} $-$ In order for the last step to produce a pseudo-Anosov flow,
then all the components of $T$ have to be tori.
\vskip .1in

In particular the manifolds in the middle step will 
all be orientable.

By glueing $T'$ onto $T$ by $A$ we obtain a closed $3$-manifold $M=M(X, A)$ 
equipped with a family of vector fields $Y_\lambda$.
Hence it provides a flow $\Psi_\lambda$ on
$M$ for each $\lambda>0$. The periodic orbits of $X_\lambda$ provide a finite number of 
periodic orbits of $\Psi_\lambda$ that we call {\em vertical orbits}.
Observe that since $X_\lambda$ is orthogonal
to the boundary, $Y_\lambda$ is smooth outside of the vertical orbits.

Our goal is to prove that, if $\lambda$ is big enough, 
then $\Psi_\lambda$ is pseudo-Anosov.

Let $l_0^u$ be the union of the circles in $T$ contained in the local unstable manifolds of the vertical orbits
(they are associated to the circles $x=\pm\pi/2$, $z=\pi/2$ in
each block), and similarly let $l_0^s$ be the union of the circles in $T'$ contained in
the local stable manifolds of the vertical orbits.
Let $\varphi$ be the first return map from a maximal subset of $T$ to itself.
Its domain is the complement in $T$
of $\mu^s_0=A^{-1}(l_0^s)$. 
For every $n>0$, let $\mu^s_n$ be the preimage of $\mu^s_0$ by $\varphi^n$: 
$T \setminus \mu^s_n$ is the domain of $\varphi^{n+1}$. Each component of $\mu^s_n$ is a 
curve in $T$, intersecting every circle 
in $l_0^u$, and spiraling around two circles in $\mu_0^s$. 
The complement $\Omega^+$ of the union $\mu^s_\infty$ of all $\mu^s_n$ is the 
domain of points where all the positive iterates 
$\varphi^n \;\;(n\geq 0)$ are defined. 
Observe that $\Omega^+$ may not be not open: it is the complement of 
$\mu^s_\infty$, which is an union of countably many
$1$-manifolds: the intersection with $T$ of the stable manifolds of the vertical orbits.

Let $C_0$ be a smooth small cone field on $T$, 
centered around $\mu$, and constant in the coordinates $x$, $y$.
 If $C_0$ is small enough,
then $A(C_0)$ is a cone field in $T'$ whose closure avoids the line field $\mu$
in $T'$. 
If in addition $\lambda > \lambda_0 >>1$, that is,
$a'(x) \geq \lambda\pi/2 > a_0 >> 1$, then the image of $A(C_0)$ across
the fundamental blocks will be very close to the constant $x$ direction
$-$ that is $\mu$. This is because $A$ is a linear map, so $A(C_0)$ is a definite
positive distance away
from the line field $\mu$.
In addition if the shearing is strong enough as above then the first return of
$A(C_0)$ will be very close to the line field $\mu$.
This implies that whenever $\varphi$ is defined, 
 then $\varphi_*(C_0)$ is strictly contained
in $C_0$. Moreover, this contraction 
from $C_0$ inside itself is uniform, since the bound from below of
$a'(x)$ is uniform. 
Furthermore: $\varphi_*(C_0) \subset C_0$ is close to $\mu$, hence
every tangent vector in $C_0$ 
has a non-trivial $y$-component, which is uniformly expanded 
by the differential of $\varphi$. It follows that, again increasing $\lambda_0$ if necessary,
all vectors in $\varphi_*(C_0)$ have a norm uniformly expanded under the differential
of $\varphi$, let us say have norm at least multiplied by $2$.

Given these properties, 
standart arguments (see for example \cite{Ha-Th}) 
show that the intersection of all iterates $\varphi^n_*(C_0)$
defines an invariant direction at every point of $\Omega^+$.
Vectors in this direction are uniformly exponentially expanded under
the action of $\varphi_*$. 

Consider now more closely the set $\mu^s_\infty$. 
Let $F$ be a component of the complement in $T$ of $l^u_0$: it is a copy of the
annulus $F_1$
(from the definition of model neighborhoods of Birkhoff annuli).
The intersection between $F$ and $\mu_0^s$ (after the glueing by $A$)
is an union of straight segments, with no tangent vectors
in $C_0$, and joining the two boundary components of $F$. 
The second generation curves, that is, the 
components of $\mu_1^s=\varphi^{-1}(\mu_0^s))$ are obtained by 
pushing backward the first generation lines through all blocks. These become
curves in $T'$ with direction very close to 
$\mu$ if the curves are close to $l^s_0$. Then apply $A^{-1}$: in every annular component $F$
they are still a union of curves joining the boundary of $F$, 
and these curves are nearly horizontal, that is, with tangent directions outside
$C_0$. 
Iterating the argument, we get that every 
connected component of $\mu^s_\infty$ has these properties:
in every annular component $F$, it is a disjoint union of 
graphs $y=g(x)$ of smooth functions, with uniformly bounded derivative
$g'$.
They are of course all included in the stable manifold of vertical orbits.

\vskip .1in
\noindent
{\underline {Claim}} $-$ $\Omega^+$ has empty interior

This is the key property.
Suppose this is not true and let
$q$ be a point in the interior of $\Omega^+$. Its positive orbit intersects
$T$ infinitely many times; hence there is an annular component $F$ 
of $T - l^u_0$ visited infinitely many times.


Consider now all paths $c$ in $\mbox{Int}(\Omega^+)$,
with tangent directions contained in $C_0$. Due to the description above, the length of these paths
is uniformly bounded from above. 

On the other hand, let $c$ be such a path containing $q$. 
There are infinitely many iterates
$\varphi^{n_k}(q)$ contained in $F$. Since $c$ is connected, and since the image of $\varphi$
avoids $l^u_0$, the paths $\varphi^{n_k}(c)$ are all contained in $F \cap \mbox{Int}(\Omega^+)$.
But they all have tangent vectors contained in $C_0$ as $\varphi_*(C_0) \subset C_0$,
and their length is 
exponentially increasing as proved above: 
contradiction. The claim is proved.

\vskip .1in

It follows that $\mu_\infty^s$ is dense. Hence, every annular component $F$ is foliated by graphs of
continuous functions $y=g(x)$, which are even $C^1$ since they are limits of smooth
functions with uniformly bounded derivatives. Pushing along the flow, 
we obtain a foliation $\Lambda^s$ in
$M$ of codimension one, which is $C^1$ outside the vertical orbits. 
Observe that this foliation is $C^1$ on $T$,
where it defines a one-dimensional foliation. 
This foliation admits closed leaves (the circles $\mu^s_0$) and all other
leaves in $T$ spiral towards these closed leaves. There is no Reeb component.

Reversing the flow direction, we construct a codimension one foliation 
$\Lambda^u$. These two foliations are transverse
to $T$ and $T'$. Moreover, there are transverse one to the other: 
indeed, in $T$, near 
$l^u_0$ \  the foliation
$\Lambda^s$ is very close to $A^{-1}(\mu)$, whereas
$\Lambda^u$ is very close to $\mu$. Iterating by powers of $\varphi$
this works in all of $T$.
Moreover, the stable (respectively unstable) manifolds 
of the vertical orbits are leaves of $\Lambda^s$ (respectively $\Lambda^u$),
and their union is dense in $M$. 
The foliations $\Lambda^s$ and $\Lambda^u$ are the natural candidates
for being the stable and unstable foliations of $\Psi_\lambda$.

Let $q$ be a point in $T$. If $q$ is in $\mu^s_\infty$, 
ie. the stable manifold of a vertical orbit, 
then the leaf of $\Lambda^s$ containing $q$ is obviously in
the stable manifold of $q$: for $t$ big enough, the vectors tangent 
to $\Lambda^s(q)$ at $q$ are divided at least by $2$ by the differential of $\Psi^t_{\lambda}$.

Now assume that $q$ lies in $\Omega^+$, ie. that all iterates $\varphi^n(q)$ are defined. 
At each of these points, there is a tangent cone $C_0(\varphi^n(q)))$, which is exponentially expanded.
But there is also a cone field $C'_0(\varphi^n(q))$, constructed by considering the reversed flow, and
which is exponentially expanded by $\varphi_*^{-1}$, 
therefore exponentially contracted by $\varphi_*$.
Since $\Lambda^{s}$ is $\varphi$-invariant, and  also
$\Lambda^s$ has no tangent vector in $C_0$, 
then tangent vectors at $\Lambda^s(\varphi^n(q))$ must lie 
in $C'_0(\varphi^n(q))$, hence are exponentially contracted.
It follows that $\Lambda^s$ is the stable foliation for $\Psi_\lambda$, 
and similarly, $\Lambda^u$
is the unstable foliation.

\vskip .1in
\noindent
{\bf {Conclusion}} $-$
There are stable and unstable foliations of $\Psi_\lambda$, which is 
a (possibly one-prong) pseudo-Anosov flow.

Observe that the flow is
a $1$-prong pseudo-Anosov flow if and only if $X$ admits vertices of degree $2$.
If there are only $2$-prong orbits before the last glueing, ie. if all vertices of $X$ have valence $4$,
then $\Psi_\lambda$ is an Anosov flow. If there are no $1$-prong
orbits, then $\Psi_\lambda$ is a pseudo-Anosov flow.

This proves Theorem I.

\vskip .1in
\noindent
{\bf {Remark}} $-$ Notice that this produces infinitely many examples of pseudo-Anosov flows
in non orientable graph manifolds. These are obtained by appropriate arrangements
of orientation reversing glueing maps from tori $T$ to $T'$.
\vskip .1in

An interesting subclass of the class of flows constructed here
is the class where the graph $X$ is a circle: all the vertices
have degree two, that is
all the vertical orbits are $1$-prong. 
Observe that condition (II) implies that the surface $\Sigma$ must be an annulus.
The intermediate glueing $N(X)$ is then one of the manifolds $N_k$. The resulting manifold $M(X,A)$ is then
a torus bundle over the circle ($k$ must be even by the discussion above).

Since the only requirement on $A$ is that it does not preserve 
the vertical direction, we obtain in particular:

\begin{corollary}{}{}
In any torus bundle over ${\bf S}^1$ which is not $T^3$ there are
$1$-prong pseudo-Anosov flows with any even number of $1$-prong
orbits.
\end{corollary}

In particular notice that there are infinitely many one prong
such examples  in nil manifolds. The fundamental groups of 
these manifolds 
have polynomial  growth as opposed to exponential growth, which is obtained
by taking a hyperbolic linear map $A$.

\vskip .15in
\noindent
{\bf {Remark}} $-$
In the construction of periodic Seifert fibered pieces in
this section the following happens: For every vertical orbit  $\delta$
in
the piece and for every quadrant $W$ associated to $\delta$,
then $W$ contains a lozenge $Z$ with a corner in $\delta$.
This is not true for every periodic Seifert fibered  piece with
respect to a pseudo-Anosov flow. It follows that the construction in
this section does not attain all possibles periodic Seifert fibered pieces.
In particular in the construction in this section the neighborhoods of
the periodic pieces  always have boundaries which are transverse
to the flow. This does not occur in general periodic pieces.

\vskip .1in
\noindent
{\bf {Dehn surgery}} $-$ Once the examples in family 2 are constructed,
then one can perform any Dehn surgery on the vertical orbits. As long
as the new meridian is not the original longitude, the resulting
flows will be a (possibly one prong) pseudo-Anosov flow.
In addition each middle step manifold is still Seifert fibered,
so the resulting manifolds are still graph manifolds.
This tremendously expands the class of examples in graph manifolds.

\section{Questions and comments} 

Some of the important questions not directly addressed in this paper are the
following:

\vskip .1in
\noindent
1) Free Seifert fibered pieces.

Let $P$ be one such piece. One fundamental question is the 
following: 
Is there is a representative for $P$ with boundary a union of 
Birkhoff tori so that the flow restricted to $P$ is up 
to finite covers topologically equivalent
 to the geodesic
flow on a hyperbolic surface with boundary?
A geodesic flow on a surface with boundary is the 
{\underline {restriction}}
of the geodesic flow to the unit tangent bundle of a compact
surface with boundary a union of closed geodesics.
Nothing is known in the case of general pseudo-Anosov flows.
In the case of smooth Anosov flows, this has been analysed
by the second author and proved to be true in almost all 
circumstances \cite{Ba3}. 

Along these lines one very important, but vaguely
phrased question is: suppose that $P$ is
a free Seifert piece. Is there no singular orbit in
the middle of $P$? The geodesic flow on the unit tangent
bundle of a surface with boundary has no singular orbits,
so any singular orbit would have to be in the ``boundary" of
this piece. Perhaps the formulation should be that any
singular orbit in the piece has to be homotopic into
the boundary of the piece.

Recall the structure and examples of periodic Seifert pieces: 
certainly they can have singular orbits which are in
some way not removable from the piece.
Notice also that periodic pieces and free pieces can occur
in the same flow: we described examples in the beginning of
section \ref{sec:examples}.


In any case the dynamics in free Seifert fibered pieces should be much
more complex than in periodic Seifert pieces. For example the
Handel-Thurston example flows \cite{Ha-Th} are obtained from geodesic
flows, by cutting along a Birkhoff torus $-$ the set of unit vectors
along a separating geodesic of the surface and glueing with a
shearing. Each piece is a free Seifert piece. Notice that there
are infinitely many closed orbits entirely contained in each piece:
they correspond to all closed geodesics contained in that piece of
the surface. 
In fact there are uncountable many orbits of the flow entirely
contained in this piece.

\vskip .1in
\noindent
2) Periodic Seifert pieces.

In section \ref{sec:examples} we produced many examples with
periodic Seifert pieces where the boundary of each piece
is transverse to the flow. What happens in general?
We obtained partial answers in section \ref{sec:birk}, but the
general picture is not known yet.

\vskip .1in
\noindent
3) The atoroidal case.

The first author has done extensive work 
\cite{Fe3,Fe7} in the {\underline{closed}}
atoroidal case (where in fact by Perelman's work
\cite{Pe1,Pe2,Pe3}, $M$ is hyperbolic).
The questions addressed in that analysis
were more of a geometric nature: 
Exactly when is the flow is quasigeodesic? Can the flow yield geometric
information about the asymptotic or large scale
geometric structure of the universal cover? There is not a general structure
theorem in such manifolds, even in particular manifolds.

As mentioned in the introduction the atoroidal, non closed case
is effectively unknown. Still there are examples in manifolds with
two atoroidal pieces $-$ the Franks-Williams examples are obtained as
follows: start with a suspension Anosov flow in a manifold
$N$ and do a derived from Anosov construction \cite{Fr-Wi}.
This transforms a periodic orbit (which is hyperbolic type) into
(say) a repelling orbit. Remove a neighborhood of this orbit
to produce a manifold $N_1$ with a semiflow which is incoming
along the boundary. Let $N_2$ be a copy of $N_1$ with
a reversed flow. Franks and Williams show examples
of glueings of $N_1$ to $N_2$ which yield Anosov flows in
the resulting manifold $M$.
$N_1$ and $N_2$ are both atoroidal and the torus decomposition
of $M$ is $N_1 \cup N_2$. 
Notice that these flows are not transitive.

Surely the Franks and Williams examples
can be generalized to a certain extent.
One fundamental remaining question is whether there are examples
of pseudo-Anosov or Anosov flows in toroidal manifolds, so that
the flow is transitive and there are non trivial atoroidal
pieces.
Are there also examples with mixed behavior? That is, examples
with atoroidal and Seifert pieces?
Finally: what is the general structure of pseudo-Anosov 
flows restricted to atoroidal pieces?
For example can one always show that the boundary tori are
isotopic to transverse tori? This is not the case for
Seifert pieces.

{\footnotesize
{
\setlength{\baselineskip}{0.01cm}

}
}

\end{document}